\documentclass[10pt,reqno]{amsart}

\usepackage{latexsym}
\usepackage{amsmath}
\usepackage{epsfig}
\usepackage{epic}
\usepackage{amssymb}
\usepackage{color}

\newtheorem{theorem}{Theorem}[section]
\newtheorem{lemma}[theorem]{Lemma}
\newtheorem{corollary}[theorem]{Corollary}
\newtheorem{proposition}[theorem]{Proposition}
\newtheorem{sublemma}{}[theorem]
\newtheorem{noname}[theorem]{}

\newcommand{\ba}{\backslash}
\newcommand{\baba}{\ba\!\!\ba}
\newcommand{\ud}{{\underline{\downarrow}\,}}
\newcommand{\da}{\downarrow}
\newcommand{\cl}{{\rm cl}}

\newcommand{\thc}{$3$-connected}
\newcommand{\ths}{$3$-separation}
\newcommand{\tws}{$2$-separation}

\newcommand{\cn}{contradiction}

\begin{document}

\title[A splitter theorem for $3$-connected $2$-polymatroids]{A splitter theorem for $3$-connected $2$-polymatroids}

\thanks{The  second and third 
  authors were supported
 by the New Zealand Marsden Fund.}

\author{James Oxley} 
\address{Department of
 Mathematics, Louisiana State University, Baton Rouge, Louisiana, USA}
\email{oxley@math.lsu.edu}

\author{Charles Semple}
\address{Department of Mathematics and Statistics,
University of Canterbury,
Christchurch,
New Zealand}
\email{c.semple@math.canterbury.ac.nz}

\author{Geoff Whittle}
\address{School of Mathematics, Statistics and Operations Research,
Victoria University,
Wellington,  
New Zealand}
\email{Geoff.Whittle@mcs.vuw.ac.nz}

\date{\today}

\begin{abstract} Seymour's Splitter Theorem is a basic inductive tool for dealing with $3$-connected matroids. This paper proves a generalization of that theorem for the class of $2$-polymatroids. Such structures include matroids, and they model  both sets of points and lines in a projective space and sets of edges in a graph.  A series compression in such a structure is an analogue of contracting an edge of a graph that is in a series pair. A $2$-polymatroid $N$ is an s-minor of a $2$-polymatroid $M$ if $N$ can be obtained from $M$ by a sequence of contractions, series compressions, and  dual-contractions, where the last are modified deletions. 
The main result proves that if $M$ and $N$ are $3$-connected $2$-polymatroids such that $N$ is an s-minor of $M$, then $M$ has a $3$-connected  s-minor  $M'$ that has an s-minor isomorphic to $N$ and has $|E(M)| - 1$ elements unless $M$ is a whirl or the cycle matroid of a wheel. In the exceptional case, such an $M'$ can be found with $|E(M)| - 2$ elements. 
 \end{abstract}

\maketitle
\vspace*{-30pt}
\section{Introduction}
\label{intro}

Let $M$ be a $3$-connected matroid other than a wheel or a whirl. Tutte~\cite{wtt} proved that   $M$ has an element  whose deletion or contraction is $3$-connected. Seymour~\cite{pds} extended this theorem by showing that, for a proper $3$-connected minor $N$ of $M$, the matroid $M$ has an element  whose deletion or contraction is $3$-connected and has an $N$-minor. 
These theorems have been   powerful inductive tools for working with $3$-connected matroids. In \cite{oswww}, with a view to attacking representability problems for 2-polymatroids, we generalized the Wheels-and-Whirls Theorem to $2$-polymatroids. In this paper, we prove a generalization of the Splitter Theorem for $2$-polymatroids.

A basic example  of a matroid is a set of points in a projective  space. If, instead, we take a finite set of points and lines in a projective space, we get an example of a $2$-polymatroid. Whereas each element of a matroid has rank zero or one,  an individual element in a $2$-polymatroid can also have rank two. Formally, for a positive integer $k$, 
a  {\em $k$-polymatroid}  $M$ is a pair 
$(E,r)$ consisting of a finite set $E$, called the {\it ground set},  and a function $r$, 
called the {\it rank function}, from the power set of $E$ into the integers satisfying the 
following  conditions:

\begin{itemize}
\item[(i)] $r(\emptyset) = 0$;
\item[(ii)] if $X \subseteq Y \subseteq E$, then $r(X) \le r(Y)$;
\item[(iii)] if $X$ and $Y$ are subsets of $E$, then $r(X) + r(Y) \ge r(X \cup Y) + r(X \cap Y)$; 
and
\item[(iv)] $r(\{e\})\leq k$ for all $e\in E$.
\end{itemize}

A matroid is just a 1-polymatroid. Equivalently, it is a $2$-polymatroid in which every element has rank at most one.  
Our focus in this paper will be on $2$-polymatroids.   
From a graph $G$,  in addition to its cycle matroid,   we can derive a second $2$-polymatroid on $E(G)$, which we denote by $M_2(G)$. The latter  is defined
by letting the rank of a set $A$ of edges be the number of vertices incident with edges in $A$.
Observe that non-loop edges of $G$ have rank two in $M_2(G)$. 

Matroid connectivity generalizes naturally to $2$-polymatroids. In particular, 
3-connectivity for matroids extends routinely to a notion of 3-connectivity for
$2$-polymatroids.  A simple $3$-connected
graph $G$ has a $3$-connected cycle matroid. On the other hand, $M_2(G)$ is $3$-connected whenever 
$G$ is a $2$-connected loopless graph.

Deletion and contraction for matroids
extend easily to $2$-polymatroids. This gives a notion of minor for $2$-polymatroids that extends
that of minor for matroids, and, via cycle matroids, that of minor for graphs. But what
happens when we consider the $2$-polymatroid $M_2(G)$? If $e$ is an edge of $G$, then deletion
in $M_2(G)$ corresponds to deletion in $G$, but it is not the same with contraction.
However, there is an operation on $M_2(G)$ that corresponds to contraction in $G$. 
Specifically, if $e$ is an element of the $2$-polymatroid $M$ and $r(\{e\}) > 0$, then 
the {\em compression} of $e$ from $M$, denoted $M\downarrow e$, is obtained by 
placing a rank-$1$ element $x$ freely on $e$, contracting $x$, and then deleting $e$ from the resulting $2$-polymatroid.
In particular, $M_2(G)\downarrow e=M_2(G/e)$
for a non-loop edge $e$ of the graph $G$. 

Representability of matroids extends easily to representability of polymatroids over fields. Indeed, much of the motivation for this paper is derived from our desire to develop tools for attacking representability problems for $2$-polymatroids.  
The class of $2$-polymatroids representable over a field $\mathbb F$ is closed under both 
deletion and contraction. When $\mathbb F$ is finite, this is not the case for compression in general although it is the case for a restricted type of compression.
In \cite{oswww}, we defined a certain type of 3-separator, which we called a `prickly' 3-separator. A series pair in a graph $G$ is a 2-element prickly 3-separator of $M_2(G)$. Larger prickly 3-separators do not arise from graphs, but do arise in more general settings. Compressing elements from prickly 3-separators   preserves representability.  We gave examples in \cite{oswww} to show that, if we wish to generalize Tutte's Wheels-and-Whirls Theorem to $2$-polymatroids, it is necessary to allow compression of elements from prickly 3-separators. The main result of \cite{oswww} proves such a generalization by showing that  a $3$-connected non-empty $2$-polymatroid that not a whirl or the cycle matroid of a 
wheel has an element $e$ such that either  $M\ba e$ or $M/e$ is $3$-connected, or
 $e$ belongs to a prickly $3$-separator, and $M\downarrow e$ is $3$-connected. 
 
 Geelen, Gerards, and Whittle~\cite{ggwrota} have announced that Rota's Conjecture~\cite{rot} is true, that is, for every finite field, there is a finite set of minor-minimal matroids that are not representable over that field. In \cite{oswww}, we showed that, for every field ${\mathbb F}$, the set of minor-minimal 2-polymatroids that are not representable over  ${\mathbb F}$ is infinite, so one generalization of Rota's Conjecture for 2-polymatroids fails. We  believe, however,  that  an alternative generalization of the conjecture does hold. Specifically,  we conjectured in \cite{oswww} that, when ${\mathbb F}$ is finite,  there are only finitely many  2-polymatroids that are minimal with the property of  being non-representable over  ${\mathbb F}$ where we allow, as reduction operations, not only   deletion and  contaction but also compression of elements from prickly 3-separators.

Our main result  appears at the end of this section. We now give the rest of the background needed to understand that result. 
The matroid terminology used here will follow  Oxley~\cite{oxbook}.   Lov\'{a}sz and  Plummer~\cite[Chapter 11]{lovaplum} have given an interesting discussion of $2$-polymatroids and some of their properties. We call $(E,r)$ a {\it polymatroid} if it is a $k$-polymatroid for some positive integer $k$. 
In a $2$-polymatroid $(E,r)$, an  element $x$ will be called a {\it line}, a {\it point}, or a {\it loop} when its rank is $2$, $1$, or $0$, respectively. For readers accustomed to using the terms `point' and `line'   for flats in a matroid of rank one and two, respectively, this may create some potential confusion. However, in this paper, we shall never use the terms `point' and `line' in this alternative way. Indeed, we will not even define a flat of a $2$-polymatroid.




Let $M$ be a polymatroid $(E,r)$. For a subset $X$ of $E$, the {\it deletion} $M\ba X$ and the {\it contraction} $M/X$ of $X$ from $M$ are the pairs $(E-X,r_1)$  and \linebreak$(E-X,r_2)$ where, for all subsets $Y$ of $E-X$, we have $r_1(Y) = r(Y)$ and $r_2(Y) = r(Y \cup X) - r(X)$. 
We shall also write $M|(E- X)$ for $M\ba X$. 
A {\it minor} of the polymatroid $M$ is any polymatroid   that can be obtained from $M$ by a sequence of operations each of which is a deletion  or a contraction. 
It is straightforward to check that every minor of a  $k$-polymatroid is also a $k$-polymatroid. The {\it closure} $\cl(X)$ of a set $X$ in $M$ is, as for matroids, the set $\{x \in E: r(X \cup x) = r(X)\}$. Two polymatroids $(E_1,r_1)$ and $(E_2,r_2)$ are {\it isomorphic} if there is a bijection $\phi$ from $E_1$ onto $E_2$ such that $r_1(X) = r_2(\phi(X))$ for all subsets $X$ of $E_1$.

One natural way to obtain a polymatroid is from a collection of flats of a matroid $M$. Indeed, every polymatroid arises in this way \cite{helg, lova, mcd}. More precisely, we have the following.

\begin{theorem}
\label{herepoly} 
Let $t$ be a function defined on the power set of a finite set $E$. Then $(E,t)$ is a polymatroid if and only if, for some matroid $M$, there is a function $\psi$ from $E$ into the set of flats of $M$ such that $t(X) = r_M(\cup_{x \in X} \psi(x))$ for all subsets $X$ of $E$.
\end{theorem}

The key idea in proving this theorem is that of freely adding a point to an element of a polymatroid. Let $(E,r)$ be a polymatroid, let $x$ be an element of $E$, and let $x'$ be an element that is not in $E$. We can extend the domain of $r$ to include all subsets of $E \cup x'$ by letting
\begin{equation*}
r(X \cup x') = 
\begin{cases}
r(X), & \text{if $r(X \cup x) = r(X)$};\\
r(X) + 1, & \text{if $r(X \cup x) > r(X)$}.
\end{cases}
\end{equation*} 
Then it is not difficult to check that $(E \cup x', r)$ is a polymatroid. We say that it has been obtained from $(E,r)$ by {\it freely adding} $x'$ to $x$. If we repeat this construction by freely adding a new element $y'$ to some element $y$ of $E$, we can show that the order in which these two operations is performed is irrelevant.

Using this idea, we can associate a matroid with every $2$-polymatroid $M$ as follows. 
Let $L$ be the set of lines of $M$. For each $\ell$ in $L$,   freely  add  two points $s_{\ell}$ and $t_{\ell}$ to $\ell$. Let $M^+$ be the $2$-polymatroid obtained after  performing all of 
these $2|L|$ operations. Let $M'$ be $M^+ \ba L$. We call $M'$ the {\it natural matroid derived from $M$}. 

Given a graph $G$ with edge set $E$, as noted earlier, one can define a $2$-polymatroid $M_2(G)$ on $E$ by, for each subset $X$ of $E$, letting $r(X)$   be $|V(X)|$ where $V(X)$ is the set of vertices of $G$ that have at least one endpoint in $X$.  A polymatroid $(E',r')$ is {\it Boolean} if is isomorphic to the $2$-polymatroid that is obtained in this way from some graph. 
One attractive feature of $M_2(G)$ is that, except for the possible presence of isolated vertices, it uniquely determines $G$. More precisely, if $G_1$ and $G_2$ are graphs neither of which has any isolated vertices and if $M_2(G_1) = M_2(G_2)$, then there is a labelling of the vertices of $G_2$ such that $G_1 = G_2$.  This contrasts with the situation for matroids where quite different graphs can have the same cycle matroids.

Let $M$ be a polymatroid $(E,r)$. The {\it connectivity function}, $\lambda_M$ or $\lambda$, of $M$ is defined, for all subsets $X$ of $E$, by $\lambda_M(X) = r(X) + r(E-X) - r(M)$. Observe that $\lambda_M(E-X) = \lambda_M(X)$. 
It is routine to check, using the submodularity of the rank function, that the connectivity function is submodular, that is, for all subsets $Y$ and $Z$ of $E$,
$$\lambda_M(Y) + \lambda_M(Z) \ge \lambda_M(Y \cup Z) + \lambda_M(Y \cap Z).$$

Let $M$ be a polymatroid. For a positive integer $n$, a subset $X$ of $E(M)$ is {\it $n$-separating} if $\lambda_M(X) \le n-1$ and is {\it exactly $n$-separating} if  $\lambda_M(X) = n-1$  We say that $M$ is  
{\it $2$-connected} if it has no proper non-empty $1$-separating subset. We will also say that $M$ is {\it disconnected} if it is not 2-connected. 
We call $M$ {\it $3$-connected} if $M$ is $2$-connected and $M$ has no {\it $2$-separation}, that is, $M$ has no partition $(X,Y)$ with $\max\{|X|, r(X)\} > 1$ and $\max\{|Y|, r(Y)\} > 1$ but $\lambda(X) \le 1$. When $M$ is a $3$-connected $2$-polymatroid $(E,r)$, a {\it $3$-separation} of $M$ is a partition 
$(X,Y)$ of $E$ such that 
$\lambda(X) = 2$ and both $r(X)$ and $r(Y)$ exceed $2$.

Duality plays a fundamental role in matroid theory and will  also be  important in our work with $2$-polymatroids. Whereas there is a standard notion of what constitutes the dual of a matroid, for $2$-polymatroids, there is more than one choice. Let $M$ be a $k$-polymatroid $(E,r)$. The {\it $k$-dual} of $M$ is the pair $(E,r^*_k)$ defined by 
$r^*_k(Y) = k|Y| + r(E-Y) - r(M)$. This notion of duality was used, for example, in Oxley and Whittle's treatment \cite{ow2p} of Tutte invariants for $2$-polymatroids. An {\it involution} on the class ${\mathcal M}_k$ of $k$-polymatroids is a function $\zeta$ from ${\mathcal M}_k$ into ${\mathcal M}_k$ such that $\zeta(\zeta(M)) = M$ for all $M$ in ${\mathcal M}_k$. Whittle~\cite{gpw} showed that  the $k$-dual is the only involution on ${\mathcal M}_k$ under which deletion and contraction are interchanged in the familiar way.  However,  a disadvantage of this duality operation is that, for a matroid $M$, we can view $M$ as a $k$-polymatroid for all $k\ge 1$. Hence $M$ has a $1$-dual, which is its usual matroid dual. But it also has a $2$-dual, a $3$-dual, and so on. 
In \cite{oswww}, we used a duality operation   on the class of all polymatroids that, when applied to a $k$-polymatroid, produces another $k$-polymatroid and that, when applied to a matroid produces its usual matroid dual. In this paper, we will use a variant on that operation that agrees with it when applied to $3$-connected $2$-polymatroids with at least two elements.

Both of these versions of duality are members of a family of potential duals for a polymatroid $(E,r)$ that were defined by   McDiarmid~\cite{mcd}  and were based on assigning a weight $w(e)$ to each element $e$  of $E$ where $w(e) \ge r(\{e\})$ for all $e$ in $E$. For a set $X$, we shall write $||X||$ for the sum $\sum_{e \in X} w(e)$. In \cite{oswww}, we took $w(e)$ to be  $\max\{r(\{e\}),1\}$. Here, instead, we will take $w(e) = r(\{e\})$ and 
 define 
the {\it dual} of a polymatroid $(E,r)$ to be the pair $(E,r^*)$ where, for all subsets $Y$ of $E$,
$$r^*(Y) = ||Y|| + r(E- Y) - r(E) = \sum_{e \in Y}  r(\{e\})  + r(E- Y) - r(E).$$
It is straightforward to check that, when  $(E,r)$ is a $k$-polymatroid, so too is $(E,r^*)$. When $M = (E,r)$, we shall write $M^*$ for $(E,r^*)$. When the  polymatroid $M$ is a matroid, its dual as just defined coincides with its usual matroid dual provided $M$ has no loops.  However, if $e$ is a loop of $M$, then $e$ is a loop of $M^*$. The definition of dual used in \cite{oswww} (where we took $||Y|| = \sum_{e \in Y}  \max\{1,r(\{e\})\}$) was chosen to ensure that, when $M$ is a matroid, its polymatroid dual coincides with its matroid dual. Here, however, we are giving up on that, albeit in a rather specialized case. Note, however, that the two definitions of dual coincide unless $M$ has a loop so, in particular, they coincide when $M$ is $3$-connected having at least two elements.  
Moreover, as noted in \cite{oswww}, these two versions of duality share   a number of important properties, the proofs of which are very similar. For example, $\lambda_M(X) = \lambda_{M^*}(X)$. Next we  discuss the reason for the use of the above definition of duality, which follows \cite{susan, jmw}.

Consider the following example, which will guide how we proceed. Begin with the matroid that is the direct sum of $PG(r-1,q)$ and $PG(k-2,q)$ viewing this as a restriction of $PG(r+k-2,q)$.  Let $N$ be the restriction of $PG(r-1,q)$ to the complement of a hyperplane $H$ of it, so $N \cong AG(r-1,q)$.  Take $k$ distinct points, $x_1,x_2,\ldots,x_k$, of $PG(r-1,q)$ that are  in $H$, and let $\{y_1,y_2,\ldots,y_k\}$ be a spanning circuit in $PG(k-2,q)$. For each $i$ in $\{1,2,\ldots,k\}$, let $\ell_i$ be the line  of $PG(r+k-2,q)$ that is spanned by $\{x_i,y_i\}$. Let $M$ be the $2$-polymatroid whose elements are the points of $N$ along with the set $L$ consisting of the  lines  $\ell_1,\ell_2,\ldots,\ell_k$. It is straightforward to check that $M$ and $N$ are $3$-connected. The only way to obtain an $N$-minor of $M$ is to delete all the elements of $L$ since contracting any member of $L$ has the effect of reducing the rank of $E(N)$. But, in each of the $2$-polymatroids $M\ba L'$, where $L'$ is a proper non-empty subset of $L - \ell_k$, the set $\ell_k$  is $2$-separating. Since our goal is a splitter theorem, where we can remove some bounded number of elements from $M$ maintaining both $3$-connectivity and an $N$-minor, we will need a strategy for dealing with this example. One significant feature of this example is the very constrained nature of the 2-separations in each $M\ba L'$ with one side of each such 2-separation consisting of a single line. This is reminiscent of what happens in Bixby's Lemma~\cite{bixby} for $3$-connected matroids where, for every element $e$ of such a matroid $N$, either $N\ba e$ is $3$-connected except  for   some possible series pairs, or $N/e$ is $3$-connected except  for  some possible parallel pairs. Indeed, in the matroid derived from $M\ba L'$, each 2-separating line yields a series pair in the derived matroid.

The strategy that we will adopt is intimately linked to our choice of definition for the dual of a polymatroid. It is well known that, under the familiar definition of duality for matroids, taking the dual of the dual returns us to the original matroid. We now consider the relationship between a  polymatroid $M$ and the polymatroid $(M^*)^*$. If $M$ is a $3$-connected $2$-polymatroid with at least two elements, then $(M^*)^* = M$. To see what happens in general, we follow \cite{jmw}. Let $M$ be the polymatroid $(E,r)$. An element $e$ of $M$ is {\it compact} if $r(\{e\}) = \lambda_M(\{e\})$ or, equivalently, if $r(E - \{e\}) = r(E)$. We call $M$ {\it compact} if every element is compact. Thus, for example, a  matroid is compact if it has no coloops. In the example in the last paragraph, although $M$   is compact, $M\ba \{\ell_1\}$ is not since, for each $i \ge 2$, we have $r(\{\ell_i\}) = 2$ whereas $\lambda_{M\ba \{\ell_1\}}(\{\ell_i\}) = 1$.

The {\it compactification} $M^{\flat}$ of the polymatroid $M$ is the pair $(E,r^{\flat})$ where 
$$r^{\flat}(X) = r(X) + \sum_{x \in X} [\lambda (\{x\}) - r(\{x\})]$$ for all subsets $X$ of $E$. It is shown in \cite{jmw} that $M^{\flat}$ is a compact polymatroid and it is clear that if $M$ is a $2$-polymatroid, then so is $M^{\flat}$. The next result \cite{jmw} encapsulates some key properties of this compactification operation and justifies the approach we take here.

\begin{lemma}
\label{compact0} 
Let $(E,r)$ be a polymatroid $M$. Then 
\begin{enumerate}
\item[(i)] $M^*$ is compact;
\item[(ii)] $(M^*)^* = M^{\flat}$; 
\item[(iii)] $\lambda_M = \lambda_{M^*} = \lambda_{M^{\flat}}$; and 
\item[(iv)] $M/X$ is compact for all non-empty subsets $X$ of $E$ and 
$$(M/X)^* = (M^*\ba X)^{\flat}.$$
\end{enumerate}
\end{lemma}

Returning to our guiding example above, although $M\ba \{\ell_1\}$ is neither compact nor $3$-connected, its compactification is both. Observe that this compactification can be obtained from the restriction of the matroid $PG(r-1,q)$ to $E(N) \cup \{x_2,x_3,\ldots,x_k\}$ by relabelling each $x_i$ by $\ell_i$ noting that these $\ell_i$ are now points rather than lines. Thus compactification here has an analogous effect to cosimplification in matroids. 
By incorporating  compactification as part of the deletion operation, which is justified by (iv) of the last lemma, we see that, after deleting a single element, we have both maintained $3$-connectivity and kept an $N$-minor. This is precisely what we want in a splitter theorem.



In $2$-polymatroids, the behaviour of contraction differs   significantly from that for matroids. In particular, consider the $2$-polymatroid $M_2(G)$ obtained from a graph $G$,  where $G$ has  vertex set $V$ and edge set $E$. Let  $e$ be an edge of $G$. Deleting $e$ from $G$ has an unsurprising effect; specifically, $M_2(G) \ba e = M_2(G \ba e)$. But, to find $M_2(G)/e$, we cannot simply look at $M_2(G/e)$. In particular, what do we do with elements whose rank is reduced to zero in the contraction? To deal with this situation, it is standard to extend the definition of a graph to allow the presence of {\it free loops}, that is, edges with no endpoints. This terminology is due to Zaslavsky \cite{zas}. For a graph $G$ with free loops, the associated $2$-polymatroid $M_2(G)$ is defined, as before, to have rank function $r(X) = |V(X)|$. The deletion of a free loop $f$ from a graph just removes $f$ from the graph. We define the contraction of $f$ to be 
the same as its deletion. For an edge $e$ that is not a free loop, 
to obtain a graph $H$ so that $M_2(G) /e = M_2(H)$, we let $H$ have edge set $E  - e$ and vertex set $V  - V(\{e\})$. 
An edge $x$ of $H$ is incident with the vertices in $V(\{x\}) - V(\{e\})$. 

The difference between $M_2(G)/e$ and $M_2(G/e)$ motivated us to introduce an operation for $2$-polymatroids in \cite{oswww} that   mimics the effect of the usual operation of contraction of an edge from the graph.

Let $(E,r)$ be a $2$-polymatroid $M$, and let $x$ be an element of $E$. We have described already what it means to add an element $x'$ freely to $x$. Our new operation $M\downarrow x$  is obtained from $M$ by freely adding $x'$ to $x$ in $M$, then contracting $x'$ from the resulting extension, and finally deleting $x$. Because each of the steps in this process results in a $2$-polymatroid, we have a well-defined operation on $2$-polymatroids. When $x$ has rank at most one  in $M$, one easily checks that $M\downarrow x = M/x$. When $x$ is a line in $M$, 
we  see that $M\downarrow x$ and $M/x$ are different as their ranks are $r(M) - 1$ and $r(M) - 2$, respectively. Combining the different parts of the definition, we see that 
$M\downarrow x$ is the $2$-polymatroid with ground set $E  - \{x\}$ and rank function given, for all subsets $X$ of $E - \{x\}$, by 
\begin{equation} 
\label{getdown}
r_{M\downarrow x}(X) = 
\begin{cases} 
r(X), & \text{if $r(x) = 0$, or $r(X \cup x) > r(X)$;   and}\\
r(X) - 1, & \text{otherwise.}
\end{cases}
\end{equation}
We shall say that $M\downarrow x$ has been obtained from $M$ by {\it compressing} $x$, and $M\downarrow x$ will be called the {\it compression} of $x$. We showed in \cite{oswww} that $M_2(G)\downarrow e = M_2(G/e).$ Songbao Mo~\cite{smo} established  a number of properties of a generalization of this  operation that he defines for connectivity functions and calls {\it elision}.

Instead of treating arbitrary minors, much of graph theory restricts attention to topological minors in which the only allowed contractions involve  edges that meet vertices of degree two.  When $e$ and $f$ are the only edges in a $2$-connected graph $G$ meeting a vertex $v$, and $G$ has at least four vertices,   $\{e,f\}$ is a $3$-separating set in $M_2(G)$. This $3$-separating set is an example of a special type of $3$-separating set that we  introduced in \cite{oswww}. In a  $2$-polymatroid $M$, a 3-separating set $Z$ is {\it prickly} if it obeys the following conditions:
\begin{itemize}
\item[(i)] Each element of $Z$ is a line;
\item[(ii)] $|Z| \ge 2$ and $\lambda(Z) = 2$; 
\item[(iii)] $r((E -Z) \cup Z') = r(E - Z) + |Z'|$ for all proper subsets $Z'$ of $Z$; and 
\item[(iv)] if $Z'$ is a non-empty subset of $Z$, then
\begin{equation*}
r(Z') = 
\begin{cases}
2 & \text{if $|Z'| = 1$};\\
|Z'| + 2 & \text{if $1  < |Z'| < |Z|$; and}\\
|Z| + 1 & \text{if $|Z'| = |Z|.$}
\end{cases}
\end{equation*}
\end{itemize}
A prickly $3$-separating set of $M$ will also be called a {\it prickly $3$-separator} of $M$. 
Observe that, when $Z$ is a prickly $3$-separating set, for all distinct $z$ and $z'$ in $Z$, the $2$-polymatroid $M\ba z$ has $(\{z'\}, E - \{z,z'\})$ as a $2$-separation. 


We are now able to formally state the main result of \cite{oswww}. 
Recall that a $2$-polymatroid is {\it pure} if every individual element  has rank $2$. It is {\it non-empty} if its ground set is non-empty. 

\begin{theorem}
\label{lastmainone}
Let $M$ be a $3$-connected non-empty $2$-polymatroid. Then one of the following holds.
\begin{itemize}
\item[(i)] $M$ has an element $e$ such that $M\ba e$ or $M/e$ is $3$-connected; 
\item[(ii)] $M$ has rank at least three and is a whirl or the cycle matroid of a wheel; or 
\item[(iii)] $M$ is a pure $2$-polymatroid having a  prickly $3$-separating set. Indeed, every minimal $3$-separating set $Z$ with at least two elements is prickly, and $M\downarrow z$ is $3$-connected and pure for all $z$ in $Z$.
\end{itemize}
\end{theorem}

In \cite{oswww}, we gave a number of examples to show the need for the third part of the above theorem. It is worth noting here, since it contrasts with what we have already mentioned and what will feature in the main result of this paper, the operation of deletion used in the last theorem does not incorporate compactification. In the main result of this paper, we will incorporate compactification as part of deletion but we will no longer need to allow arbitrary prickly compressions, only those that arise from a $2$-element prickly $3$-separator. 
Let $Z$ be such a set in a $2$-polymatroid $M$. For $z$ in $Z$, we will call $M\downarrow z$ a {\it series compression} of $M$.

For a compact $2$-polymatroid $M_1$, we call  $M_2$ an {\it s-minor} of $M_1$ if $M_2$ can be obtained from $M_1$ by  a sequence of contractions, deletions followed by compactifications, and series compressions.  The next result   is the main theorem of the paper. It concerns s-minors of 3-connected 2-polymatroids. Such a 2-polymatroid is compact provided it has at least three elements. 

\begin{theorem}
\label{mainone}
Let $M$ be a $3$-connected  $2$-polymatroid and $N$ be a $3$-connected proper  s-minor of $M$ having at least four elements.  Then one of the following holds.
\begin{itemize}
\item[(i)] $M$ has an element $e$ such that $M/ e$ is  $3$-connected having an s-minor isomorphic to $N$; or
\item[(ii)] $M$ has an element $e$ such that $(M\ba e)^{\flat} $ is  $3$-connected having an s-minor isomorphic to $N$; or
\item[(iii)] $M$ has a two-element  prickly $3$-separating set $Z$ such that, for each $z$ in $Z$, the series compression $M\downarrow z$ is  $3$-connected having an s-minor isomorphic to $N$; or
\item[(iv)] $r(M) \ge 3$ and $M$ is a whirl or the cycle matroid of a wheel.
\end{itemize}
\end{theorem}

For compact $2$-polymatroids $M_1$ and $M_2$, we call  $M_2$ a {\it c-minor} of $M_1$ if $M_2$ can be obtained from $M_1$ by  a sequence of operations each consisting of a contraction or of a deletion followed by a compactification. As we shall show in Section~\ref{redc}, the last theorem can be proved by establishing the following result.

\begin{theorem}
\label{modc0}
Let $M$ be a $3$-connected  $2$-polymatroid and $N$ be a $3$-connected proper  c-minor of $M$ having at least four elements.  Then one of the following holds.
\begin{itemize}
\item[(i)] $M$ has an element $e$ such that $M/ e$ is  $3$-connected having a c-minor isomorphic to $N$; or 
\item[(ii)] $M$ has an element $e$ such that $(M\ba e)^{\flat} $ is  $3$-connected having a c-minor isomorphic to $N$; or 
\item[(iii)] $M$ has a prickly $3$-separator $\{y,z\}$ such that $M\da y$ is $3$-connected having a c-minor isomorphic to $N$; or 
\item[(iv)] $r(M) \ge 3$ and $M$ is a whirl or the cycle matroid of a wheel.
\end{itemize}
\end{theorem}

The paper is structured as follows. The next section includes some basic preliminaries. In Sections~\ref{clc} and \ref{pc2s}, we develop a number of results relating to connectivity and local connectivity, and to parallel connection and 2-sums. In Section~\ref{strat}, we describe the strategy for proving Theorem~\ref{mainone}. That section serves as a guide to the remaining sections of the paper, with the purpose of each of these sections being to complete an identified step in the proof.  Section~\ref{redc} plays an important role in this proof by showing that the main theorem can be proved by adding the assumption that all series compressions are performed last in the production of an s-minor of $M$ isomorphic to $N$. That result is helpful but it cannot obscure the fact that the proof of Theorem~\ref{mainone} is complex with some subtleties in the logic that need to be carefully negotiated.


\section{Preliminaries}
\label{prelim}

Much of the terminology for matroids carries over to $2$-polymatroids. For example, suppose $x$ and $y$ are distinct points of a $2$-polymatroid $M$, that is, $r(\{x\}) = 1 = r(\{y\})$. If $r(\{x,y\}) = 1$, then $x$ and $y$ are {\it parallel points} of $M$. On the other hand, if   
$r(E - \{x,y\}) = r(E) -1 < r(E - x) = r(E-y)$, then $\{x,y\}$ is a {\it series pair of points} of $M$. Evidently,   if $\{x,y\}$ is a parallel or series pair of points,  then 
$\lambda_M(\{x,y\}) \le 1$. If $x$ and $y$ are distinct lines of $M$ and $r(\{x,y\}) = 2$, then $x$ and $y$ are {\it parallel lines} of $M$.

One tool that is used repeatedly in our earlier work is the submodularity of the connectivity function. Once again, this will play a vital role here. Partitions $(X_1,X_2)$ and $(Y_1,Y_2)$  of a set $E$ are said to {\it cross} if all four of the sets \linebreak $X_1\cap Y_1,$  $X_1 \cap Y_2,  X_2 \cap Y_1$, and $X_2 \cap Y_2$ are non-empty. We shall frequently encounter crossing partitions of the ground set of a $2$-polymatroid. We shall use the term  {\it by uncrossing} to refer to an application of  the submodularity of the connectivity function. 

In this paper, we shall frequently switch between considering the deletion $M\ba X$ of a set $X$ of elements of a $2$-polymatroid $M$ and  the compactification $(M\ba X)^{\flat}$ of this deletion, which we shall sometimes call the {\it compactified deletion of $X$}.  We shall often use the following abbreviated notation for the latter: 
$$(M\ba X)^{\flat} = M\baba X.$$

We shall often encounter the situation when we have a $2$-polymatroid $M$ such that $M^{\flat}$ is  $3$-connected although $M$ itself is not. This occurs when $M$ has a line $\ell$ such that $(\{\ell\},E - \ell)$ is a $2$-separation. We call such a $2$-separation of $M$ {\it trivial}. Thus, in general, a partition $(X,Y)$ of $E$ is a {\it non-trivial $2$-separation} of $M$ if $\lambda_M(X) \le 1$ and $\min\{|X|,|Y|\} \ge 2$.

For a $2$-polymatroid $M$,  we recall that a minor of $M$ is any $2$-polymatroid that can be obtained from $M$ by a sequence of contractions and deletions where, here, deletions are not automatically accompanied by compactifications.  When $M$ and $N$ are compact, we defined   $N$ to be a    c-minor  of $M$ if it can be obtained from  $M$ by  a sequence of contractions and deletions followed by compactifications.  In the proof of  Theorem~\ref{modc0}, it  is  convenient to be able to separate the compactifications from the deletions. Thus we define a {\it c-minor} of an arbitrary 2-polymatroid $M$ to be any $2$-polymatroid that can be obtained from $M$ by  a sequence of contractions,  deletions, and compactifications. As we shall show in Corollary~\ref{complast2}, this extension of the definition is consistent with our original definition. 
For a $2$-polymatroid $N$,  a {\it special $N$-minor} of $M$ is any c-minor of $M$ that is either equal to $N$ or differs from $N$ by having a single point relabelled.




\begin{lemma}
\label{complast}
Let $P$ and $Q$ be $2$-polymatroids such that $Q$ can be obtained from $P$ by a sequence of deletions, contractions, and compactifications with the last move being a compactification. Then $Q$ can be obtained from $P$ by the same sequence of deletions and contractions with none of the compactifications being done except for the last move.
\end{lemma}

To prove this lemma, we shall require a preliminary result.

\begin{lemma}
\label{complast1}
Let $P$ be the $2$-polymatroid $(E,r)$. For $A \subseteq E$,
\begin{itemize}
\item[(i)] $(P^{\flat}\ba A)^{\flat} = (P\ba A)^{\flat}$; and 
\item[(ii)] $(P^{\flat}/ A)^{\flat} = (P/ A)^{\flat}$.
\end{itemize}
\end{lemma}

\begin{proof} Let $P_1$ be a $2$-polymatroid with ground set $E$  and rank function $r_1$. Then, for $X \subseteq E - A$, we have 
\begin{eqnarray}
\label{xae}
r_{(P_1 \ba A)^{\flat}}(X) & =  &r_{P_1 \ba A}(X) + \sum_{x \in X} [\lambda_{P_1 \ba A}(\{x\})  - r_{P_1\ba A}(\{x\})] \nonumber \\
					&=& r_1(X) + \sum_{x \in X} [r_1(E - A - x)  - r_1(E - A)].
					\end{eqnarray}
Thus
\begin{equation}
\label{xae0}
r_{(P  \ba A)^{\flat}}(X)	= r(X) + \sum_{x \in X} [r(E - A - x)  - r(E - A)].
\end{equation}

Next we observe that, for $x$ in $X$, 
\begin{align}
\label{xae1}
r_{P^{\flat}}(E-A - x) - r_{P^{\flat}}(E-A) & =   r(E - A - x) +   \sum_{y \in E-A-x} [\lambda(\{y\})  - r(\{y\})]  \nonumber \\
& \hspace*{0.75in} -  r(E - A) -   \sum_{y \in E-A} [\lambda(\{y\})  - r(\{y\})] \nonumber \\
& =  r(E - A - x) - r(E - A) - \lambda(\{x\})  + r(\{x\}).
					 					\end{align}
Thus, by (\ref{xae}), (\ref{xae1}),  and (\ref{xae0}), 
\begin{align*}
r_{(P^{\flat} \ba A)^{\flat}}(X) &= r_{P^{\flat}}(X) + \sum_{x \in X} [r_{P^{\flat}}(E - A - x)  - r_{P^{\flat}}(E - A)]\\
						& =  r(X)	+ \sum_{x \in X} [\lambda(\{x\}) - r(\{x\}) +r(E - A - x)  - r(E - A) \\
						&  \hspace*{2.5in} - \lambda(\{x\}) + r(\{x\})]\\
						& =  r(X) + 	\sum_{x \in X} [r(E - A - x)  - r(E - A)]\\
						& = 	r_{(P  \ba A)^{\flat}}(X).	
\end{align*}
We conclude that (i) holds. 

Again, for $X \subseteq E - A$, we have 
\begin{eqnarray}
\label{xae2}
r_{(P_1 / A)^{\flat}}(X) & =  &r_{P_1 / A}(X) + \sum_{x \in X} [\lambda_{P_1 / A}(\{x\})  - r_{P_1/ A}(\{x\})] \nonumber \\
					&=& r_1(X \cup A) -r_1(A) + \sum_{x \in X} [r_{P_1 / A}(E - A - x)  - r_{P_1 / A}(E - A)]  \nonumber \\
					& = & r_1(X \cup A) -r_1(A) + \sum_{x \in X} [r_{1}(E   - x)  - r_{1}(E)].
					\end{eqnarray}
Thus
\begin{equation}
\label{xae3}
r_{(P  / A)^{\flat}}(X)	= r(X \cup A)  -r(A) + \sum_{x \in X} [r(E - x)  - r(E)].
\end{equation}
Therefore, by (\ref{xae2}), (\ref{xae1}), and (\ref{xae3})
\begin{align*}
r_{(P^{\flat} /A)^{\flat}}(X) &= r_{P^{\flat}}(X \cup A) -r_{P^{\flat}}(A) + \sum_{x \in X} [r_{P^{\flat}}(E   - x)  - r_{P^{\flat}}(E)]\\
						& = r(X \cup A) -r (A)  	+ \sum_{x \in X} [ \lambda(\{x\}) - r(\{x\}) +r(E - x)  - r(E)\\
						&  \hspace*{2.6in} - \lambda(\{x\}) + r(\{x\})]\\
						& =  r(X \cup A) -r (A)  	+ \sum_{x \in X} [r(E - x)  - r(E)]\\
						& = 	r_{(P / A)^{\flat}}(X).	
\end{align*}
Hence (ii) holds.
\end{proof}

\begin{proof}[Proof of Lemma~\ref{complast}.]
We may assume that there are disjoint subsets $A_1,A_2,\ldots,A_n$ of $E$ such that, in forming $Q$ from $P$, these sets are removed in order via deletion or contraction with the possibility that, after each such move, a compactification is performed. To prove the lemma, we argue by induction on $n$. 
It follows immediately from Lemma~\ref{complast1} that the lemma holds if $n= 1$. Assume the result holds for $n < m$ and let $n = m\ge 2$. Then there is a $2$-polymatroid $R$ such that $Q$ is $(R \ba A_n)^{\flat}$ or $(R / A_n)^{\flat}$, so, by Lemma~\ref{complast1},  $Q$ is $(R^{\flat} \ba A_n)^{\flat}$ or $(R^{\flat} / A_n)^{\flat}$, respectively. In forming $R$, a certain sequence of deletions, contractions, and compactifications is performed. Let $R_0$ be the $2$-polymatroid that is obtained from $P$ by performing the same sequence of operations except for the compactifications. Then, by the induction assumption, $R^{\flat} = R_0^{\flat}$. Since 
$(R^{\flat} \ba A_n)^{\flat} = (R_0^{\flat} \ba A_n)^{\flat} = (R_0 \ba A_n)^{\flat}$ and $(R^{\flat} / A_n)^{\flat} =(R_0^{\flat} / A_n)^{\flat} = (R_0 / A_n)^{\flat}$, the lemma follows by induction. 
\end{proof}

The following are straightforward consequences  of Lemma~\ref{complast}. We prove only the second of these. 

\begin{corollary}
\label{complast3}
Let $P$ and $Q$ be  $2$-polymatroids such that $Q$ is compact. Then $Q$ is a c-minor of $P$ if and only if $Q$ can be obtained from $P$ by a sequence of deletions and contractions followed by a single compactification.
\end{corollary}

\begin{corollary}
\label{complast2}
Let $P$ and $Q$ be compact $2$-polymatroids. Then $Q$ is a c-minor of $P$ if and only if $Q$ can be obtained from $P$ by a sequence of operations each of which consists of either  a contraction or a deletion followed by a compactification.
\end{corollary}

\begin{proof} We need to show that if $Q$ is a c-minor of $P$, then $Q$ can be obtained as described. Now $Q^{\flat} = Q$. Thus, by Lemma~\ref{complast}, $Q$ can be obtained from $P$ by a sequence of deletions and contractions with one compactification being done as the final move. By Lemma~\ref{complast1}, we can perform a compactification after each deletion and still obtain $Q$ at the end of the process. Since $P$ is compact and each contraction of a compact $2$-polymatroid is compact, we retain compactness throughout this sequence of moves, so the result holds.
\end{proof}

\begin{lemma}
\label{compact2}
Let $M$ be a polymatroid. Then 
$$(M^{\flat})^* = M^* = (M^*)^{\flat}.$$
\end{lemma}

\begin{proof}
By Lemma~\ref{compact0}(i), $M^*$ is compact, so $M^* = (M^*)^{\flat}.$ Also, by Lemma~\ref{compact0}(ii), 
  $(M^{\flat})^* = ((M^*)^*)^* = (M^*)^{\flat}$.
\end{proof}

\begin{lemma}
\label{csm}
Let $P$ and $Q$ be $2$-polymatroids, where $Q$ is compact.  Then 
$P$ has a c-minor isomorphic to $Q$ if and only if $P^*$ has a c-minor isomorphic to $Q^*$. 
\end{lemma}

\begin{proof} Suppose $P$ has a c-minor isomorphic to $Q$. By Corollary~\ref{complast3}, $Q$ can be obtained from $P$ by a sequence of deletions and contractions with one compactification being done as the final move. By Lemma~\ref{complast1}, we can perform a compactification after each deletion and after each contraction and still obtain $Q$ at the end of the process. Indeed, since $(P^{\flat}\ba A)^{\flat} = (P\ba A)^{\flat}$  and 
  $(P^{\flat}/ A)^{\flat} = (P/ A)^{\flat}$, we see that $P^{\flat}$ has a c-minor isomorphic to $Q$. Thus we may assume that, in forming $Q$ from $P^{\flat}$, we remove, in order, disjoint sets $A_1, A_2,\ldots, A_n$ where each such removal is followed by a compactification. To prove that 
  $P^*$ has a c-minor isomorphic to $Q^*$, we shall argue by induction on $n$. 
  
Suppose $n = 1$. Then $Q$ is  $(P^{\flat}\ba A_1)^{\flat}$ or $(P^{\flat}/ A_1)^{\flat}$.
Then, by Lemmas~\ref{compact0} and \ref{compact2},
 $$((P^{\flat}\ba A_1)^{\flat})^* = (((P^*)^*\ba A_1)^{\flat})^*= ((P^*/A_1)^*)^* = (P^*/A_1)^{\flat} = P^*/A_1$$
 and
$$((P^{\flat}/ A_1)^{\flat})^* = (P^{\flat}/ A_1)^* = ((P^{\flat})^*\ba A_1)^{\flat} = (P^*\ba A_1)^{\flat}.$$
Since $Q$ is compact, we deduce that the result holds for $n = 1$. Assume it holds for $n< k$ and let $n = k \ge 2$. Then there is a compact $2$-polymatroid $R$ that is a c-minor of $P$ such that 
$Q$ is $(R\ba A_n)^{\flat}$ or $(R/A_n)^{\flat}$. By the induction assumption, $R^*$ is a c-minor of $P^*$, and $Q^*$ is a c-minor of $R^*$. Hence 
$Q^*$ is a c-minor of $P^*$.

For the converse, we note that, by what we have just proved, if $Q^*$ is a c-minor of $P^*$, then $(Q^*)^*$ is a c-minor of $(P^*)^*$, that is, 
$Q^{\flat}$ is a c-minor of $P^{\flat}$. But $Q$ is compact so $Q$ is a c-minor of $P^{\flat}$. Hence $Q$ is a c-minor of $P$.
\end{proof}



When we compactify a $2$-polymatroid, loosely speaking what we are doing is dealing simultaneously with a number of $2$-separations. It will be helpful to be able to treat these $2$-separations one at a time. In the introduction, we defined the compression $M\downarrow x$ for an element $x$ of a $2$-polymatroid $M$. Ultimately, that operation removes $x$. Let $M\underline{\downarrow} x$ be the $2$-polymatroid that is obtained from $M$ by freely adding an element $x'$ on $x$ and then contracting $x'$. Thus $M\underline{\downarrow} x$ has ground set $E$ and rank function given, for all subsets $X$ of $E$, by 
\begin{equation} 
\label{getdown4 }
r_{M\underline{\downarrow} x}(X) = 
\begin{cases} 
r(X), & \text{if $r(x) = 0$, or $r(X \cup x) > r(X)$;   and}\\
r(X) - 1, & \text{otherwise.}
\end{cases}
\end{equation}
We shall say that $M\underline{\downarrow} x$ has been obtained by {\it compactifying} $x$. Evidently 

$$M\da x = (M\ud x) \ba x.$$

\begin{lemma}
\label{compel}
Let $M$ be a $2$-connected $2$-polymatroid that is  not compact. Let $Z$ be the set of lines $z$ of $M$ such that $\lambda(\{z\}) = 1$. Then 
$$M^{\flat} = ((\ldots((M\underline{\downarrow} z_1)\underline{\downarrow} z_2)\ldots)\underline{{\downarrow}} z_n)$$
where $Z = \{z_1,z_2,\ldots,z_n\}$.
\end{lemma}

\begin{proof} We argue by induction on $n$. Suppose $n = 1$. Let $X \subseteq E(M)$. Then  
\begin{equation*} 
\label{getdown2}
r_{M^{\flat}}(X) = 
\begin{cases} 
r(X), & \text{if $z_1 \not \in X$;   and}\\
r(X) - 1, & \text{otherwise.}
\end{cases}
\end{equation*}
On the other hand,
\begin{equation*} 
\label{getdown3}
r_{M\underline{\downarrow} z_1}(X) = 
\begin{cases} 
r(X), & \text{if   $r(X \cup z_1) > r(X)$;   and}\\
r(X) - 1, & \text{otherwise.}
\end{cases}
\end{equation*}
The result is easily checked in this case. 

Now assume that $n \ge 2$ and that the lemma holds if $|Z| \le n-1$.   Let $M_1 = M\underline{\downarrow} z_1$. Then $M_1$ is easily shown to be $2$-connected having $\{z_2,z_3,\ldots,z_n\}$ as its set of lines $z$ for which $\lambda_{M_1}(\{z\}) = 1$. Thus, by the induction assumption, 
$$M_1^{\flat} = ((\ldots((M_1\underline{\downarrow} z_2)\underline{\downarrow} z_3)\ldots)\underline{{\downarrow}} z_n).$$
Since $M_1 = M\underline{\downarrow} z_1$, it suffices to show that $M_1^{\flat} = M^{\flat}.$

Suppose $X \subseteq E$. Then 
$$r^{\flat}(X) = r(X) + \sum_{x \in X}(\lambda(\{x\}) - r(\{x\})).$$
Now 
\begin{equation*} 
\label{getdown5.1}
r_{M_1}(X) = 
\begin{cases} 
r(X), & \text{if   $r(X \cup z_1) > r(X)$;   and}\\
r(X) - 1, & \text{otherwise.}
\end{cases}
\end{equation*}
Thus 
\begin{equation*} 
\label{getdown6}
r_{M_1}(X) = 
\begin{cases} 
r(X), & \text{if   $z_1 \not \in X$;   and}\\
r(X) - 1, & \text{otherwise.}
\end{cases}
\end{equation*}
Hence 
\begin{eqnarray*}
r_{M_1^{\flat}}(X)  & = & r_{M_1}(X) + \sum_{x \in X}(\lambda_{M_1}(\{x\}) - r_{M_1}(\{x\}))\\
				& = & r_{M_1}(X) + \sum_{x \in X \cap (Z - z_1)}(\lambda_{M}(\{x\}) - r_{M}(\{x\}))\\
				& = &  r_{M}(X) + \sum_{x \in X \cap Z}(\lambda_{M}(\{x\}) - r_{M}(\{x\}))\\
				& = & r^{\flat}(X).
\end{eqnarray*}
We conclude, by induction, that the lemma holds.
\end{proof}

We will need some elementary properties of deletion, contraction, and series compression.

\begin{lemma}
\label{elemprop}
Let $A$ and $B$ be disjoint subsets of the ground set $E$ of a $2$-polymatroid $P$. Then 
\begin{itemize}
\item[(i)] $P/A/B = P/ (A \cup B) = P/B/A$;
\item[(ii)] $P\baba A\baba B = P\baba (A \cup B) = P\baba B \baba A$; and 
\item[(iii)] $P/ A\baba B =  P\baba B / A$.
\end{itemize}
\end{lemma}

\begin{proof} Because the proofs of all three parts are routine, we only include a proof of (iii). 
Suppose $X \subseteq E- (A \cup B)$. Then 

\begin{align*}
r_{P/A\baba B}(X) &= r_{((P/A)\ba  B)^{\flat}}(X)\\
			    &= r_{P/A}(X) + \sum_{x \in X}[\lambda_{P/A\ba B}(\{x\}) - r_{P/A\ba B}(\{x\})]\\
			    & = r_{P/A}(X) + \sum_{x \in X}[r_{P/A}(E-A-B-x) - r_{P/A}(E-A-B)]\\
			     & = r(X\cup A) - r(A) + \sum_{x \in X}[r(E-B-x) - r(E-B)]\\
			    & = r(X\cup A) - r(A) + \sum_{x \in X}[\lambda_{P\ba B}(\{x\}) - r_{P\ba B}(\{x\})]\\
			    & = r_{P\baba B}(X\cup A) - r_{P\baba B}(A)\\
			    & = r_{P\baba B/A}(X).
			    \end{align*}
We conclude that (iii) holds.
\end{proof}

The remainder of this section presents a number of basic properties of 2-element prickly 3-separators and of the compression operation.

\begin{lemma}
\label{atlast}
Let $P$ be a $2$-polymatroid having $j$ and $k$ as lines and with $r(\{j,k\}) = 3$. Suppose $X \subseteq E(P) - k$ and $j \in X$. Then $r_{P\da k}(X) = r(X \cup k) - 1$.
\end{lemma}

\begin {proof} 
By definition, 
\begin{equation*}
r_{P\da k}(X) = 
\begin{cases} 
r(X), & \text{if   $r(X \cup k) > r(X)$;}\\
r(X) - 1, & \text{otherwise.}
\end{cases}
\end{equation*}
As $j \in X$ and $\sqcap(j,k) = 1$, it follows that $r(X \cup k)$ is $r(X)$ or $r(X) + 1$. It follows that 
$r_{P\da k}(X) = r(X \cup k) - 1$.
\end{proof}

\begin{lemma}
\label{atlast2}
Let $P$ be a $2$-polymatroid having $j$ and $k$ as lines and with $r(\{j,k\}) = 3$. Suppose $\ell$ is a line of $P$ that is not in $\{j,k\}$ and is not parallel to $k$. Then $\{\ell\}$ is  $2$-separating in $P$ if and only if it is $2$-separating in $P\da k$. 
\end{lemma}

\begin {proof}  Clearly $\{\ell\}$ is $2$-separating in $P$ if and only if $r(E- \ell) \le r(E) - 1$. Since $\ell$ is not parallel to $k$, we see that $r_{P\da k}(\ell) = r(\ell) = 2$. Now $\{\ell\}$ is $2$-separating in $P\da k$ if and only if $r_{P\da k}(E- \{k,\ell\}) \le r_{P\da k}(E- k) - 1$. By Lemma~\ref{atlast}, the last inequality holds if and only if $r(E - \ell) - 1 \le r(E) - 1 - 1.$ We conclude that the lemma holds.
\end{proof}

\begin{lemma}
\label{symjk}
Let $\{j,k\}$ be a prickly $3$-separator in a $2$-polymatroid $P$. 
Then $P\da j$ can be obtained from $P\da k$ by relabelling $j$ as $k$. 
\end{lemma}

\begin{proof} 
Suppose $X \subseteq E - \{j,k\}$. Then, since both $r(X \cup j)$ and $r(X \cup k)$ exceed $r(X)$, 
$$r_{P\da j}(X) = r_P(X) = r_{P\da k}(X).$$

Now, as $\sqcap(j,k) = 1$, it follows that either 
$r(X \cup j \cup k) = r(X \cup j) + 1$, or $r(X \cup j \cup k) = r(X \cup j)$. Thus 
$r_{P\da k}(X \cup j) = r(X \cup j \cup k) - 1$. By symmetry, 
$r_{P\da j}(X \cup k) = r(X \cup j \cup k) - 1$, and the lemma follows.
\end{proof}

The first part of the next lemma was proved by Jowett, Mo, and Whittle~\cite[Lemma 3.6]{jmw}.

\begin{lemma}
\label{elemprop23}
Let  $P$ be a compact polymatroid $(E,r)$.  For $A\subseteq E$,  
\begin{itemize}
\item[(i)] $P/A$ is compact; and
\item[(ii)] if $P$ is a $2$-polymatroid and $\{j,k\}$ is a prickly $3$-separator of $P$, then $P\da k$ is compact. 
\end{itemize}
\end{lemma}

\begin{proof} We prove (ii). 
It suffices to show that $r_{P\da k}(E - k - y) = r_{P\da k}(E-k)$ for all $y$ in $E - k$. Since $P$ is compact, $r(E- k) = r(E)$, so 
$r_{P\da k}(E-k) = r(E) - 1$. Now 
\begin{equation*} 
\label{getdown5}
r_{P\da k}(E - k - y) = 
\begin{cases} 
r(E-k-y), & \text{if   $r(E-y) -1 \ge r(E-y - k)$;   and}\\
r(E-k-y) - 1, & \text{otherwise.}
\end{cases}
\end{equation*}
It follows that 
$r_{P\da k}(E - k - y) = r(E) - 1 = r_{P\da k}(E-k)$ unless $r(E- k - y) = r(E-y) - 2$. Consider the exceptional case. Evidently $y \neq j$ as $r(E- k - j) = r(E-j) - 1$. Thus $j \in E- k -y$. Since $\sqcap(j,k) = 1$, it follows that $r(E- y) \le r(E-k-y) + 1$. This contradiction completes the proof of (ii). 
\end{proof}

\begin{lemma}
\label{pricklytime0} 
In a $2$-polymatroid $P$, let $k$ and $y$ be distinct elements. Then 
\begin{itemize}
\item[(i)] $P\da k \ba y = P\ba y \da k$; and 
\item[(ii)] $P\da k /y = P/y \da k.$
\end{itemize}
\end{lemma}

\begin{proof} Part (i) is essentially immediate. We now prove (ii). 
If $r(\{k\}) \le 1$, then $P\da k = P/k$, so 
$$P\da k /y = P/k/y = P/y/k = P/y \da k.$$
Thus we may assume that $r(\{k\}) = 2$.

Suppose $y$ is a line such that $r(\{y,k\}) = 2$.   Then 
$$P/y\da k  = P/y/k = P/k/y = P\da k/y$$ 
where the last equality follows by considering how $P\da k$ is constructed. 
Thus we may assume that $y$ is not a line that is parallel to $k$. Hence 
$$r(\{y,k\}) > r(\{y\}).$$ 
Let $X$ be a subset of $E - k - y$. Then 
$$r_{P\da k /y}(X)  = r_{P\da k}(X \cup y) - r_{P\da k}(\{y\}) = r_{P\da k}(X \cup y) - r(\{y\})$$
where the second equality follows because  $r(\{y,k\}) > r(\{y\}).$ We deduce that 
\begin{equation*}
r_{P\da k /y}(X) = 
\begin{cases}
r(X\cup y) - r(\{y\}), & \text{if $r(X \cup y \cup k) > r(X \cup y)$};\\
r(X\cup y) - r(\{y\}) - 1, & \text{otherwise}.
\end{cases}
\end{equation*}
On the other hand, 
since $r_{P/y}(X \cup k) = r(X \cup k \cup y) - r(\{y\})$ and $r_{P/y}(X) = r(X  \cup y) - r(\{y\})$, we see that 
\begin{equation*}
r_{P/y\da k }(X) = 
\begin{cases}
r(X\cup y) - r(\{y\}), & \text{if $r(X \cup y \cup k) > r(X \cup y)$};\\
r(X\cup y) - r(\{y\}) - 1, & \text{otherwise}.
\end{cases}
\end{equation*}
Thus 
$$r_{P\da k /y}(X) = r_{P/y\da k }(X)$$
so the lemma holds.
\end{proof}

\begin{lemma}
\label{elemprop24}
Let  $P$ be a compact $2$-polymatroid $(E,r)$ having $\{j,k\}$ as a prickly $3$-separator. Suppose $y \in E - \{j,k\}$.  If $\{j,k\}$ is not a prickly $3$-separator of $P/y$, then  
\begin{itemize}
\item[(i)] $r(\{j,k,y\}) = 3$ and $P/y$ has $\{j,k\}$ as a $1$-separating set; or
\item[(ii)] $P\da k/y = P/y\baba k$; or 
\item[(iii)] $P\da k/y = P/y/k$; or 
\item[(iv)] $P\da k/y$ can be obtained from $P/y\baba j$ by relabelling $k$ as $j$.
\end{itemize}
\end{lemma}

\begin{proof} Suppose first that $r_{P/y}(\{j,k\}) = 1$. Then $y$ is a line of $P$ that is in the closure of $\{j,k\}$. Thus $\lambda_{P/y}(\{j,k\}) = 0$ and (i) holds. 

Next assume that $r_{P/y}(\{j,k\}) = 2$. Then $\lambda_{P/y}(\{j,k\}) = 1$. Thus  $P/y$ can be written as the 2-sum, with basepoint $p$ of two polymatroids, one of which, $P_1$, has ground set $\{j,k,p\}$ and has rank $2$. As $P$ is compact, so is $P/y$. There are four choices for $P_1$:
\begin{itemize}
\item[(a)] $j$ and $k$ are parallel lines and $p$ is a point lying on them both; 
\item[(b)] $P_1$ is isomorphic to the matroid $U_{2,3}$; 
\item[(c)] $P_1$ has $k$ as a line and has $j$ and $p$ as distinct points on this line; or
\item[(d)] $P_1$ has $j$ as a line and has $k$ and $p$ as distinct points on this line;
\end{itemize}
By Lemma~\ref{pricklytime0}, $P\da k/y = P/y \da k$. If $P_1$ is one of the $2$-polymatroids in (b) or (d), then, as $k$ is a point of $P/y$, it follows that $P/y \da k = P/y / k$, so (iii) holds. Next suppose that $P_1$ is the $2$-polymatroid in (a). Then, as $P/y$ is compact, it follows that $P/y \da k = P/y \baba k$, so (ii) holds. Finally, suppose that $P_1$ is the $2$-polymatroid in (c). Then $P\da j/y = P/y \da j = P/y \baba j$. By Lemma~\ref{symjk}, 
$P\da j$ can be obtained from $P\da k$ by relabelling $j$ as $k$. Thus $P\da j/y$ can be obtained from $P/y \baba j$  by relabelling $k$ as $j$, that is, (iv) holds.

We may now assume that $r_{P/y}(\{j,k\}) = 3$. Then $\sqcap(y, \{j,k\}) = 0$ and one easily checks that $\{j,k\}$ is a prickly $3$-separator of $P/y$; a contradiction. 
\end{proof}

\begin{lemma}
\label{elemprop25}
Let  $\{j,k\}$ be a prickly $3$-separator in a $2$-polymatroid $P$. If $P\da k$ is $3$-connected, then so is $P$.
\end{lemma}

\begin{proof}
Let $(X,Y)$ be an exact $m$-separation of $P$ for some $m$ in $\{1,2\}$ where $k \in X$. 
Then $r(X) + r(Y) - r(P) = m-1$. Now $r(P\da k) = r(P) -1$. 

Consider $r_{P\da k}(X-k) + r_{P\da k}(Y)$. Suppose first that $j \in X - k$.  Then, by Lemma~\ref{atlast},  $r_{P\da k}(X - k) = r(X) - 1$ and $r_{P\da k}(Y) = r(Y)$. Hence
$$r_{P\da k}(X - k) + r_{P\da k}(Y) - r(P\da k) = m-1.$$
As $P\da k$ is $2$-connected, we cannot have $m = 1$ since both $X - k$ and $Y$ are non-empty. Thus $m = 2$. Now $\max\{|X|,r(X)\} \ge 2$ and 
$\max\{|Y|,r(Y)\} \ge 2$. Thus $\max\{|Y|,r_{P\da k}(Y)\} \ge 2$. If $X = \{j,k\}$, then $r(X) + r(Y) - r(P) = 1$; a contradiction to the fact that $\{j,k\}$ is a $3$-separator of $P$. We deduce that $|X - k| \ge 2$, so $(X-k,Y)$ is a $2$-separation of $P\da k$; a contradiction.

We may now assume that $j \in Y$. Then 
$$r_{P\da k}(X - k) + r_{P\da k}(Y) - r(P\da k) \le r(X) - 1 + r(Y) - 1 -r(P) + 1 = m-2.$$
As $P\da k$ is $2$-connected, it follows that $X - k$ is empty. Then $r(\{k\}) + r(E-k) - r(E) = 1$; a \cn. Thus the lemma holds.
\end{proof}

\section{Some results for connectivity and local connectivity}
\label{clc}

This section notes a number of properties of the connectivity and local-connectivity functions that will be used in the proof of the main theorem. 
First  we show that  compression is, in most situations, a self-dual operation. We proved this result in \cite[Proposition 3.1]{oswww} for the variant of duality used there. By making the obvious replacements in that proof,  it is straightforward to check that the result holds with the modified definition of duality used here. We omit the details.

\begin{proposition}
\label{compdual}
Let $e$ be a line of a $2$-polymatroid $M$ and suppose that $M$ contains no line parallel to $e$.  Then 
$$M^*\downarrow e = (M\downarrow e)^*.$$
\end{proposition}

The next result implies that the main theorem is  a self-dual result.

\begin{proposition}
\label{sminordual}
Let $P$ and $Q$ be compact $2$-polymatroids. Then $Q$ is an  s-minor of $P$ if and only if $Q^*$ is an s-minor of $P^*$.
\end{proposition}

\begin{proof} By Lemma~\ref{compact0},  both $P^*$ and $Q^*$ are compact.  Moreover, $(P^*)^* = P$ and $(Q^*)^* = Q$. 
Assume $Q$ is an s-minor of $P$. To prove the lemma, it suffices to show that $Q^*$ is an s-minor of $P^*$. By Lemma~\ref{compact0} again, for an element $\ell$ of $P$, we have that 
   $(P\baba \ell)^* = P^*/ \ell$ and $(P/ \ell)^* = P^*\baba \ell$. Moreover, if $\{j,k\}$ is a prickly 3-separator of $P$, then one easily checks that it is a prickly  3-separator of $P^*$. By Lemma~\ref{elemprop23}, $P\da k$ is compact and, by Proposition~\ref{compdual}, $(P\da k)^* = P^*\da k$. Because the dual of each allowable move on $P$ produces a 2-polymatroid that is obtained from $P^*$ by an allowable move, the lemma follows by a straightforward induction argument.
\end{proof}

There is an attractive link between the connectivity of a $2$-polymatroid $M$ and the connectivity of the natural matroid associated with $M$.

\begin{lemma}
\label{missinglink}
Let $M$ be a $2$-polymatroid with at least two elements and let $M'$ be the natural matroid derived from $M$. Then 
\begin{itemize}
\item[(i)] $M$ is $2$-connected if and only if $M'$ is $2$-connected; and
\item[(ii)] $M$ is $3$-connected if and only if $M'$ is $3$-connected.
\end{itemize}
\end{lemma}

\begin{proof} 
The result is immediate if $M$ is a matroid or has a loop, so we may assume that $M$ is loopless and has at least one line. 
Let $L$ be the set of lines of $M$ and let $M^+$ be the  matroid that is obtained from $M$ by freely adding two points on each line in $L$.  Then $M' = M^+\ba L$.

Suppose that $M$ has a $k$-separation $(X,Y)$ for some $k$ in $\{1,2\}$. Replacing each line in each of $X$ and $Y$ by two points freely placed on the line gives sets $X'$ and $Y'$ that partition $E(M')$ such that $r(X') = r(X)$ and $r(Y') = r(Y)$. Hence $(X',Y')$ is a $k$-separation of $M'$. 

Now suppose that $M'$ has a $k$-separation for some $k$ in $\{1,2\}$. Choose such a $k$-separation $(X',Y')$ to minimize the number $m$ of lines of $M$ that have exactly one of the corresponding points of $M'$ in $X'$. If $m = 0$, then there is a $k$-separation of $M$ that corresponds naturally to $(X',Y')$. Thus we may assume that $M$ has a line $\ell$ whose corresponding points, $s_{\ell}$ and $t_{\ell}$, are in $X'$ and $Y'$, respectively. Now
\begin{equation}
\label{kay}
r(X') + r(Y') - r(M') = k-1.
\end{equation}

Suppose $|E(M')| = 3$. Then $M$ consists of a point and a line. For each $n$  in  $\{2,3\}$, both $M$ and $M'$ are $n$-connected   if and only if the point lies on the line. Thus the result holds if $|E(M')| = 3$. Now assume that $|E(M')| = 4$. Then $M$ consists of either two lines, or a line and two points. Again the result is easily checked. Thus we may assume that $|E(M')| \ge 5$. 
We may also assume that $|X'| \ge |Y'|$.  Then $|X'|\ge 3$.  Now $r(X'- s_{\ell}) + r(Y' \cup s_{\ell}) - r(M') \ge k$, otherwise the choice of $(X',Y')$ is contradicted. Thus $r(X' - s_{\ell}) = r(X')$ and $r(Y' \cup s_{\ell}) = r(Y') + 1$. Hence, in $M^+$, as $s_{\ell}$ and $t_{\ell}$ are freely placed on $\ell$, we see that 
$r((X' - s_{\ell}) \cup \ell) = r(X' - s_{\ell})$, so 
$$r(X' - s_{\ell}) = r((X' - s_{\ell}) \cup t_{\ell}) = r(X' \cup t_{\ell}).$$ Hence $(X' \cup t_{\ell},Y'-t_{\ell})$ violates the choice of $(X',Y')$ unless either  $k=1$ and $Y' = \{t_{\ell}\}$, or $k =2$ and $Y'$ consists of two points. In the first case,   $r(X') = r(M')$, so $r(X') + r(Y') - r(M') = 1,$ a contradiction to (\ref{kay}). In the second case, since one of the points in $Y'$ is $t_{\ell}$, the points are not parallel so $r(Y') = 2$ and $r(X') = r(M') - 1$. Thus $r(X' \cup t_{\ell}) = r(M') - 1$ and $r(Y'-t_{\ell}) = 1$; a contradiction to (\ref{kay}).
\end{proof}

Let $M$ be a  polymatroid $(E,r)$. If $X$ and $Y$ are subsets of $E$, the {\it local connectivity} 
$\sqcap(X,Y)$ between $X$ and $Y$ is defined by $\sqcap(X,Y) =  r(X) + r(Y) - r(X \cup Y)$. 
Sometimes we will write $\sqcap_M$ for $\sqcap$, and $\sqcap^*$ for $\sqcap_{M^*}$. It is straightforward to prove the following. Again this holds for both the version of duality used here and the variant used in \cite{oswww}.

\begin{lemma} 
\label{sqcapdual}
Let $M$ be a  polymatroid $(E,r)$. For disjoint subsets $X$ and $Y$ of $E$, 
$$\sqcap_{M^*}(X,Y) = \sqcap_{M/(E - (X \cup Y))}(X,Y).$$
\end{lemma}

The next lemma will be used repeatedly, often without explicit reference. Two sets $X$ and  $Y$ in a polymatroid $M$ are {\it skew} if $\sqcap(X,Y) = 0$.

\begin{lemma}
\label{skewer}
Let $M$ be a  $2$-polymatroid and $z$ be an element of $M$ such that $(A,B)$ is a $2$-separation of $M/z$. Suppose $z$ is  skew to   $A$.  Then 
$(A,B\cup z)$   is a  $2$-separation of $M$. Moreover,  if $M$ is $3$-connected, then $A$ is not a single line in $M/z$.
\end{lemma}

\begin{proof} Clearly $r(A \cup z) - r(z) = r(A)$, so $(A,B \cup z)$ is a \tws\ of $M$.    If $M$ is \thc\ and $A$ consists of a single line $a$ of $M/z$, then  $a$ is a line of $M$, so $a$ and $z$ are skew, and we obtain the \cn\ that $M$ has a \tws.
\end{proof}


Numerous properties of the connectivity function of a matroid are proved simply by applying properties of the rank function; they do not rely on the requirement that $r(\{e\}) \le 1$ for all elements $e$. Evidently, such properties also hold for the connectivity function of a polymatroid. The next few lemmas note some of these properties.

The first two are   proved in \cite[Lemmas 8.2.3 and 8.2.4]{oxbook}. 
\begin{lemma}
\label{8.2.3} Let $(E,r)$ be a polymatroid   and
let $X_1, X_2, Y_1$, and $Y_2$ be subsets of $E$ with $Y_1 \subseteq X_1$ and $Y_2 \subseteq X_2$. 
Then 
$$\sqcap(Y_1,Y_2) \le \sqcap(X_1,X_2).$$
\end{lemma}

\begin{lemma}
\label{8.2.4} Let $(E,r)$ be a polymatroid $M$ and
let $X,C$, and $D$ be disjoint subsets of $E$. 
Then 
$$\lambda_{M\ba D/C}(X) \le \lambda_M(X).$$
Moreover, equality holds if and only if 
$$r(X \cup C) = r(X) + r(C)$$ and 
$$r(E - X) + r(E-D) = r(E) + r(E- (X \cup D)).$$
\end{lemma} 

The following \cite[Corollary 8.7.6]{oxbook} is a straightforward consequence of the last 
lemma. 

\begin{corollary}
\label{8.7.6}
Let $X$ and $D$ be disjoint subsets of the ground set $E$ of a polymatroid $M$. Suppose that $r(M\ba D) = r(M)$. Then 
\begin{itemize}
\item[(i)] $\lambda_{M\ba D}(X) = \lambda_M(X)$ if and only if $D \subseteq \cl_M(E - (X \cup D))$; and 
\item[(ii)] $\lambda_{M\ba D}(X) = \lambda_M(X\cup D)$ if and only if 
$D \subseteq \cl_M(X)$.
\end{itemize}
\end{corollary}

It is well known that, when $M$ is a matroid, for all subsets $X$ of $E(M)$, 
$$\lambda_M(X) = r_M(X) + r_{M^*}(X) - |X|.$$
It is easy to check that the following variant on this holds for polymatroids. 
Recall that $||X|| = \sum_{x \in X} r(\{x\}).$

\begin{lemma}
\label{rr*}
In a  polymatroid $M$, for all subsets $X$ of $E(M)$,
$$\lambda_M(X) = r_M(X) + r_{M^*}(X) - ||X||.$$
In particular, if every element of $X$ has rank one, then 
$$\lambda_M(X) = r_M(X) + r_{M^*}(X) - |X|.$$
\end{lemma}

The next lemma contains another useful equation whose proof is straightforward.

\begin{lemma}
\label{obs1}
Let $(X,Y)$ be a partition of the ground set of a polymatroid $M$. Suppose $z \in Y$. Then 
$$\sqcap(X,\{z\}) + \sqcap_{M/z}(X,Y-z) =  \sqcap(X,\{z\}) + \lambda_{M/z}(X) = \lambda_M(X).$$
\end{lemma}

The  next two lemmas are natural generalizations of matroid results that appear in \cite{osw}. 

\begin{lemma} 
\label{univ} 
Let $(E,r)$ be a  polymatroid and
let $X$ and $Y$ be disjoint subsets of $E$. 
Then 
$$\lambda(X \cup Y) = \lambda(X) + \lambda(Y) - \sqcap(X,Y) - \sqcap^*(X,Y).$$
\end{lemma}

\begin{lemma}
\label{oswrules} Let $A, B, C$, and $D$ be subsets of the ground set of a polymatroid. Then 
\begin{itemize}
\item[(i)] $\sqcap(A \cup B, C \cup D) + \sqcap(A,B) + \sqcap(C,D) = 
\sqcap(A \cup C, B \cup D) + \sqcap(A,C) + \sqcap(B,D)$; and 
\item[(ii)] $\sqcap(A \cup B, C) + \sqcap(A,B) = 
\sqcap(A \cup C, B) + \sqcap(A,C).$
\end{itemize}
\end{lemma}

\begin{lemma}
\label{general}
Let $M$ be a polymatroid and $(A,B,Z)$ be a partition of its ground set into possibly empty subsets. 
Then 
$$\lambda_{M/Z}(A) = \lambda_{M\ba Z}(A) - \sqcap_M(A,Z) - \sqcap_M(B,Z) + \lambda_M(Z).$$
\end{lemma}

\begin{proof} 
We have $B = A \cup Z$ and $A = B \cup Z$. Then 
\begin{align*}
\lambda_{M/Z}(A) &= r_{M/Z}(A) + r_{M/Z}(B) - r(M/Z)\\
& = r(A \cup Z) - r(Z) + r(B\cup Z) - r(Z) -r(M) + r(Z)\\
& = r(B) + r(A) - r(M) - r(Z)\\
&= r(A) + r(B) - r(M\ba Z) + r(M\ba Z) - r(M) - r(Z)\\
& = \lambda_{M\ba Z}(A) + r(M\ba Z) - r(M) - r(Z).
\end{align*}

The required result holds if and only if 
$$\sqcap_M(A,Z) + \sqcap_M(B,Z) - \lambda_M(Z) = r(M) + r(Z) - r(M\ba Z).$$
Now 
\begin{align*}
\sqcap(A,Z) + \sqcap(B,Z) - \lambda_M(Z) &= r(A) + r(Z) - r(A \cup Z) + r(B) + r(Z) - r(B\cup Z)\\
& \hspace*{1.7in} - r(Z) - r(M\ba Z) + r(M)\\
& = r(A) + r(Z) - r(B) + r(B) + r(Z) - r(A) - r(Z)\\
& \hspace*{1.7in} - r(M\ba Z) + r(M)\\
& = r(M) + r(Z) - r(M\ba Z),
\end{align*}
as required.
\end{proof}

\begin{corollary}
\label{general2}
Let $M$ be a polymatroid and $(A,B,Z)$ be a partition of its ground set into possibly empty subsets. 
Suppose $r(M\ba Z) = r(M)$. Then
$$\lambda_{M/Z}(A) = \lambda_{M\ba Z}(A) - \sqcap_M(A,Z) - \sqcap_M(B,Z) + r(Z).$$
\end{corollary}

\begin{proof}
As $\lambda_M(Z) = r(Z) + r(M\ba Z) - r(M)$ and $r(M\ba Z) = r(M)$, the result is an immediate consequence of the last lemma.
\end{proof}

\begin{lemma}
\label{general3}
Let $M$ be a polymatroid and $(A,B,C)$ be a partition of its ground set into possibly empty subsets. 
Suppose $\lambda(A) = 1 = \lambda(C)$ and $\lambda(B) = 2$. Then $\sqcap(A,B) = 1$.
\end{lemma}

\begin{proof} 
We have 
$$2= r(B) + r(A\cup C) -r(M)$$ and 
\begin{align*}
r(M)  &= r(A\cup B) + r(C) - 1\\
& = r(A) + r(B) - \sqcap(A,B) + r(C) - 1\\
& = r(A) + r(B) + r(C) - 1 - \sqcap (A,B).
\end{align*}
Thus
$$2 = r(B) + r(A \cup C) - r(A) - r(B) - r(C) + 1 + \sqcap(A,B)$$
so $\sqcap(A,B) = 1 + \sqcap(A,C) \ge 1$. By Lemma~\ref{8.2.3}, $\sqcap(A,B) \le \sqcap(A,B\cup C) = 1$, so $\sqcap(A,B) = 1$.
\end{proof}

\begin{lemma}
\label{general4}
Let $M$ be a polymatroid and $(A,B,C)$ be a partition of its ground set into possibly empty subsets. 
 Then $\sqcap^*(A,B) = \lambda_{M/C}(A)$.
\end{lemma}

\begin{proof} 
By making repeated use of Lemma~\ref{compact0}, we have
\begin{align*}
\sqcap^*(A,B)  & = \sqcap_{M^*}(A,B)\\
 & = \lambda_{M^*\ba C}(A)\\
& = \lambda_{(M^*\ba C)^{\flat}}(A)\\
& =  \lambda_{(M/C)^*}(A)\\
& =  \lambda_{M/C}(A).
\end{align*}
\end{proof}

The following is a consequence of a result of Oxley and Whittle~\cite[Lemma 3.1]{owconn}.

\begin{lemma}
\label{Tutte2}
Let $M$ be a $2$-connected $2$-polymatroid with $|E(M)| \ge 2$. If $e$ is a point of $M$, then $M\ba e$ or $M/e$ is $2$-connected.
\end{lemma}

The next result is another straightforward extension of a matroid result. 

\begin{lemma}
\label{matroidef}
Let $M$ be a $2$-connected $2$-polymatroid having $e$ and $f$ as points. Then 
\begin{itemize}
\item[(i)] $\lambda_{M/ f}(\{e\}) = 0$ if and only if $e$ and $f$ are parallel points; and  
\item[(ii)] $\lambda_{M\ba f}(\{e\}) = 0$ if and only if $e$ and $f$ form a series pair.
\end{itemize}
\end{lemma}

\begin{proof} We prove (i) omitting the similar proof of (ii).  If $e$ and $f$ are parallel points of $M$, then $\lambda_{M/ f}(\{e\}) = 0$. Now assume that  $\lambda_{M/ f}(\{e\}) = 0$. Let $M'$ be the natural matroid derived from $M$. Then  $M'/f$ has $\{e\}$ as a component. Hence $\{e,f\}$ is a series or parallel pair in $M'$. But if $\{e,f\}$ is a series pair, then $M'/f$ is $2$-connected; a \cn. We conclude that $\{e,f\}$ is a   parallel pair of points in $M$, so (i) holds.
\end{proof}

The next result is a generalization of a lemma of Bixby \cite{bixby} (see also \cite[Lemma 8.7.3]{oxbook}) that is widely used when dealing with $3$-connected matroids.

\begin{lemma}
\label{newbix}
Let $M$ be a $3$-connected $2$-polymatroid and $z$ be a point of $M$. Then either 
\begin{itemize}
\item[(i)] $M/z$ is $2$-connected having one side of every $2$-separation being a pair of points of $M$ that are parallel in $M/z$; or
\item[(ii)] $M\ba z$ is $2$-connected having one side of every $2$-separation being either a single line of $M$, or a pair of points of $M$ that form a series pair in $M\ba z$.
\end{itemize}
\end{lemma}

\begin{proof} 
If $z$ lies on a line in $M$, then $M\ba z$ is \thc. Thus we may assume that $z$ does not lie on a line in $M$. Take the matroid $M'$ that is naturally derived from $M$. Then, by Bixby's Lemma, either $M'/z$ is $2$-connected having one side of every $2$-separation being a pair of parallel points of $M'$, or $M'\ba z$ is $2$-connected having one side of every $2$-separation being a series pair of points of $M'$. In the first case, if $\{a,b\}$ is a parallel pair of points of $M'/z$, then $\{a,b,z\}$ is a circuit of $M'$. Because the points added to $M$ to form $M'$ are freely placed on lines, we cannot have a circuit containing just one of them. Since $z$ is not on a line of $M$, we deduce that $a$ and $b$ are points of $M$. We conclude that, in the first case, (i) holds. 

Now suppose that $M'\ba z$ is not \thc\ and has $\{u,v\}$ as a series pair. Then either $u$ and $v$ are both matroid points of $M$,  or  $M$ has a line on which the points $u$ and $v$ are freely placed in the formation of $M'$. We deduce that (ii) holds.
\end{proof}

We recall from \cite{oswww} that, when   $\{a,b,c\}$ is a set of three points in a $2$-polymatroid $Q$, we call $\{a,b,c\}$  a {\it triangle} if every subset of $\{a,b,c\}$ of size at least two has rank two. If, instead, $r(E - \{a,b,c\}) = r(Q) - 1$ but $r(X) = r(Q)$ for all proper supersets $X$ of $E - \{a,b,c\}$, then we call $\{a,b,c\}$ a {\it triad} of $Q$. When $Q$ is $3$-connected, $\{a,b,c\}$ is a triad of $Q$ if and only if $\{a,b,c\}$ is a triangle of $Q^*$. It is straightforward to check that a triangle and a triad of $Q$ cannot have exactly one common element.  Just as for matroids, we call a sequence $x_1,x_2,\dots,x_k$ of distinct points of a $2$-polymatroid $Q$ a {\it fan} of {\it length} $k$ if $k \ge 3$ and the sets $\{x_1,x_2,x_3\}, \{x_2,x_3,x_4\},\dots,\{x_{k-2},x_{k-1},x_k\}$  are alternately triangles and triads beginning with either a triangle or a triad.



The following lemma will be helpful in proving our main result when fans arise in the argument.

\begin{lemma}
\label{fantan}
Let $M$ and $N$ be $3$-connected $2$-polymatroids where $|E(N)| \ge 4$ and $M$ is not a whirl or the cycle matroid of a wheel. Suppose $M$ has a fan $ x_1,x_2,x_3,x_4$ where $\{x_1,x_2,x_3\}$ is a triangle and $M/x_2$ has a c-minor isomorphic to $N$. Then  $M$ has a point $z$ such that either $M\ba z$ or $M/z$ is $3$-connected having a c-minor isomorphic to $N$, or both $M\ba z$ and $M/z$ have c-minors isomorphic to $N$.
\end{lemma}

\begin{proof} 
Assume that the lemma fails. Extend $x_1,x_2,x_3,x_4$ to a maximal fan $x_1,x_2,\dots,x_n$. 
Since $M/x_2$ has a c-minor isomorphic to $N$ and has $x_1$ and $x_3$ as a parallel pair of points, it follows that each of $M/x_2\ba x_1$ and $M/x_2\ba x_3$ has a c-minor isomorphic to $N$. Thus each of $M\ba x_1$ and $M\ba x_3$ has a c-minor isomorphic to $N$. Hence $M/x_4$ has a c-minor isomorphic to $N$. A straightforward  induction argument establishes that $M/x_i$ has a c-minor isomorphic to $N$ for all even $i$, while $M\ba x_i$ has a c-minor isomorphic to $N$  for all odd $i$. Then $M/x_i$ is not \thc\ when $i$ is even, while $M\ba x_i$ is not \thc\ when $i$ is odd. 

Next we show that
\begin{sublemma}
\label{whereare}
$M$ has no
 triangle that contains more than one element $x_i$ with $i$ even; and $M$ has no
 triad that contains more than one element $x_i$ with $i$ odd.
\end{sublemma}

Suppose $M$ has a triangle  that contains $x_i$ and $x_j$ where $i$ and $j$ are distinct even integers. Since $M/x_i$ has $x_j$ in a parallel pair of points, $M/x_i\ba x_j$, and hence $M\ba x_j$, has a c-minor isomorphic to $N$. As $M/ x_j$ also  has a c-minor isomorphic to $N$, we have a \cn. Thus the first part of \ref{whereare} holds. A similar argument proves the second part.

Suppose $n$ is odd. Then, since neither $M\ba x_n$ nor $M\ba x_{n-2}$ is \thc, by \cite[Lemma 4.2]{oswww}, $M$ has a triad $T^*$ containing $x_n$ and exactly one of $x_{n-1}$ and $x_{n-2}$. By \ref{whereare}, $T^*$ contains $x_{n-1}$. Then, by the maximality of the fan, the third element of $T^*$ lies in $\{x_1,x_2,\dots,x_{n-2}\}$. But, as  each of the  points in the last set is in a triangle that is contained in that set, we obtain the \cn\ that 
$M$ has a triangle having a single element in common with the triad $T^*$.

We may now assume that $n$ is even.  As neither $M/x_n$ nor $M/x_{n-2}$ is \thc, by \cite[Lemma 4.2]{oswww}, $M$ has a triangle $T$ that contains $x_n$ and exactly one of $x_{n-1}$ and $x_{n-2}$. By \ref{whereare}, $x_{n-1} \in T$.  The maximality of the fan again implies that the  third element of $T$ is in $\{x_1,x_2,x_3,\dots,x_{n-2}\}$. As every element of the last set, except $x_1$, is in a triad that is contained in the set, to avoid having $T$ meet  such a triad in a single element, we must have that   $T = \{x_n,x_{n-1},x_1\}$.  
If $n= 4$, then $M|\{x_1,x_2,x_3,x_4\} \cong U_{2,4}$ so $\{x_2,x_3,x_4\}$ is a triangle; a \cn\ to \ref{whereare}.  We deduce that $n> 4$. Now neither $M\ba x_1$ nor $M\ba x_{n-1}$ is \thc. Thus, by \cite[Lemma 4.2]{oswww}, $M$ has a triad $T^*_2$ containing $x_1$ and exactly one of  $x_n$ and $x_{n-1}$. By \ref{whereare}, $x_n \in T^*_2$. The triangles $\{x_1,x_2,x_3\}$ and $\{x_3,x_4,x_5\}$ imply that $x_2 \in T^*_2$. Let $X = \{x_1,x_2,\dots, x_n \}$. Then, using the triangles   we know, including  $\{x_n,x_{n-1},x_1\}$, we deduce that $r(X) \le \tfrac{n}{2}$. Similarly,  the triads in $M$, which are triangles in $M^*$, imply that $r^*(X) \le \tfrac{n}{2}$. Thus, by Lemma~\ref{rr*}, $\lambda(X) = 0$. Hence $X = E(M)$. As every element of $M$ is a point, $M$ is a matroid. Since every point of $M$ is in both a triangle and a triad, by Tutte's Wheels-and-Whirls-Theorem \cite{wtt}, we obtain the \cn\ that $M$ is a whirl or the cycle matroid of a wheel.
\end{proof}

\begin{lemma}
\label{hath}
Let $M$ be a $3$-connected $2$-polymatroid having $a$ and $\ell$ as distinct elements. Then $(E(M) - \{a,\ell\},\{\ell\})$ is not a $2$-separation of $M/a$.
\end{lemma}

\begin{proof}
Assume the contrary. Then $\ell$ is a line in $M/a$, so $\sqcap(a,\ell) = 0$. We have 
$$r_{M/a}(E(M) - \{a,\ell\}) + r_{M/ a}(\ell) =  r(M/a) + 1.$$
As $\sqcap(a,\ell) = 0$, it follows that $(E(M) - \ell,\{\ell\})$ is  a $2$-separation of $M$; a \cn.
\end{proof}

\section{Parallel connection and $2$-sum}
\label{pc2s}

In this section, we follow Mat\'{u}\v{s}~\cite{fm} and Hall~\cite{hall} in defining the  parallel connection and $2$-sum of polymatroids. 
For a positive integer $k$, let $M_1$ and $M_2$ be $k$-polymatroids $(E_1,r_1)$ and $(E_2,r_2)$. Suppose first that $E_1 \cap E_2 = \emptyset$. 
The {\it direct sum} $M_1 \oplus M_2$ of $M_1$ and $M_2$ is the $k$-polymatroid $(E_1 \cup E_2,r)$ where, for all subsets $A$ of $E_1 \cup E_2$, we have $r(A) = r(A\cap E_1) + r(A \cap E_2)$. The following result is easily checked.

\begin{lemma}
\label{dualsum}
For $k$-polymatroids  $M_1$ and $M_2$ on disjoint sets,
$$(M_1 \oplus M_2)^* = M_1^* \oplus M_2^*.$$
\end{lemma}

Clearly a $2$-polymatroid is $2$-connected if and only if it cannot be written as the direct sum of two non-empty $2$-polymatroids. 
Now suppose that  $E_1 \cap E_2 = \{p\}$ and $r_1(\{p\}) = r_2(\{p\})$. Let $P(M_1,M_2)$ be $(E_1 \cup E_2, r)$ where $r$ is defined for all subsets $A$ of $E_1 \cup E_2$ by 
$$r(A) = \min\{r_1(A \cap E_1) + r_2(A\cap E_2), r_1((A \cap E_1)\cup p) + r_2((A \cap E_2)\cup p) - r_1(\{p\})\}.$$
As Hall notes, it is routine to check that $P(M_1,M_2)$ is a $k$-polymatroid. We call it the {\it parallel connection} of $M_1$ and $M_2$ with respect to the {\it basepoint} $p$. When $M_1$ and $M_2$ are both matroids, this definition coincides with the usual definition of the parallel connection of matroids. Hall extends the definition of parallel connection to deal with the case when $r_1(\{p\}) \neq r_2(\{p\})$ but we shall not do that here. 

Now suppose that $M_1$ and $M_2$ are $2$-polymatroids having at least two elements, that $E(M_1) \cap E(M_2) = \{p\}$, 
that neither $\lambda_{M_1}(\{p\})$ nor $\lambda_{M_2}(\{p\})$ is $0$, and that 
$r_1(\{p\}) = r_2(\{p\}) = 1$. We define the {\it $2$-sum}, $M_1 \oplus_2 M_2$, of $M_1$ and $M_2$ to be $P(M_1,M_2)\ba p$. We remark that this extends Hall's definition  since, to ensure that  $M_1 \oplus_2 M_2$ has more elements than each of $M_1$ and $M_2$, he requires that they each  have at least three elements. He imposes the same requirement in his   Proposition 3.6. The next result is that result with this restriction omitted. Hall's proof \cite{hall} remains valid.

\begin{proposition} 
\label{dennis3.6}
Let $M$ be a $2$-polymatroid $(E,r)$ having a partition $(X_1,X_2)$ such that $r(X_1) + r(X_2) = r(E) + 1$. Then there are $2$-polymatroids $M_1$ and $M_2$ with ground sets $X_1 \cup p$ and $X_2 \cup p$, where $p$ is a new element not in $E$, such that 
$M = P(M_1,M_2)\ba p$. In particular, for all $A \subseteq X_1 \cup p$,
\begin{equation*}
r_1(A) = 
\begin{cases} 
r(A), & \text{if   $p \not\in A$;}\\
r((A-p) \cup X_2) - r(X_2) + 1, & \text{if $p \in A$.}
\end{cases}
\end{equation*}
\end{proposition}



\begin{lemma} 
\label{dennisplus}
Let $(X,Y)$ be a partition of the ground set of a  $2$-polymatroid $M$ such that $\lambda(X) = 1$. Then, for some element $p$ not in $E(M)$, there are $2$-polymatroids $M_X$ and $M_Y$ on $X \cup p$ and $Y \cup p$, respectively, such that $M = M_X \oplus_2 M_Y$. Moreover, for $y \in Y$, 
\begin{itemize}
\item[(i)] $\lambda_{M_Y \ba y}(\{p\}) = \sqcap(X,Y-y)$;
\item[(ii)] $\lambda_{M_Y/y}(\{p\}) + \sqcap(X,\{y\}) = \lambda(X) = 1$; 
\item[(iii)] if $\sqcap(X,Y-y) = 1$, then $M\ba y = M_X \oplus_2 (M_Y\ba y)$;
\item[(iv)] if $\sqcap(X,\{y\}) = 0$, then $M/ y = M_X \oplus_2 (M_Y/ y)$; and 
\item[(v)] if $r(\{y\}) \le 1$, then 
\begin{equation*}
M\da y = 
\begin{cases} 
(M_X/p) \oplus (M_Y\ba y /p), & \text{if  $\sqcap(X,\{y\}) = 1$;}\\
 M_X \oplus_2 (M_Y\da y), & \text{if $\sqcap(X,\{y\}) = 0$.}  
\end{cases}
\end{equation*}
\item[(vi)] if $y$ is a line, then  
\begin{equation*}
M\da y = 
\begin{cases} 
(M_X\ba p) \oplus (M_Y\da y \ba p), & \text{if  $r(Y) = r(Y - y) +2$;}\\
M_X \oplus_2 (M_Y \da y)  & \text{if $r(Y) \le r(Y - y) +1$.}  
\end{cases}
\end{equation*}
In particular,  $M\da y = M_X \oplus_2 (M_Y \da y)$ when $\sqcap_{M\da y}(X,Y-y) = 1.$
\end{itemize}
\end{lemma}

\begin{proof} The existence of $M_X$ and $M_Y$ such that $M = P(M_X,M_Y)\ba p$ is an immediate consequence of Proposition~\ref{dennis3.6}. 
To see that $P(M_X,M_Y)\ba p = M_X \oplus_2 M_Y$, one needs only to  check that 
$r_{M_X}(\{p\}) = 1 = r_{M_Y}(\{p\})$ and $\lambda_{M_X}(\{p\}) = 1 = \lambda_{M_Y}(\{p\})$.

The proof of (i) follows by a  straightforward application of the rank formula in Proposition~\ref{dennis3.6}. We omit the details. To see that (ii) holds,   note that 
\begin{align*}
\lambda_{M_Y/y}(\{p\}) & = r_{M_Y}(\{p,y\}) - r(\{y\}) + r_{M_Y}(Y) - r_{M_Y}(Y\cup p)\\
& = r(y \cup X) - r(X) + 1  - r(\{y\}) + r(Y) - r(X \cup Y) + r(X) - 1\\
& = r(y \cup X)    - r(\{y\}) + r(Y) - r(X \cup Y)  \\
& = r(X) - \sqcap(X,\{y\}) + r(Y) - r(X \cup Y)\\
& = \lambda_M(X) - \sqcap(X,\{y\}).
\end{align*}

By Hall~\cite[Proposition 3.1]{hall}, $M\ba y = P(M_X,M_Y\ba y)\ba p$. If $\sqcap(X,Y-y) = 1$, then, by (i), $\lambda_{M_Y \ba y}(\{p\}) = 1$. 
Hence, by Hall~\cite[Proposition 3.1]{hall}, $M\ba y = M_X \oplus_2(M_Y \ba y)$; that is, (iii) holds. 

To prove (iv), assume that $\sqcap(X,\{y\}) = 0$. We could again follow Hall~\cite[Proposition 3.1]{hall} to get that $M/y = P(M_X, M_Y/y)\ba p$. But since he omits a full proof of this fact, we include it for completeness.

By Proposition~\ref{dennis3.6}, $M/y = P(M_1,M_2)\ba p$ for some $M_1$ and $M_2$. 
For $A \subseteq X \cup p$, 
\begin{align*}
r_{M_1}(A) &= 
\begin{cases} 
r_{M/y}(A), & \text{if   $p \not\in A$;}\\
r_{M/y}((A-p) \cup (Y-y)) - r_{M/y}(Y-y) + 1, & \text{if $p \in A$;}
\end{cases}\\
&= 
\begin{cases} 
r(A\cup y) - r(\{y\}), & \text{if   $p \not\in A$;}\\
r((A-p) \cup Y)   - r(Y)   + 1, & \text{if $p \in A$;}
\end{cases}\\
& = r_{M_X}(A).
\end{align*}
Thus $M_1 = M_X$.

Now, for $A \subseteq (Y-y) \cup p$, 
\begin{align*}
r_{M_2}(A) &= 
\begin{cases} 
r_{M/y}(A), & \text{if   $p \not\in A$;}\\
r_{M/y}((A-p) \cup X) - r_{M/y}(X) + 1, & \text{if $p \in A$;}
\end{cases}\\
&= 
\begin{cases} 
r(A\cup y) - r(\{y\}), & \text{if   $p \not\in A$;}\\
r((A-p) \cup X \cup y) - r(\{y\}) - r(X \cup y) + r(\{y\}) + 1, & \text{if $p \in A$;}
\end{cases}\\
& = 
\begin{cases} 
r(A\cup y) - r(\{y\}), & \text{if   $p \not\in A$;}\\
r((A-p) \cup X \cup y) - r(\{y\}) - r(X)  + 1, & \text{if $p \in A$.}
\end{cases}
\end{align*}
But 
\begin{align*}
r_{M_Y/y}(A) &= r_{M_Y}(A \cup y) - r_{M_Y}(\{y\})\\
&= 
\begin{cases} 
r(A\cup y) - r(\{y\}), & \text{if   $p \not\in A$;}\\
r((A-p) \cup X \cup y) - r(\{y\}) - r(X)  + 1, & \text{if $p \in A$;}
\end{cases}\\
&= r_{M_2}(A).
\end{align*}
Thus $M_2 = M_Y/y$, so $M/y = P(M_X,M_Y/y)\ba p$. As $\sqcap(X,\{y\}) = 0$, we see, by (ii), that $\lambda_{M_Y/y}(\{p\}) = 1$. Hence $M/y = M_X \oplus_2 (M_Y/y)$; that is, (iv) holds.

For (v), since $r(\{y\}) \le 1$, we have  $M\da y = M/y$. If $\sqcap(X,\{y\}) = 1$, then $y$ is parallel to $p$ in $M_Y$, so, by \cite[Proposition 3.1]{hall},  
  $M\da y = (M_X/p) \oplus (M_Y/p)$. If $\sqcap(X,\{y\}) = 0$, then, as $M_Y \da y = M_Y/y$, it follows by (iv) that 
$$M\da y = M/y =  M_X \oplus_2 (M_Y/y) = M_X \oplus_2 (M_Y \da y).$$

To prove (vi), suppose first that $r(Y) = r(Y- y) +2$. We have 
$$r_{M_Y}(\{y,p\}) = r(y \cup X) - r(X) + 1 = 3 - \sqcap(X, \{y\}).$$ 
Assume $\sqcap(X, \{y\}) = 0$. Then $M_Y$ is the 2-sum, with basepoint $q$, say, of two 2-polymatroids, one of which has ground set $\{q,y,p\}$ and  consists of two points and the line $y$ freely placed in the plane. Clearly, $M\da y = (M_X \ba p) \oplus (M_Y \ba y \ba p)$.    Now assume that     $\sqcap(X, \{y\}) = 1$. Then $M_Y$ is the direct sum of  two 2-polymatroids, one of which has rank $2$  and  consists of  the line $y$ with the point $p$ on it. Once again, we see that $M\da y = (M_X \ba p) \oplus (M_Y \ba y \ba p)$.

We may now assume that $r(Y) \le r(Y-y) + 1$. Hence $r_{M\da y}(Y-y) = r(Y) - 1$.  Clearly $r(X \cup y) > r(X)$. Thus 
\begin{equation}
\label{numbth}
\sqcap_{M\da y}(X,Y-y) = 1.
\end{equation} 
By Proposition~\ref{dennis3.6},  $M\da y = P(M_1,M_2)\ba p$ for some $2$-polymatroids $M_1$ and $M_2$ with ground sets $X \cup p$ and $(Y-y) \cup p$, respectively. We shall show that $M_1 = M_X$ and $M_2 = M_Y\da y$.  

First observe that, for $A \subseteq X$, we have 

\begin{equation*}
r_{M_1}(A) = 
\begin{cases} 
r_{M\da y}(A), & \text{if   $p \not\in A$;}\\
r_{M\da y}((A-p) \cup (Y- y)) - r_{M\da y}(Y-y) + 1, & \text{if $p \in A$.}
\end{cases}
\end{equation*}
Since $r(X \cup y) > r(X)$, we see that if $p \not\in A$, then $r_{M_1}(A) = r_{M\da y}(A) = r_M(A) = r_{M_X}(A)$.

Now suppose $p \in A$. Assume $r((A-p) \cup (Y-y)) = r((A-p) \cup Y)$. Then 
$$r_{M\da y}((A-p) \cup (Y-y)) = r((A-p) \cup (Y-y)) - 1 = r((A-p) \cup Y) - 1.$$
Moreover, $r_{M\da y}(Y-y) = r(Y) -1.$ Hence 
$$r_{M_1}(A) = r_M((A-p) \cup Y) - r_M(Y) + 1 = r_{M_X}(A).$$

To show that $M_1 = M_X$, it remains to consider when $p \in A$ and $r((A-p) \cup (Y-y)) < r((A-p) \cup Y)$. Then, as $r(Y-y) \ge r(Y) - 1$, we deduce that 
$r((A-p) \cup (Y-y)) = r((A-p) \cup Y) - 1$, so $r(Y -y) = r(Y) - 1$. Thus we have 
\begin{eqnarray*}
r_{M_1}(A) & = & r_{M\da y}((A-p) \cup (Y-y)) - r_{M\da y}(Y-y) + 1\\
		& = & r((A-p) \cup (Y-y)) - r(Y-y) + 1\\
		& = & r((A-p) \cup Y) - 1 - r(Y) + 1 + 1\\
		& = & r_{M_X}(A).
\end{eqnarray*} 
We conclude that $M_1 = M_X$.

To show that $M_2 = M_Y\da y$, suppose that $A \subseteq (Y-y) \cup p$. Now 
\begin{equation*}
r_{M_2}(A) = 
\begin{cases} 
r_{M\da y}(A), & \text{if   $p \not\in A$;}\\
r_{M\da y}((A-p) \cup X) - r_{M\da y}(X) + 1, & \text{if $p \in A$.}
\end{cases}
\end{equation*}
Suppose $p \not\in A$. Then 
\begin{align*}
r_{M_2}(A) &= 
\begin{cases} 
r(A), & \text{if  $r(A \cup y) > r(A)$;}\\
r(A) - 1, & \text{otherwise;}
\end{cases}\\
&=r_{M_Y\da y}(A).
\end{align*}

Now assume that $p \in A$. Then $r_{M\da y}(X) = r(X)$. Thus 
\begin{equation*}
r_{M_2}(A) = 
\begin{cases} 
r((A- p) \cup X) - r(X) + 1, & \text{if  $r((A - p) \cup X \cup y) > r((A-p) \cup X)$;}\\
r((A- p) \cup X) - 1 - r(X) + 1, & \text{otherwise.}
\end{cases}
\end{equation*}
Moreover, 
\begin{equation*}
r_{M_Y\da y}(A) = 
\begin{cases} 
r_{M_Y}(A), & \text{if  $r_{M_Y}(A\cup y) > r_{M_Y}(A)$;}\\
r_{M_Y}(A) - 1, & \text{otherwise.}
\end{cases}
\end{equation*}

Now $r_{M_Y}(A) = r((A-p) \cup X) - r(X) + 1$. Thus 
\begin{align*}
r_{M_Y}(A\cup y) - r_{M_Y}(A) &= r((A-p) \cup y \cup X) - r(X) + 1 - r((A-p) \cup X) \\
& \hspace*{2.5in}+r(X) -1\\
					& = r((A-p) \cup y \cup X)   - r((A-p) \cup X).
\end{align*}					
We conclude that, when $p \in A$, we have $r_{M_Y\da y}(A) = r_{M_2}(A)$. Thus $M_Y\da y = M_2.$ Hence
$M\da y = P(M_X, M_Y\da y) \ba p$. Using (\ref{numbth}), it is straightforward to show that $\lambda_{M_Y\da y}(\{p\}) = 1$. It follows that 
$M\da y = M_X \oplus_2 (M_Y \da y)$.
\end{proof}

The following was shown by Hall~\cite[Corollary 3.5]{hall}. 

\begin{proposition}
\label{connconn} 
Let $M_1$ and $M_2$ be $2$-polymatroids $(E_1,r_1)$ and $(E_2,r_2)$ where $E_1 \cap E_2 = \{p\}$ and $r_1(\{p\}) = r_2(\{p\}) = 1$ and each of $M_1$ and $M_2$ has at least two elements. Then the following are equivalent.
\begin{itemize}
\item[(i)] $M_1$ and $M_2$ are both $2$-connected; 
\item[(ii)] $M_1 \oplus_2 M_2$ is $2$-connected; 
\item[(iii)] $P(M_1,M_2)$ is $2$-connected. 
\end{itemize}
\end{proposition}

One situation that will often occur will be when we have a certain $3$-connected $2$-polymatroid $N$ arising as a c-minor of a $2$-polymatroid $M$ that has a $2$-separation. Recall that a   special $N$-minor  of $M$ is a c-minor of $M$ that either equals $N$ or differs from $N$ by having a single point relabelled.

\begin{lemma}
\label{p49}
Let $M$ be a $2$-polymatroid that can be written as the $2$-sum $M_X \oplus_2 M_Y$ of $2$-polymatroids $M_X$ and $M_Y$ with ground sets $X \cup p$ and $Y \cup p$, respectively. Let $N$ be a $3$-connected $2$-polymatroid with $|E(N)| \ge 4$ and $E(N) \subseteq E(M)$. If $M_X$ has a special $N$-minor, then $M$ has a special $N$-minor. 
\end{lemma}

\begin{proof} Since $M_X\ba p = M\ba Y$ and $M_X/p = M/Y$, we may assume that the special $N$-minor of $M_X$ uses $p$. Hence every other element of the special $N$-minor of $M_X$ is in $E(N)$. For $y$ in $Y$, we will denote by $M_X(y)$ the $2$-polymatroid that is obtained from $M_X$ by relabelling $p$ by $y$. We argue by induction on $|Y|$. 

Suppose $|Y| = 1$ and let $y$ be the element of $Y$. If $y$ is a point,  then the result is immediate since $M = M_X(y)$. If $y$ is a line, then compactifying this line gives $M_X(y)$ and again the result holds.

Now suppose that $|Y| > 1$ and choose $y$ in $Y$. Suppose $\sqcap(\{y\},X) = 1$, which is certainly true if $|Y| = 1$. Then $M|(X \cup y) = M_X(y)$ if $y$ is a point. If $y$ is a line, then compactifying $y$ in $M|(X \cup y)$ gives $M_X(y)$. In each case, the result holds. We may now assume that $|Y| > 1$ and $\sqcap(\{y\},X) = 0$. Then, by 
Lemma~\ref{dennisplus}(iv), $M/y = M_X \oplus_2 (M_Y/y)$ so the result follows by induction. 
\end{proof}

\begin{lemma}
\label{p69}
Let $M$ be a $2$-polymatroid that can be written as the $2$-sum $M_X \oplus_2 M_Y$ of $2$-polymatroids $M_X$ and $M_Y$ with ground sets $X \cup p$ and $Y \cup p$, respectively. Let $N$ be a $3$-connected $2$-polymatroid with $|E(N)| \ge 4$ such that $N$ is a c-minor of $M$.  If $|E(N) \cap X| \ge |E(N)| - 1$, then $M_X$ has a special $N$-minor that uses $E(N) \cap X$. 
\end{lemma}

\begin{proof}
As $N$ is a c-minor of $M$, it follows by Corollary~\ref{complast3} that $N$ can be obtained from $M$ by a sequence of deletions and contractions followed by one compactification at the end. Let $N_1$ be the $2$-polymatroid that is obtained prior to the last compactification. 
We know that we can shuffle these deletions and contractions at will. In producing $N_1$ from $M$, let $C_Y$ and $D_Y$ be the sets of elements of $Y$ that are contracted and deleted, respectively.

Suppose $\sqcap(X,C_Y) = 1$. Now $M = P(M_X,M_Y) \ba p$. Consider $P(M_X,M_Y) /C_Y\ba D_Y$. This has $p$ as a loop, so 
$P(M_X,M_Y) \ba p /C_Y\ba D_Y = P(M_X,M_Y)/ p /C_Y\ba D_Y$. Since $P(M_X,M_Y)/ p= (M_X/p) \oplus (M_Y/p)$, we deduce that $Y = D_Y \cup C_Y$, so $N_1$ is a c-minor of $M_X/p$. Thus $N$ is a c-minor of $(M_X)^{\flat}$ and hence of $M_X$.

We may now assume that $\sqcap(X,C_Y) = 0$. Suppose $Y \cap E(N) = \emptyset$. Then $M\ba D_Y/C_Y = M\ba Y = M_X\ba p$. Hence $N_1$ is a c-minor of $M_X$. As we can perform a compactification whenever we want, we deduce that $N$ is a c-minor of $(M_X)^{\flat}$. It remains to consider the case when $Y \cap E(N)$ consists of a single element, $y$. In $M/C_Y\ba D_Y$, we must have $\sqcap(X,\{y\}) = 1$, otherwise $\sqcap(X,\{y\}) = 0$ and $\{y\}$ is 1-separating in $N_1$ and hence in $N$; a contradiction. We deduce that, in $M_Y/C_Y\ba D_Y$, the element $y$ is either a point parallel to the basepoint $p$ or a line through $p$. In the latter case, $(M/C_Y\ba D_Y)^{\flat}$ is $(M_X(y))^{\flat}$ where $M_X(y)$ is obtained from $M_X$ by relabelling $p$ by $y$. In both cases, $(M_X(y))^{\flat}$ has $N$ as a c-minor so $(M_X)^{\flat}$ and hence $M_X$ has a special $N$-minor.
\end{proof}

\begin{lemma}
\label{useful}
Let $p$ be a point in a $2$-polymatroid $P$ having ground set $E$. If $\sqcap(p,E-p) = 1$, then $P$ has as a minor a $2$-element $2$-connected $2$-polymatroid using $p$.
\end{lemma}

\begin{proof}
We argue by induction on $|E-p|$. the result is certainly true if $|E-p| = 1$. Assume it true for $|E-p| < n$ and let $|E-p| = n$. If $E-p$ contains an element $z$ such that $\sqcap(p,z) = 1$, then the result is immediate. Thus $E- p$ contains an element $z$ such that $\sqcap (p,z) = 0$. Then 
$\sqcap_{P/z}(p,E - \{p,z\}) = r(p) + r(E-p) - r(P) = 1.$ Thus, by the induction assumption, $P/z$ and hence $P$ has, as a minor, a $2$-element $2$-connected $2$-polymatroid using $p$.
\end{proof}


\begin{lemma}
\label{claim1} 
Let $(X,Y)$ be an exact $2$-separation of a $2$-polymatroid $M$ and let $N$ be a $3$-connected $2$-polymatroid that is a c-minor of $M$. Suppose that $|X - E(N)| \le 1$ and  
 $y \in Y$.
\begin{itemize}
\item[(i)] If $\sqcap_{M\ba y}(X,Y-y) = 1$, then $M\ba y$ has a special $N$-minor.
\item[(ii)] If $\sqcap_{M/ y}(X,Y-y) = 1$, then $M/y$ has a special $N$-minor.
\item[(iii)] If $\sqcap_{M\da y}(X,Y-y) = 1$, then $M\da y$ has a special $N$-minor.
\end{itemize}
\end{lemma}

\begin{proof} 
By Lemma~\ref{dennisplus}, $M = M_X \oplus_2 M_Y$ where $M_X$ and $M_Y$have ground sets $X \cup p$ and $Y\cup p$, respectively. 
By Lemma~\ref{p69}, $M_X$ has a special $N$-minor using $E(N) \cap X$. Suppose $\sqcap_{M\ba y}(X,Y-y) = 1$. Then $\sqcap_{M_Y}(\{p\},Y-y) = 1$. Thus, by Lemma~\ref{useful}, $M_Y\ba y$ has as a minor a $2$-polymatroid with ground set $\{p,z\}$ for some $z$ in $Y-y$ where either $p$ and $z$ are parallel points, or $z$ is a line and $p$ is a point on this line. It follows that $(M \ba y)^{\flat}$ has as a c-minor the $2$-polymatroid that is obtained from $(M_X)^{\flat}$ by relabelling $p$ by $z$. Hence $M\ba y$ has a special $N$-minor and (i) holds. 

Now suppose that $\sqcap_{M/ y}(X,Y-y) = 1$. Then, by Lemma~\ref{obs1}, $\sqcap(X,\{y\}) = 0$. Thus, by Lemma~\ref{dennisplus}(iv), 
$M/y = M_X \oplus_2 (M_Y/y)$. Then, by replacing $M_Y\ba y$ by $M_Y/ y$ in the argument in  the previous paragraph, we deduce that (ii) holds.

Finally, suppose that $\sqcap_{M\da y}(X,Y-y) = 1$. Assume first that $r(\{y\}) \le 1$. Then $M\da y = M/y$, so $\sqcap_{M/ y}(X,Y-y) = 1$, and the result follows by (ii). Now let $y$ be a line of $M$. Then, by Lemma~\ref{dennisplus}(vi), $M\da y = M_X \oplus_2 (M_Y \da y)$. Again, by replacing $M_Y \ba y$ by $M_Y \da y$ in the argument in the first paragraph, we get that (iii) holds.
\end{proof}

\begin{lemma}
\label{switch}
Let $Q$ be a $2$-polymatroid having $k$ and $\ell$ as distinct elements and suppose that $\ell$ is a $2$-separating line. 
Then
$$Q \ud \ell \da k = Q \da k \ud \ell.$$
\end{lemma}

\begin{proof}
The result is easily checked if  $\lambda(\ell) = 0$, so assume that   $\lambda(\ell) = 1$. Then, by Lemma~\ref{dennisplus}, $Q = P(Q_1,Q_2)\ba p$ for some $2$-polymatroids $Q_1$ and $Q_2$ with ground sets $(E(Q) - \ell) \cup p$ and $\{\ell, p\}$ where $Q_2$ consists of the line $\ell$ with the point freely placed on it. Moreover, either 
\begin{itemize}
\item[(i)] $k$ is a point that is parallel to $p$ in $Q_1$; or 
\item[(ii)] $Q \da k = P(Q_1 \da k,Q_2)\ba p$. 
\end{itemize}

Consider the first case. Then $Q\da k = Q/k$ and $Q\da k \ud \ell$ can be obtained from $Q_1/p$ by adjoining $\ell$ as a loop. On the other hand, $Q \ud \ell$ can be obtained from $Q_1$ by relabelling $p$ as $\ell$. Thus $Q \ud \ell \da k$, which equals $Q \ud \ell/k$, can be obtained from $Q_1/p$ by adjoining $\ell$ as a loop. Hence the result holds in case (i).

Now suppose (ii) holds. Then $Q \da k \ud \ell$ can be obtained from $Q_1 \da k$ by relabelling $p$ as $\ell$. 
On the other hand, $Q\ud \ell$ can be obtained from  $Q_1$ by relabelling $p$ as $\ell$. Hence $Q\ud \ell \da k$ can be obtained from $Q_1\da k$ by relabelling $p$ as $\ell$. Thus the lemma holds.
\end{proof}

We end this section with three lemmas concerning $2$-element prickly $3$-separators.

\begin{lemma}
\label{portia}
Let $\{j,k\}$ be a prickly $3$-separator in a $3$-connected $2$-polymatroid $M$. Then $M\da j$ and $M\da k$ are $3$-connected.
\end{lemma}

\begin{proof} It suffices to show that $M\da j$ is $3$-connected. We form $M\da j$ by freely adding a point $j'$ to $j$, deleting $j$, and contracting $j'$. As $M$ is $3$-connected, so is the $2$-polymatroid $M'$ we get by adding $j'$.  Now $M\da j = M'\ba j/j'$. Assume this $2$-polymatroid is not $3$-connected, letting $(U,V)$ be an $m$-separation of it for some $m$ in $\{1,2\}$. Then 
$$r_{M'/j'}(U) + r_{M'/j'}(V) = r(M'/j') + m-1.$$ 
Thus 
$$r_{M'}(U\cup j') + r_{M'}(V\cup j') = r(M') + m.$$
Without loss of generality, $k \in V$. Then $r_{M'}(V\cup j') = r_M(V \cup j)$ and $r_{M'}(U\cup j') = r_M(U)+ 1$. Therefore 
$$r_{M}(U) + r_{M}(V\cup j) = r(M) + m- 1.$$
As $M$ is $3$-connected, we deduce that $m = 2$. Then $\max\{|U|, r_{M'/j'}(U)\} \ge 2$. Hence $(U, V \cup j)$ is a $2$-separation of $M$; a contradiction.
\end{proof}

\begin{lemma}
\label{pricklytime}
The set $\{j,k\}$ is a prickly $3$-separator of the $2$-polymatroid $M$ if and only if it is a prickly $3$-separator in $M^*$.
\end{lemma}

\begin{proof} Suppose $\{j,k\}$ is a prickly $3$-separator of $M$. By Lemma~\ref{compact0}, $\lambda_{M^*}(\{j,k\}) = \lambda_{M}(\{j,k\}) = 2$. Moreover, it is straightforward to check that 
$r_{M^*}(\{j\}) = 2 = r_{M^*}(\{k\})$, that $r_{M^*}(\{j,k\}) = 3$, and that $\sqcap_{M^*}(\{j\},E-\{j,k\}) = 1 = \sqcap_{M^*}(k,E-\{j,k\})$. Hence $\{j,k\}$ is a prickly $3$-separator of $M^*$. Conversely, suppose that $\{j,k\}$ is a prickly $3$-separator of $M^*$. Then, by what we have just shown,  $\{j,k\}$ is a prickly $3$-separator of $(M^*)^*$, that is, of $M^{\flat}$. Now $2 = \lambda_{M^{\flat}}(\{j,k\}) = \lambda_{M}(\{j,k\})$. Moreover, since $r_{M^{\flat}}(\{j\}) = 2$, it follows that $\lambda(\{j\}) = 2$, so $r(\{j\}) = 2$ and $r(E-j) = r(E)$. Similarly, $\lambda(\{k\}) = 2 = r(\{k\})$ and $r(E-k) = r(E)$. It follows, since  $r_{M^{\flat}}(\{j,k\}) = 3$, that $r(\{j,k\}) = 3$. By using the fact that $\sqcap_{M^{\flat}}(\{j\},E-\{j,k\}) = 1 = \sqcap_{M^{\flat}}(\{k\},E-\{j,k\})$, it is not difficult to check that  
$\sqcap(\{j\},E-\{j,k\}) = 1 = \sqcap(\{k\},E-\{j,k\})$. We conclude that $\{j,k\}$ is a prickly $3$-separator of $M$, so the lemma holds. 
\end{proof}

\begin{lemma}
\label{ess3}
Let $\{j,k\}$ be a prickly $3$-separator in a $2$-polymatroid $P$. Then 
\begin{itemize}
\item[(i)] $P\da k \ba j = P\ba k,j$; and
\item[(ii)] $P\da k / j = P/ k,j$.
\end{itemize}
\end{lemma}

\begin{proof} Suppose $X \subseteq E(P) - \{j,k\}$. Then 
$r_{P\da k}(X) = r_P(X)$ as $r(X \cup k) > r(X)$. Thus (i) holds. 

To see (ii), observe that $r_{P\da k /j}(X) = r_{P\da k}(X \cup j) - r(\{j\})$ since $r(\{j,k\}) > r(\{j\})$. Now, as $\sqcap(\{j\},\{k\}) = 1$, we deduce that 
$r(X \cup j) \le r(X \cup j \cup k) \le r(X \cup j) +1$. Thus
\begin{equation*}
r_{P\da k}(X \cup j) = 
\begin{cases} 
r(X \cup j), & \text{if  $r(X \cup j \cup k) = r(X \cup j) + 1$;}\\
r(X \cup j) - 1, & \text{if  $r(X \cup j \cup k) = r(X \cup j)$.}
\end{cases}
\end{equation*}
Hence 
$r_{P\da k}(X \cup j) = r(X \cup j \cup k) - 1.$
Thus 
$r_{P\da k /j}(X) = r(X \cup j \cup k) - 3 = r_{P/k,j}(X)$, so (ii) holds. 
\end{proof}

\section{The strategy of the proof}
\label{strat}

The proof of Theorem~\ref{mainone} is long and will occupy  the rest of the paper. In this section, we outline the steps in the proof. 
 We shall assume that the theorem fails for $M$. Hence $|E(M)| \ge |E(N)| + 2$. As $|E(N)| \ge 4$, we deduce that $|E(M)| \ge 6$. 
 

We know that $M$ has  $N$ as an s-minor. This means, of course, that  $N$ can be obtained from $M$ by a sequence of contractions, deletions accompanied by compactifications, and series compressions. Our first goal will be to prove the following.

\begin{lemma}
\label{endtime}
The $2$-polymatroid $M$ has an s-minor that is isomorphic to $N$ such that, in the production of this s-minor, all of the series compressions are done last in the process. 
\end{lemma}

Next we focus on the c-minor $N_0$ of $M$ that is obtained in the above process after all of the contractions and compactified deletions are done but before doing any of the series compressions. By Lemma~\ref{elemprop25}, $N_0$ is $3$-connected. In view of this, we see that, to prove Theorem~\ref{mainone}, it suffices to prove Theorem~\ref{modc0}, which we restate here for the reader's convenience.

\begin{theorem}
\label{modc}
Let $M$ and $N$ be distinct $3$-connected $2$-polymatroids such that $N$ is a c-minor of $M$ and $|E(N)| \ge 4$. Then 
\begin{itemize}
\item[(i)] $r(M) \ge 3$ and $M$ is a whirl or the cycle matroid of a wheel; or 
\item[(ii)] $M$ has an element $\ell$ such that $M\baba \ell$ or $M/\ell$ is $3$-connected having a c-minor isomorphic to $N$; or 
\item[(iii)] $M$ has a prickly $3$-separator $\{y,z\}$ such that $M\da y$ is $3$-connected having a c-minor isomorphic to $N$.
\end{itemize}
\end{theorem}

Our focus then becomes proving Theorem~\ref{modc}. For the rest of this section, we assume that the pair $(M,N)$ a counterexample to that theorem. The first two steps in the argument, whose proofs appear in Section~\ref{edlp},  are as follows.

\begin{lemma} 
\label{Step0}
$M$ has no point $z$ such that both $M\baba z$ and $M/z$ have c-minors isomorphic to $N$.
\end{lemma}

\begin{lemma} 
\label{Step1}
$M$ has no element $\ell$ such that   $M\ba \ell$ or $M/ \ell$ is disconnected having a c-minor isomorphic to $N$.
\end{lemma}

Note that the use of $\ell$ above, and in what follows, does not imply that $\ell$ is a line, although most of our attention will be focused on that case.

Now $N$ occurs as a c-minor of $M$. Although we will often work with c-minors of $M$ that are isomorphic to $N$, at a certain point in the argument, we will settle on a particular labelled c-minor of $M$ that is isomorphic to $N$.

When  $M$ has $N$ as a c-minor and has a 2-separation $(X,Y)$, either $X$ or $Y$, say $X$, contains at least $|E(N)| - 1$ elements of $N$. We call $X$ the {\it $N$-side} of the 2-separation and $Y$ the {\it non-$N$-side}.

Suppose $M\baba \ell$ has $N$ as a c-minor. Because the theorem fails, $M\baba \ell$ is not $3$-connected. Now, by Lemma~\ref{compact0}(iii), 
$\lambda_{M\baba \ell} = \lambda_{M\ba \ell}$. Thus a partition $(X,Y)$ of $E - \ell$ with $\min\{|X|,|Y|\} \ge 2$ is a 2-separation of $M\baba \ell$ if and only if it is a 2-separation of $M\ba \ell$. It follows that we can label the $N$- and non-$N$-sides of a non-trivial $2$-separation of $M\ba \ell$ based on their labels in the corresponding $2$-separation of $M\baba \ell$. Among all $2$-separations of $M\baba \ell$, let the maximum cardinality of the non-$N$-side be $\mu(\ell)$. Similarly, if $M/\ell$ has $N$ as a c-minor, let $\mu^*(\ell)$ be the maximum cardinality of the non-$N$-side of a $2$-separation of $M/\ell$. We observe that  $\mu(\ell)$ and $\mu^*(\ell)$ are not defined unless $M\baba \ell$ and $M/ \ell$, respectively, have $N$ as a c-minor.

The next step in the argument establishes the following. 


\begin{noname} 
\label{Step2.2}
$M$ has no element $\ell$ for which $\mu(\ell) = 2$ or $\mu^*(\ell)= 2$.
\end{noname}

The argument for (\ref{Step2.2}) is quite long since it involves a detailed analysis of the various structures that can arise on the non-$N$-side when $\mu(\ell) = 2$. We then use duality to eliminate the cases when $\mu^*(\ell)= 2$. These arguments appear in Section~\ref{alltwos}. 


Recall that a special $N$-minor of $M$ is any c-minor of $M$ that is either equal to $N$ or differs from $N$ by having a single point relabelled. The next major step in the argument, which is dealt with in Lemma~\ref{bubbly},  proves the following. 

\begin{noname} 
\label{Step3}
If $(X,Y)$ is a $2$-separation of $M\ba \ell$ where $X$ is the $N$-side and $|Y| = \mu(\ell)$, then $Y$ contains an element $y$ such that both $M\baba y$  and $M/ y$ have special $N$-minors.
\end{noname}

In Lemma~\ref{nonN}, we use the element found in the last step to prove the following.

\begin{noname} 
\label{Step4}
There is a c-minor $N'$ of $M$ that is isomorphic to $N$ such that $M$ has a $3$-separator $(X,Y)$ with $|E(N') \cap Y| \le 1$ such that if $|Y| = 2$, then both elements of $Y$ are lines.
\end{noname}

The particular c-minor $N'$ whose existence is proved in (\ref{Step4}) is the one used throughout the rest of the argument. From that point on in the argument, we use $N$ to denote $N'$. An exactly 3-separating set $Y$ is called a {\it non-$N$-$3$-separator} if $|E(N) \cap Y| \le 1$ and, when $|Y| = 2$, both elements of $Y$ are lines. By (\ref{Step4}), a non-$N$-$3$-separator exists. Hence there is a minimal such set. 


At the beginning of Section~\ref{bigtime}, we prove that 

\begin{noname} 
\label{Step5}
$M$ has a minimal non-$N$-$3$-separator with at least three elements. 
\end{noname}

The rest of Section~\ref{bigtime} is devoted to showing the following.

\begin{noname} 
\label{Step5.5}
A minimal non-$N$-$3$-separator of $M$ with exactly  three elements consists of three lines. 
\end{noname}

The purpose of Section~\ref{threeel} is   to prove that 

\begin{noname} 
\label{Step6}
$M$ has a minimal non-$N$-$3$-separator with at least four elements. 
\end{noname}

The argument to show (\ref{Step6}) is quite long since it involves treating all non-$N$-$3$-separators that consist of exactly three lines.

We say that an element $\ell$ of $M$ is {\it doubly labelled} if both $M\ba \ell$ and $M/\ell$ have special $N$-minors. The next step, which is shown in Section~\ref{fourel}, establishes the following.   

\begin{noname} 
\label{Step7}
If $Y_1$ is a minimal non-$N$-$3$-separator of $M$ with at least four elements, then $Y_1$ contains a doubly labelled element. 
\end{noname}

Next we take the doubly labelled element $\ell$ identified in the last step. We then take non-trivial $2$-separations $(D_1,D_2)$ and $(C_1,C_2)$ of $M\ba \ell$ and $M/ \ell$, respectively, having $D_1$ and $C_1$ as their $N$-sides. We show that these 2-separations can be chosen so that each of  $D_2$ and $C_2$  is contained in $Y_1 - \ell$, and neither contains any points of $M$. 

We then show that each of $D_1\cap C_2, D_2 \cap C_1$, and $D_2 \cap C_2$ consists of a single line of $M$, that the union of these lines spans $\ell$, and these four lines together make up $Y_1$.

The final contradiction is obtained by showing that $M/\ell_{22}$ is \thc\ having a c-minor isomorphic to $N$, where $\ell_{22}$ is the unique element in $D_2 \cap C_2$.

\section{The reduction to c-minors}
\label{redc}

\setcounter{theorem}{1}

The goal of this section is to prove Lemma~\ref{endtime} and thereby show that Theorem~\ref{mainone} can be proved by verifying Theorem~\ref{modc}.

\begin{proof}[Proof of Lemma~\ref{endtime}.]
Consider the s-minors of $M$ that are isomorphic to $N$ and are obtained using the minimum number of series compressions. Suppose $N_1$ is such an s-minor and let the number of series compressions used in its production be $m$. If $m = 0$, then $N_1$ is an s-minor of $M$ satisfying the requirements of the lemma. Hence we may assume that $m > 0$. Let $n_1$ be the number of elements that are removed after the last series compression has been completed. For $2 \le i \le m$, let $n_i$ be the number of elements that are removed via deletion or contraction between the $(m-i+1)$st and the $(m-i+2)$nd series compressions. Consider the sequence $(n_1,n_2,\dots,n_m)$ and  let $N_0$  be a choice for $N_1$ for which the corresponding sequence is lexicographically minimal. If each $n_i$ is zero, then we have found, as desired, an s-minor of $M$ in which all of the series compressions are performed after all of the contractions and compactified deletions.  Assume then that $n_i$ is the first non-zero $n_j$. Let $P$ be the $2$-polymatroid that we have immediately prior to the $(m-i+1)$st series compression, with this series compression involving compressing the line $k$ from the prickly $3$-separator $\{j,k\}$ of $P$. Let $Q$ be the 2-polymatroid we have immediately prior to the $(m - i +2)$nd series compression.

By Lemma~\ref{ess3}, we may assume that $j$ is neither deleted or contracted in producing $N_0$ 
otherwise we can replace the compression of $k$ by a deletion followed by a compactification or by a contraction. 
By Lemma~\ref{elemprop}, we may assume that either 
\begin{itemize}
\item[(a)] all of the elements removed in producing $Q$ from $P\da k$ are done so by deletion followed by compactification; or 
\item[(b)] the next move in the production of $Q$ is the contraction of an element, say $y$. 
\end{itemize}
   
Assume that (b) holds. 
By Lemma~\ref{pricklytime0}, $P\da k /y = P/y \da k.$  Assume that $\{j,k\}$ is not a prickly $3$-separator of $P/y$. We now apply Lemma~\ref{elemprop24}. If $r(\{y,j,k\}) = 3$, then $j$ is a loop of $P\da k/y$ so $j$ must be deleted or contracted to produce $N_0$; a \cn. 
 If $P\da k/y$ is $P/y\baba k$ or $P/y/k$, then we do not need to compress $k$ in the production of $N_0$, so the choice of $N_0$ is contradicted. We  
 are left with the possibility that           
 $P\da k/y$ can be obtained from $P/y\baba j$ by relabelling $k$ as $j$. Again we  obtain the \cn\  that we can reduce the number of series compressions where, if $j \in E(N_0)$, we replace $N_0$ by the 2-polymatroid in which $j$ is relabelled by $k$.  
We conclude that $\{j,k\}$ is a prickly $3$-separator of $P/y$. In that case,  interchanging the compression of $k$ and the contraction of $y$ in $P$ produces a $2$-polymatroid in which $n_i$ is reduced and so the choice of $N_0$ is contradicted. We  deduce that (b) does not hold, so (a) holds.

In the construction of $N_0$, let $y$ be the first element that is deleted following the compression of $k$. Now, by Lemma~\ref{pricklytime0}, 
$P\da k \baba y = (P\da k \ba y)^{\flat} = (P\ba y \da k)^{\flat}.$  As $P$ is compact, $r(E - y) = r(E)$. If $r(E - \{y,j,k\}) = r(E) - 3$, then $P\ba y$ has $\{j,k\}$ as a $1$-separating set. This is a contradiction as $j$ cannot be deleted or contracted in the production of $N_0$ from $P$. Hence 

\begin{equation}
\label{2down}
r(E - \{y,j,k\}) \ge r(E) - 2.
\end{equation}

Let $S$ be the set of $2$-separating lines in $P\ba y$. Clearly no member of $S - k$ is parallel to $k$. 
We show next that

\begin{sublemma}
\label{newjk} 
$S \cap \{j,k\} \neq \emptyset.$
\end{sublemma}

Suppose, instead,  that neither $j$ nor $k$ is in $S$. Then, by Lemma~\ref{atlast2}, $S$ is the set of $2$-separating lines of $P\ba y \da k$. Let $S = \{\ell_1,\ell_2,\dots,\ell_t\}$. Then 
$$P\baba y \da k = P\ba y \ud \ell_1 \ud \ell_2\ud \dots\ud \ell_t \da k.$$ 
Thus, by repeated application of Lemma~\ref{switch} and using Lemma~\ref{pricklytime0}, we see that 
\begin{align*}
P\baba y \da k & = P\ba y \da k \ud \ell_1 \ud \ell_2\ud \dots\ud \ell_t \\
			& = P\da k \ba y \ud \ell_1 \ud \ell_2\ud \dots\ud \ell_t \\
			& = P\da k \baba y.
\end{align*}
To prevent us from being able to reduce $n_i$, we must have that 

\begin{sublemma}
\label{newjk2} 
$\{j,k\}$ is not a prickly $3$-separator of $P\baba y$. 
\end{sublemma}

Continuing with the proof of \ref{newjk}, suppose $\lambda_{P\ba y}(\{j,k\}) = 1$. Then $P\ba y$ is the $2$-sum with basepoint $p$ of two $2$-polymatroids $P_1$ and $P_2$ having ground sets $(E - \{y,j,k\}) \cup p$ and $\{j,k,p\}$, respectively. Since neither $j$ nor $k$ is  2-separating in $P\ba y$, it follows that, in the rank-$3$ $2$-polymatroid $P_2$, the point $p$ does not lie on either of the lines $j$ or $k$. By Lemma~\ref{dennisplus}(iv), $P\ba y \da k = P_1 \oplus (P_2 \da k)$. Now $P_2 \da k$ consists of the line $j$ with the point $p$ lying on it. Hence $j$ is a $2$-separating line of $P\ba y \da k$, so $S \cup j$ is the set of $2$-separating lines of  $P\ba y \da k$. Since $P_2 \da k\ud j = P_2 /k$, we deduce that  $P\ba y \da k\ud j = P\ba y/k$. It follows that $S$ is the set of 2-separating lines of  $P\ba y/k$. 
Thus 
\begin{align*}
P\da k \baba y				& = 	(P\ba  y \da k)^{\flat}\\
                          & = P\ba y \da k \ud j \ud \ell_1 \ud \ell_2\ud \dots\ud \ell_t  \\
                          & =  (P\ba y \da k \ud j) \ud \ell_1 \ud \ell_2\ud \dots\ud \ell_t  \\
			& = P  \ba y /k \ud \ell_1 \ud \ell_2\ud \dots\ud \ell_t\\
			& = P/k \ba y \ud \ell_1 \ud \ell_2\ud \dots\ud \ell_t\\
			& = P/k\baba y.
\end{align*}
We conclude that, instead of compressing $k$, we can contract it, which contradicts that choice of $N_0$. We conclude that  $\lambda_{P\ba y}(\{j,k\}) = 2$. Moreover, $\sqcap(j, E-\{y,j,k\}) = 1 = \sqcap(k, E-\{y,j,k\})$ since neither $j$ nor $k$ is in $S$.  Thus $\{j,k\}$ is a prickly $3$-separator of $P\ba y$. It follows without difficulty that $\{j,k\}$ is a prickly $3$-separator of $P\baba y$; a contradiction to \ref{newjk2}. We conclude that \ref{newjk} holds. 

We now know that $j$ or $k$ is in $S$.  Suppose next that both $j$ and $k$ are in $S$. Thus $r(E - \{y,j\}) = r(E) - 1 = r(E - \{y,k\})$. By submodularity and (\ref{2down}), we deduce that $r(E - \{y,j,k\}) = r(E) - 2$. Hence $P\ba y$ is the $2$-sum with basepoint $p$ of two $2$-polymatroids $P_1$ and $P_2$ having ground sets $(E - \{y,j,k\}) \cup p$ and $\{j,k,p\}$, respectively. Moreover, in $P_2$, the point $p$ lies on both $j$ and $k$. 
Now $P\baba y = P\ba y \ud \ell_1 \ud \ell_2\ud \dots\ud \ell_t \ud j \ud k$. Hence $P\baba y$ has $j$ and $k$ as parallel points. 
Thus 
\begin{align*}
P\da k \baba y				& = 	(P\ba  y \da k)^{\flat}\\
                          & = P\ba y \da k  \ud \ell_1 \ud \ell_2\ud \dots\ud \ell_t \ud j \\
			& = P  \ba y  \ud \ell_1 \ud \ell_2\ud \dots\ud \ell_t \ud j \da k ~\text{~~~~by Lemma~\ref{switch};} \\
			& = P  \ba y \ud \ell_1 \ud \ell_2\ud \dots\ud \ell_t \ud j \ud k \ba k ~\text{~~~~as $Q\da x = Q \ud x \ba x$ when $r(\{x\}) = 2$;} \\
			& = P\baba y\ba k\\
			& =  P\baba y\baba k,
\end{align*}
where the last step follows because $P\da k \baba y$ is compact and so $P\baba y\ba k$ is compact. Again we have a contradiction since we have managed to remove $k$ via deletion rather than by series compression.

Now assume that  $k$ is in $S$ but $j$ is not. Then 
\begin{align*}
P\da k \baba y				& = 	(P\ba  y \da k)^{\flat}\\
                          & = P\ba y \ \ud \ell_1 \ud \ell_2\ud \dots\ud \ell_t \da k \\
			& = P  \ba y  \ud \ell_1 \ud \ell_2\ud \dots\ud \ell_t \ud k \ba k\\
			& = P\baba y \ba k\\
			& = P\baba y \baba k.
\end{align*}
Once again we have managed to avoid the need to perform a series compression on   $k$; a contradiction. 

Finally, suppose $j$ is in $S$ but $k$ is not. Then we use the fact that $P\da j$ is $P\da k$ with $j$ relabelled as $k$. 
The argument in the last paragraph yields a \cn\ where, when $j \in E(N_0)$, we replace $N_0$ by the 2-polymatroid in which $j$ is labelled as $k$. 
\end{proof}


\section{Eliminating doubly labelled points} 
\label{edlp}

In this section, we prove that, when $(M,N)$ is a counterexample to Theorem~\ref{modc}, $M$ has no doubly labelled point and has no element whose deletion or contraction is disconnected having a c-minor isomorphic to $N$. 

The following elementary lemmas will be helpful.

\begin{lemma}
\label{helpful}
Let $T$ be a set of three points in a  $2$-polymatroid $Q$ and suppose $x \in T$. 
\begin{itemize}
\item[(i)] If $T$ is a triangle of $Q$, then $\lambda_{Q/ x}(T-x) \le 1$. 
\item[(ii)] If $T$ is a triad of  $Q$, then $\lambda_{Q\ba x}(T-x) \le 1$. 
\end{itemize}
\end{lemma}

\begin{lemma}
\label{tryto}
Let $T_1$ and $T_2$ be distinct triads in a $2$-polymatroid $Q$. Then $r(E(Q) - (T_1 \cup T_2)) \le r(Q) - 2.$
\end{lemma}

\begin{proof}
We know that $r(E(Q) - T_i) = r(Q) - 1$ for each $i$. The lemma follows easily by applying the submodularity of the rank function.
\end{proof}

\begin{proof}[Proof of Lemma~\ref{Step0}.]
Suppose $M$ has a point $z$ such that both $M\baba z$ and $M/z$ have c-minors isomorphic to $N$. Then neither $M\ba z$ nor $M/z$ is \thc. We may also assume that $M$ is neither a whirl nor the cycle matroid of a wheel. By \cite[Lemma 4.1]{oswww}, $M$ has   points $s$ and $t$ such that $\{z,s,t\}$ is a triangle or a triad of $M$. By replacing $M$ by $M^*$ if necessary, we may assume that $\{z,s,t\}$ is a triangle of $M$. Then $M/z$ has $s$ and $t$ as a pair of parallel points. Thus both $M/z\ba s$ and $M/z\ba t$ have c-minors isomorphic to $N$. As the theorem fails, neither $M\ba s$ nor $M\ba t$ is \thc. Thus, by \cite[Lemma 4.2(i)]{oswww}, $M$ has a triad that contains $z$ and exactly one of $s$ and $t$. We may assume that
 the triad is $\{z,s,u\}$. Then $t,z,s,u$ is a fan in $M$.
 
Now take a fan $x_1,x_2,\dots,x_k$ in $M$ of maximal length such that both $M\ba x_2$ and $M/x_2$ have c-minors isomorphic to $N$. Then $k \ge 4$.  A straightforward induction argument, whose details we omit, gives the following.

\begin{sublemma}
\label{fant}
For all $i$ in $\{2,3,\dots,k-1\}$, both $M\ba x_i$ and $M/x_i$ have c-minors isomorphic to $N$.
\end{sublemma}


Now consider $\{x_{k-2},x_{k-1},x_k\}$. Suppose first that it is a triangle. As $M/x_{k-1}$ has a c-minor isomorphic to $N$, so do  $M/x_{k-1}\ba x_k$ and hence $M\ba x_k$. As $M\ba x_{k-1}$ also has a c-minor isomorphic to $N$, neither $M\ba x_k$ nor $M\ba x_{k-1}$ is $3$-connected. Thus, by \cite[Lemma 4.2]{oswww}, it follows that $M$ has a triad $T^*$ containing $x_k$ and exactly one of $x_{k-2}$ and $x_{k-1}$. Let its third element be $x_{k+1}$. By the choice of $k$, it follows that $x_{k+1} \in \{x_1,x_2,\dots,x_{k-3}\}$. Suppose $k = 4$. Then $x_1 \in T^*$. Then  $\{x_1,x_2,x_3,x_4\}$ contains two distinct triads so, by Lemma~\ref{tryto}, 
$r(E - \{x_1,x_2,x_3,x_4\}) \le r(M) - 2$. Thus $\lambda(\{x_1,x_2,x_3,x_4\}) \le 1$; a contradiction since $|E| \ge 6$. We deduce that $k \ge 5$.

As $M$ cannot have a triangle and a triad that meet in a single element, either
\begin{itemize}
\item[(i)] $x_{k+1} = x_1$ and $\{x_1,x_2,x_3\}$ is a triad; or 
\item[(ii)] $T^*$ contains $\{x_k,x_{k-2}\}$, and $x_{k+1} \in \{x_{k-3},x_{k-4}\}$.
\end{itemize}
In the latter case, let $X = \{x_{k-4},x_{k-3},x_{k-2},x_{k-1},x_k\}$. Then, by Lemma~\ref{tryto}, 
$$r(X) + r(E-X) - r(M) \le 3 + r(M) - 2 - r(M) = 1.$$
Since $M$ is $3$-connected, we obtain a \cn\ unless $E-X$ is empty or contains a single element, which must be a point. In the exceptional case, $M$ is a $3$-connected matroid having 5 or 6 elements and containing a 5-element subset that contains two triangles and two triads. But there is no \thc\ matroid with these properties. We deduce that (ii) does not hold.

We now know that (i) holds and that $T^*$ contains $\{x_k,x_{k-1}\}$. Then $k$ is even. Let $X = \{x_1,x_2,\dots,x_k\}$. As $M \ba x_2$ has a c-minor isomorphic to $N$ and has $\{x_1,x_3\}$ as a series pair of points, it follows that $M \ba x_2/x_1$, and hence, $M/x_1$ has a c-minor isomorphic to $N$. Thus, by \cite[Lemma 4.2]{oswww}, $M$ has a triangle containing $x_1$ and exactly one of $x_2$ and $x_3$. This triangle must also contain $x_k$ or $x_{k-1}$. Hence
$r(X) \le r(\{x_2,x_4,x_6,\dots,x_k\}) \le \tfrac{k}{2}.$ Also $r^*(X) \le r(\{x_1,x_3,x_5,\dots,x_{k-1}\})\le \tfrac{k}{2}.$ Thus, by Lemma~\ref{rr*}, $\lambda(X) = 0$, so $X = E(M)$. 
Hence $M$ is a \thc\ matroid in which every element is in both a triangle and a triad, so $M$ is a whirl or the cycle matroid of a wheel; a \cn.

We still need to consider the case when  $\{x_{k-2},x_{k-1},x_k\}$ is a triad of $M$. Then it is a triangle of $M^*$ and the result follows by replacing $M$ by $M^*$ in the argument above.
\end{proof}

\begin{proof}[Proof of Lemma~\ref{Step1}.] 
Suppose $M\ba \ell$ is disconnected having a c-minor isomorphic to $N$. Then $E(M\ba \ell)$ has a non-empty proper subset  $X$ 
 such that $\lambda_{M\ba \ell}(X) = 0$ and   $M\ba \ell \ba X$ has $N$ as a c-minor. Then, by Lemma~\ref{Step0}, every element of $X$ must be a line. Let $Y = E(M\ba \ell) - X$. Since $r(M\ba \ell) = r(M),$ we deduce that 
\begin{equation}
\label{xym}
r(X) + r(Y) = r(M).
\end{equation}
As $r(X) \ge 2$ and $(X, Y \cup \ell)$ is not a \tws\ of $M$, we deduce that $r(Y \cup \ell) = r(Y) + 2$. It follows, since $Y$ and $\ell$ are skew and $M\ba X$ has $N$ as a c-minor, that $M/ \ell$ has $N$ as a c-minor. Since $(M,N)$ is a counterexample to the theorem, $M/ \ell$ is not \thc. Thus there is a partition $(C_1,C_2)$ of $E(M) - \ell$ such that, for some $k$ in $\{1,2\}$,  
\begin{equation}
\label{eqk}
r_{M/ \ell}(C_1) + r_{M/ \ell}(C_2) \le r(M/ \ell) + k-1
\end{equation} 
where, if $k = 2$, we may assume that $\min\{|C_1|, r_{M/ \ell}(C_1),|C_2|, r_{M/ \ell}(C_2)\} \ge 2$. 
Hence
\begin{equation}
\label{c12}
r(C_1 \cup \ell) + r(C_2\cup \ell) \le r(M) + 3.
\end{equation}
By (\ref{xym}), (\ref{c12}), and submodularity, 
$r(X \cup C_1 \cup \ell) + r(X \cap C_1) + r(Y \cup C_2 \cup \ell) + r(Y \cap C_2) \le 2r(M) + 3$. 
Then 
$$r(X \cup C_1 \cup \ell) + r(Y \cap C_2) \le r(M) + 1  \text{~~ or}$$ 
$$r(Y \cup C_2 \cup \ell) + r(X \cap C_1) \le r(M) + 1,$$ 
so 
$$r(Y \cap C_2) \le 1  \text{~~or~~}  r(X \cap C_1) \le 1.$$ 
By symmetry, 
$$r(Y \cap C_1) \le 1  \text{~~or~~}  r(X \cap C_2) \le 1.$$
Since $X$ does not contain any points, either $r(Y \cap C_2) \le 1$ and $r(Y \cap C_1) \le 1$; or, for some $i$ in $\{1,2\}$, 
$$X \cap C_i = \emptyset \text{~~ and ~~} r(Y \cap C_i) \le 1.$$
In the former case, $ |Y| \le 2$; a \cn\ since $Y$ contains $E(N)$. In the latter case, we may assume that $C_1$ consists of a single point $p$. Then we deduce that $k = 1$ in (\ref{eqk}). Thus $p$ is a point of $M$ and $\{p\}$ is a component of $M/ \ell$. Hence both $M\ba p$ and $M/p$ have $N$ as c-minors; a \cn\ to Lemma~\ref{Step0}. We conclude that if $M\ba \ell$ has a c-minor isomorphic to $N$, then $M\ba \ell$ is $2$-connected.

Now suppose that $M/ \ell$ is disconnected having a c-minor isomorphic to $N$. By Lemma~\ref{compact0}, 
$$\lambda_{M/ \ell} = \lambda_{(M/ \ell)^*} = \lambda_{(M^*\ba \ell)^{\flat}} = \lambda_{M^*\ba \ell}.$$
Thus, by replacing $M$ by $M^*$ in the argument above, we deduce that if $M/ \ell$ has  a c-minor isomorphic to $N$, then $M/ \ell$ is $2$-connected.
\end{proof}

\section{If all 2-separations have a side with at most two elements}
\label{alltwos}

The purpose of this section is to treat (\ref{Step2.2}). The argument here is long as it involves analyzing numerous cases. The setup is that $M$ and $N$ are   \thc\ $2$-polymatroids such that $|E(N)| \ge 4$. The pair $(M,N)$ is a counterexample to Theorem~\ref{modc} and  $M$ has an element $\ell$ such that $M\ba \ell$ has $N$ as a c-minor. Thus $M\baba \ell$ is not $3$-connected. We assume that the non-$N$-side of every non-trivial $2$-separation of $M\ba \ell$ has exactly two elements. Thus $\mu(\ell) = 2$. Let   $(X,Y)$ be a non-trivial 2-separation of $M\ba \ell$ in which $Y$ is the non-$N$-side. Now $M\ba \ell$ can be written as the 2-sum, with basepoint $p$, of $2$-polymatroids $M_X$ and  $M_Y$ having ground sets $X \cup p$ and $Y \cup p$. 
The first lemma identifies the  various possibilities for $M_Y$.

\begin{lemma}
\label{old2} 
Let $P$ be a $2$-connected $2$-polymatroid with three elements and rank at least two. Suppose $P$ has a distinguished point $p$. Then $P$ is one of the nine $2$-polymatroids, $P_1, P_2,\dots,P_9$, depicted in Figure~\ref{9lives}.
\end{lemma}

\begin{figure}[h]
\center
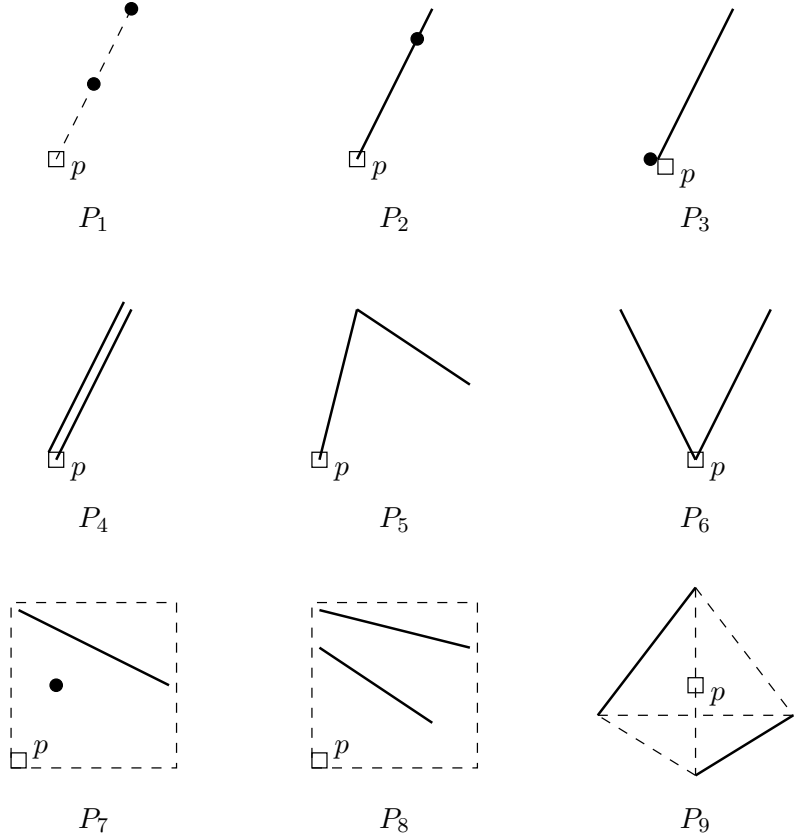
\caption{The nine possible $3$-element $2$-polymatroids in Lemma~\ref{old2}.}
\label{9lives}
\end{figure}


\begin{proof} As $P$ is $2$-connected having rank at least $2$, we see that $2 \le r(P) \le 4$. If $r(P) = 2$, then $P$ is one of $P_1,P_2,P_3,$ or $P_4$; if $r(P) = 3$, then $P$ is one of $P_5,P_6,P_7,$ or $P_8$; if $r(P) = 4$, then $P$ is $P_9$.
\end{proof}

We shall systematically eliminate the various possibilities for $M_Y$. In each case, we will label the two elements of $M_Y$ other than $p$ by $a$ and $b$.

\begin{lemma}
\label{not23} 
$M_Y$ is not isomorphic to $P_2$ or $P_3$.
\end{lemma}

\begin{proof}
Assume the contrary. Then $M_Y$ and hence $M$ has a point $q$ on a line $y$ where $q\neq p$. Thus $M\ba q$ is \thc. As $M\baba \ell$ has $N$ as a c-minor, it follows that $M\baba \ell\ba q$, and hence $M\ba q$, has a c-minor isomorphic to $N$; a \cn.
\end{proof}

\begin{lemma}
\label{fourmost} 
$M_Y$ is not isomorphic to $P_4$.
\end{lemma}

\begin{proof}
Assume the contrary. Let the two parallel lines in $M_Y$ be  $y$ and $y'$ where we may assume that $y \not\in E(N)$. Now $M\ba y$ is \thc, so $M\ba y$ does not have $N$ as a c-minor. Thus $M/y$ has $N$ as a c-minor. But $y'$ is a loop of $M/y$, so $y' \not\in E(N)$ and $M\ba y'$ has $N$ as a c-minor. Since $M\ba y'$ is \thc, we have a \cn.     
\end{proof}

The next lemma is designed to facilitate the elimination of the cases when $M_Y$ is one of $P_1$, $P_7$, or $P_9$.

\begin{lemma}
\label{179}
Suppose both $a$ and $b$ are skew to $p$ in $M_Y$, and both $M_Y/a$ and $M_Y/b$ are $2$-connected. Then $M/a$ and $M/b$ have $2$-separations $(X_a,Y_a)$ and $(X_b,Y_b)$ such that $\ell \in Y_a \cap Y_b$. Moreover, both $M/a$ and $M/b$ have special $N$-minors, and
\begin{itemize}
\item[(i)] $b \in X_a$ and $a \in X_b$;
\item[(ii)] both $Y_a$ and $Y_b$ properly contain $\{\ell\}$; 
\item[(iii)] $(X_a,Y_a - \ell)$ and $(X_b,Y_b-\ell)$ are $2$-separating partitions of $M/a\ba \ell$ and $M/b\ba \ell$, respectively, and $\ell \in \cl_{M/a}(Y_a - \ell)$ and $\ell \in \cl_{M/b}(Y_b - \ell)$; 
\item[(iv)] $(X_a \cup a,Y_a - \ell)$ and $(X_b \cup b,Y_b-\ell)$ are $2$-separating partitions of $M\ba \ell$; 
\item[(v)] for $c$ in $\{a,b\}$, provided $a$ or $b$ is a point, $(X_c,Y_c - \ell)$ is a $2$-separation of $M/c\ba \ell$ and $(X_c \cup c,Y_c - \ell)$ is a $2$-separation of $M\ba \ell$;
\item[(vi)] either $(Y_a - \ell) \cap (Y_b - \ell) \neq \emptyset$; or each of $X_b \cap (Y_a - \ell)$ and $X_a \cap (Y_b - \ell)$ consists of a single point, both $a$ and $b$ are lines of $M$, and, when $r(\{a,b\}) = 4$, the element $\ell$ is a point of $M$.
\end{itemize}
\end{lemma}

\begin{proof} Since both $M_Y/a$ and $M_Y/b$ are $2$-connected, it follows by Lemma~\ref{claim1} that  both $M\ba \ell/a$ and $M\ba \ell/b$ have special $N$-minors. Hence so do both $M/a$ and $M/b$. Since the theorem fails, $M/a$ and $M/b$ have $2$-separations $(X_a,Y_a)$ and $(X_b,Y_b)$ such that $\ell \in Y_a \cap Y_b$.

To see that (i) holds, it suffices to show that $b \in X_a$. Assume $b\in Y_a$. Then 
$$r_{M/a}(X_a) + r_{M/a}(Y_a) = r(M/a) + 1,$$ so 
$r_{M}(X_a \cup a) - r_M(\{a\})+ r_{M}(Y_a \cup a) = r(M) + 1$. As $a$ is skew to $p$ in $M_Y$, it follows that $a$ is skew to $X$ in $M$. Since $X_a \subseteq X$, it follows that $(X_a,Y_a \cup a)$ is a \tws\ of $M$; a \cn. Hence (i) holds.

Part (ii) is an immediate consequence of Lemma~\ref{hath}. To prove (iii), first observe that, by Proposition~\ref{connconn}, $M/a\ba \ell$ is $2$-connected. We show next that 
\begin{equation}
\label{aell}
r(M/a \ba \ell) = r(M/a).
\end{equation}
Suppose not. Then $r(M/a \ba \ell) \le r(M/a) - 1.$   Since $M/a$ is   $2$-connected, it follows that   equality must hold here and $\ell$ is a line of $M/a$. This gives a \cn\ to Lemma~\ref{hath}. Hence (\ref{aell}) holds. 

Now
\begin{align*}
r(M/a\ba \ell) + 1 & \le r_{M/a\ba \ell}(X_a) + r_{M/a\ba \ell}(Y_a - \ell)\\
& \le r_{M/a}(X_a) + r_{M/a}(Y_a)\\
& = r(M/a) +1\\
& = r(M/a \ba \ell) +1,
\end{align*}
where the last equality follows from (\ref{aell}). We see that equality must hold throughout the last chain of inequalities. Hence 
$(X_a, Y_a - \ell)$ is a $2$-separating partition of $M/a\ba \ell$, and $\ell \in \cl_{M/a}(Y_a - \ell)$. Using symmetry, we deduce that  (iii) holds. 

Since $b\in X_a$, we see that $Y_a - \ell \subseteq X$, so $a$ is skew to $Y_a - \ell$. It follows by (iii) that $(X_a \cup a,Y_a - \ell)$ is a 2-separating partition of $M\ba \ell$, and (iv) follows by symmetry.

To show (v), observe that, since $Y_c - \ell$ avoids $\{a,b\}$, it follows that $c$ is skew to $Y_c - \ell$. Thus it suffices to show that $(X_c,Y_c - \ell)$ is a $2$-separation of $M/c\ba \ell$. Assume it is not. Then $Y_c - \ell$ consists of a single point $e$ of $M/c\ba \ell$. Then $e$ is a point of $M$ and, by (iii),  $r_{M/c}(\{e,\ell\}) = r_{M/c}(\{e\}) = 1$, so 
\begin{equation}
\label{cel}
r_M(\{c,e,\ell\})= r_M(\{c,e\}) = 1 + r(\{c\}).
\end{equation} 

Suppose $c$ is a point. If $\ell$ is a line, then $c$ and $e$ are on $\ell$, so $M\ba e$ is \thc. Since $M/c$ has $e$ and $\ell$ as parallel points, $M\ba e$ is \thc\ having a c-minor isomorphic to $N$; a \cn. Thus we may assume that $\ell$ is a point. Then $\{e,\ell,c\}$ is a triangle in $M$. Thus, for $\{c,d\} = \{a,b\}$, we see that $(X \cup \ell \cup c,\{d\})$ is a \tws\ of $M$ unless $d$ is a point of $M$. In  the exceptional case,  $M$ has $e,c,\ell,d$ as a fan with $M/c$ having a c-minor isomorphic to $N$. Thus, by Lemmas~\ref{fantan} and \ref{Step0}, we have a \cn. 

We may now assume that $c$ is a line. Then $r(\{c,e\}) = 3$, so, by (\ref{cel}), $r(\{c,\ell\}) = 3$. Thus  $(X,\{c,\ell\})$ is a \tws\ of $M\ba d$ where $\{c,d\} = \{a,b\}$. Moreover, by hypothesis, $d$ is a point. Thus, by Lemma~\ref{newbix}, we obtain the \cn\ that $M/d$ is \thc\ unless $M/d$ has a parallel pair $\{z_1,z_2\}$ of points. In the exceptional case, we deduce that $z_1$, say, is $\ell$. Hence  $(X \cup \ell \cup d,\{c\})$ is a \tws\ of $M$; a \cn. We conclude that (v) holds.

To prove (vi), assume that $(Y_a - \ell) \cap (Y_b - \ell) = \emptyset$. Then $Y_b - \ell \subseteq X_a \cup a$. But $\ell \in \cl((Y_b - \ell) \cup b)$ and $b \in X_a$, so $\ell \in \cl(X_a \cup a)$. Because $M$ is \thc, it follows that $Y_a - \ell$ consists of a single point $a'$. By symmetry, $Y_b - \ell$ consists of a single point $b'$. Then $(X_a \cup a, Y_a - \ell)$ is not a \tws\ of $M\ba \ell$, so, by (v), each of $a$ and $b$ is a line of $M$. 

To finish the proof of (vi), it remains to show that, when $r(\{a,b\}) = 4$, the element $\ell$ is a point of $M$. Assume $\ell$ is a line. Then, in $M/a$, we have  $a'$ and $\ell$ as parallel points, so $M/a \ba a'$, and hence $M\ba a'$, has a c-minor isomorphic to $N$. As $\ell \in \cl_{M/a}(\{a'\})$, it follows that $\ell \in \cl_M(X \cup a)$. By symmetry, $\ell \in \cl_M(X \cup b)$. Thus
\begin{align*}
r(X) + 2 + r(X) + 2 & = r(X \cup a) + r(X \cup b)\\
& = r(X \cup a \cup \ell) + r(X \cup b \cup \ell)\\
& \ge r(X \cup \ell) + r(M)\\
& = r(X \cup \ell) + r(X) + 3.
\end{align*}
Thus 
$$r(X \cup \ell) \le r(X) + 1.$$
Then 
\begin{align*}
3 + r(X) + 1 & \ge r(\{a',a,\ell\}) + r(X \cup \ell)\\
& \ge r(\{a',\ell\}) + r(X \cup \{a',a,\ell\})\\
& = r(\{a',\ell\}) + r(X \cup a)\\
& = r(\{a',\ell\}) + r(X) + 2.
\end{align*}
We deduce that $r(\{a',\ell\}) = 2$, so $a'$ is a point on the line $\ell$. Thus $M\ba a'$ is   \thc\ having a c-minor isomorphic to $N$; a \cn.
\end{proof}

Next we eliminate the possibility that $M_Y$ is $P_1$.

\begin{lemma}
\label{noone} 
$M_Y$ is not isomorphic to $P_1$.
\end{lemma}

\begin{proof} 
Assume $M_Y$ is isomorphic to $P_1$. Since $\{a,b\}$ is a series pair in $M\ba \ell$, it follows that both $M/a$ and $M/b$ have c-minors isomorphic to $N$. Hence neither $M/a$ nor $M/b$ is \thc. 

We show next that 

\begin{sublemma}
\label{nonesub0}
$\ell$ is a line of $M$. 
\end{sublemma}

Assume $\ell$ is a point. Then $\{\ell,a,b\}$ is a triad of $M$. Since neither $M/a$ nor $M/b$ is \thc, it follows by \cite[Lemma 4.2]{oswww} that $M$ has a triangle containing $a$ and exactly one of $b$ and $\ell$. If $M$ has $\{a,b,c\}$ as  a triangle, then $M/a$ has $\{b,c\}$ as a parallel pair of points. Thus $M/a\ba b$, and hence $M\ba b$, has a c-minor isomorphic to $N$. Thus $b$ is a doubly labelled point; a \cn\ to Lemma~\ref{Step0}. We deduce that $M$ has $\{a,\ell\}$ in a triangle with a point $d$, say. Then $M$ has $d,a,\ell,b$ as a fan with $M/a$ having a c-minor isomorphic to $N$. Thus, by Lemmas~\ref{fantan} and \ref{Step0}, we have a \cn. We conclude that \ref{nonesub0} holds. 

By Lemma~\ref{179}, $M/a$ and $M/b$ have $2$-separations $(X_a,Y_a)$ and $(X_b,Y_b)$ such that $\ell \in Y_a \cap Y_b$. Moreover, both $M/a$ and $M/b$ have special $N$-minors, and 
\begin{itemize}
\item[(i)] $b \in X_a$ and $a \in X_b$;
\item[(ii)] both $Y_a$ and $Y_b$ properly contain $\{\ell\}$; 
\item[(iii)] $(X_a,Y_a - \ell)$ and $(X_b,Y_b-\ell)$ are $2$-separating partitions of $M/a\ba \ell$ and $M/b\ba \ell$, respectively, and $\ell \in \cl_{M/a}(Y_a - \ell)$ and $\ell \in \cl_{M/b}(Y_b - \ell)$; 
\item[(iv)] $(X_a \cup a,Y_a - \ell)$ and $(X_b \cup b,Y_b-\ell)$ are $2$-separating partitions of $M\ba \ell$; and 
\item[(v)] $(Y_a - \ell) \cap (Y_b - \ell) \neq \emptyset$.
\end{itemize}

\begin{sublemma}
\label{noonesub3}
$(Y_a - \ell)\cup (Y_b - \ell) = E - \{a,b,\ell\}$.  
\end{sublemma}

We know that $\lambda_{M\ba \ell}(Y_a - \ell) = 1 = \lambda_{M\ba \ell}(Y_b - \ell)$ and $(Y_a - \ell) \cap (Y_b - \ell) \neq \emptyset$, so 
$\lambda_{M\ba \ell}((Y_a - \ell)\cap (Y_b - \ell)) \geq 1$. Thus, by applying the submodularity of the connectivity function, we see that 
\begin{align*}
1 + 1  & = \lambda_{M\ba \ell}(Y_a - \ell) + \lambda_{M\ba \ell}(Y_b - \ell)\\
& \ge \lambda_{M\ba \ell}((Y_a - \ell)\cap (Y_b - \ell)) + \lambda_{M\ba \ell}((Y_a - \ell)\cup (Y_b - \ell))\\
& \ge  1 + \lambda_{M\ba \ell}((Y_a - \ell)\cup (Y_b - \ell)).
\end{align*}
Since $M\ba \ell$ is $2$-connected, we deduce that $\lambda_{M\ba \ell}((Y_a - \ell)\cup (Y_b - \ell)) = 1$. 

This application of  the submodularity of the connectivity function is an example of an `uncrossing' argument. For the rest of the paper, we will omit the details of such arguments and will follow our stated practice of using the abbreviation `by uncrossing' to mean `by applying the submodularity of the connectivity function.' 

Now $(X_a \cup a, Y_a)$ is not a 2-separation of $M$, so, as $\ell \in \cl_{M/a}(Y_a- \ell)$, we see that  
$$r(Y_a - \ell) < r(Y_a) \le r(Y_a  \cup a) = r((Y_a - \ell) \cup a) \le r(Y_a - \ell) + 1.$$
Hence
$$r(Y_a) = r(Y_a  \cup a) = r((Y_a - \ell) \cup a) = r(Y_a - \ell) + 1.$$
Thus $r((Y_a - \ell) \cup (Y_b - \ell) \cup \{a,b\}) = r(Y_a \cup Y_b  \cup \{a,b\}) = r(Y_a \cup Y_b) \le r((Y_a - \ell)\cup (Y_b - \ell)) + 1.$
Also, as $\{a,b\}$ is a series pair of points in $M\ba \ell$, we see that 
$r(X_a  \cap X_b) \le r((X_a \cap X_b) \cup \{a,b\}) - 1$. Therefore, 
$\lambda_{M}(Y_a \cup Y_b \cup \{a,b\}) \leq 1$. Thus, we may assume that 
$X_a \cap X_b$ consists of a single matroid point, $z$, otherwise $(Y_a - \ell)\cup (Y_b - \ell) = E- \{a,b,\ell\}$ as desired.

Now $\lambda_{M\ba \ell}(X_a \cap X_b) = \lambda_{M\ba \ell}(\{a,b,z\}) = 1$. If $a \notin \cl(\{b,z\})$, then 
$\lambda_{M\ba \ell}((Y_a - \ell) \cup (Y_b - \ell) \cup a) \le 1$, so 
$\lambda_{M}((Y_a-\ell) \cup (Y_b - \ell) \cup a \cup \ell) = \lambda_{M}(Y_a \cup Y_b \cup a) \le 1$; a \cn. Thus $a \in \cl(\{b,z\})$. Hence $\{a,b,z\}$ is a triangle of $M$. It follows   that the point $b$ is doubly labelled; a \cn\ to Lemma~\ref{Step0}. We conclude that \ref{noonesub3} holds. 

\begin{sublemma}
\label{noonesub4}
$Y_a  -  Y_b \neq \emptyset$ and $Y_b  - Y_a \neq \emptyset$.  
\end{sublemma}

By symmetry, it suffices to prove the first of these. Assume $Y_a - Y_b = \emptyset$. Then, as $(Y_a - \ell) \cup (Y_b - \ell) = E - \{a,b,\ell\}$, we deduce that $X_b = \{a\}$, so $(X_b,Y_b - \ell)$ is not a \tws\ of $M\ba \ell/b$; a \cn. Thus \ref{noonesub4} holds. 

By  \ref{noonesub4} and the fact that $(X_a \cup a) \cap (X_b\cup b)$ contains $\{a,b\}$, we see that each of $X_a\cup a$ and $X_b\cup b$ has at least three elements. It follows by the definition of $\mu(\ell)$ that each of $Y_a- \ell$ and $Y_b- \ell$ has exactly two elements. Since each of $Y_a - Y_b, (Y_a \cap Y_b) - \ell,$ and $Y_b - Y_a$ is non-empty, each of these sets has exactly one element. As the union of these sets is $E - \{\ell,a,b\}$, we deduce that $|E(M)| = 6$ and $|X_a \cup a| = 3$. Since  at least one of $a$ and $b$ is not in $E(N)$, we deduce that each of $X_a\cup a$ and 
$Y_a - \ell$  contains at most two elements of $N$; a \cn\ as one of these sets must contain at least three elements of $E(N)$. We conclude that Lemma~\ref{noone} holds. 
\end{proof}


\begin{lemma}
\label{no7} 
$M_Y$ is not isomorphic to $P_7$.
\end{lemma}

\begin{proof} 
Assume that $M_Y$ is isomorphic to $P_7$, letting $a$ be the line. Then, by Lemma~\ref{179},   $M/a$ and $M/b$ have $2$-separations $(X_a,Y_a)$ and $(X_b,Y_b)$ such that $\ell \in Y_a \cap Y_b$. Moreover, both $M/a$ and $M/b$ have special $N$-minors, and  
\begin{itemize}
\item[(i)] $b \in X_a$ and $a \in X_b$;
\item[(ii)] both $Y_a$ and $Y_b$ properly contain $\{\ell\}$; 
\item[(iii)] $(X_a,Y_a - \ell)$ and $(X_b,Y_b-\ell)$ are $2$-separating partitions  of $M/a\ba \ell$ and $M/b\ba \ell$, respectively, and $\ell \in \cl_{M/a}(Y_a - \ell)$ and $\ell \in \cl_{M/b}(Y_b - \ell)$; 
\item[(iv)] $(X_a \cup a,Y_a - \ell)$ and $(X_b \cup b,Y_b-\ell)$ are $2$-separating partitions   of $M\ba \ell$; and 
\item[(v)] $(Y_a - \ell) \cap (Y_b - \ell) \neq \emptyset$.
\end{itemize}

We show next that 

\begin{sublemma}
\label{*5}
$X_a \cap X_b$ is empty or consists of a single point. 
\end{sublemma}

Suppose $b \not\in \cl(Y_b)$. Then $r(Y_b \cup b) = r(Y_b) + 1$. Thus $(X_b \cup b, Y_b)$ is a \tws\ of $M$; a \cn. 
Hence $r(Y_b \cup b) = r(Y_b)$. Now
\begin{align*}
r((Y_a - \ell) \cup (Y_b - \ell)) + 2 & \ge r((Y_a - \ell) \cup (Y_b - \ell) \cup a)\\
& = r(Y_a  \cup (Y_b - \ell) \cup a)\\
& = r(Y_a  \cup Y_b  \cup a)\\
& = r(Y_a  \cup Y_b  \cup a \cup b).
\end{align*}
Also $r(X_a \cap X_b) \le r((X_a \cup a) \cap (X_b \cup b)) - 2$ since $X_a \cap X_b \subseteq X$ and $\sqcap_M(X,Y) = 1$ while $r_M(\{a,b\}) = 3$. 
Thus 
\begin{align*}
\lambda_M(X_a \cap X_b)   & = r(Y_a  \cup Y_b  \cup a \cup b) + r(X_a \cap X_b) - r(M)\\
& \le r((Y_a - \ell) \cup (Y_b - \ell)) + 2 + r((X_a \cup a) \cap (X_b \cup b)) - 2 - r(M\ba \ell)\\
& = \lambda_{M\ba \ell}((X_a\cup a) \cap (X_b\cup b))\\
& = 1,
\end{align*}
where the second-last step follows by uncrossing $(X_a \cup a, Y_a - \ell)$ and $(X_b \cup b, Y_b - \ell)$. 
We deduce that \ref{*5} holds.

\begin{sublemma}
\label{*6}
$E(M) - \{\ell,a,b\}$ contains no  point $\gamma$ such that $\{a,b,\gamma\}$ is $2$-separating in $M\ba \ell$.
\end{sublemma}

To see this, suppose that such a point $\gamma$ exists. Recall that $M\ba \ell$ has $N$ as a c-minor so at most one element of $\{a,b\}$ is in $E(N)$. 
Thus at most two elements of $\{a,b,\gamma\}$ are in $E(N)$. But $|E(N)| \ge 4$. Hence $\{a,b,\gamma\}$ is the non-$N$-side of a $2$-separation of $M\ba \ell$ contradicting the fact that $\mu(\ell) = 2$. We conclude that \ref{*6} holds. 

An immediate consequence of \ref{*6} is that $X_a \cap X_b$ does not consist of a single point. Hence, by \ref{*5}, $X_a\cap X_b = \emptyset$.
As $(X_a,Y_a)$ is a \tws\ of $M/a$, it follows that   $X_a$ cannot contain just the element $b$. Thus 
$(X_a \cup a) \cap (Y_b - \ell) \neq \emptyset$. We show next that 

\begin{sublemma}
\label{*7}
$(X_b \cup b) \cap (Y_a - \ell) \neq \emptyset$.
\end{sublemma}

Suppose $(X_b \cup b) \cap (Y_a - \ell) =  \emptyset$. Then $Y_b - \ell = E(M) - \{a,b,\ell\} = X$ so $r(Y_b - \ell) = r(M) - 2$. Hence 
$r((Y_b - \ell) \cup b) \le  r(M) - 1$. But $\ell \in \cl_{M/b}(Y_b - \ell)$. Thus $r(Y_b  \cup b) \le r(M) - 1$, so $\{a\}$ is 2-separating in $M$; a \cn. We deduce that \ref{*7} holds. 

By uncrossing, $\lambda_{M\ba \ell}((X_b\cup b) \cap (Y_a - \ell)) = 1 = \lambda_{M\ba \ell}((X_a\cup a) \cap (Y_b - \ell))$. As $\ell$ is in both 
$\cl((Y_a - \ell) \cup a)$ and $\cl((Y_b - \ell) \cup b)$, we deduce that each of $(X_a \cup a) \cap (Y_b - \ell)$ and $(X_b \cup b) \cap (Y_a - \ell)$ consists of a single point. Thus we get a \cn\ to \ref{*6} that completes the proof of Lemma~\ref{no7}.
\end{proof}

On combining Lemmas~\ref{not23}, \ref{noone}, and \ref{no7}, we immediately obtain the following.

\begin{corollary}
\label{pointless} 
The non-$N$-side of every $2$-separation of $M\ba \ell$ does not contain any points. 
\end{corollary}

\begin{lemma} 
\label{no9}
$M_Y$ is not isomorphic to $P_9$.
\end{lemma}

\begin{proof}
Assume $M_Y$ is isomorphic to $P_9$. Since each of  $M_Y\ba \ell/a$ and $M_Y\ba \ell/b$ consists of a line through $p$, it follows that both $M/a$ and $M/b$ have c-minors isomorphic to $N$. Hence neither $M/a$ nor $M/b$ is \thc. 
Then $M/a$ and $M/b$ have $2$-separations $(X_a,Y_a)$ and $(X_b,Y_b)$ such that $\ell \in Y_a \cap Y_b$. Moreover, by Lemma~\ref{179}, \begin{itemize}
\item[(i)] $b \in X_a$ and $a \in X_b$;
\item[(ii)] both $Y_a$ and $Y_b$ properly contain $\{\ell\}$; 
\item[(iii)] $(X_a,Y_a - \ell)$ and $(X_b,Y_b-\ell)$ are $2$-separating partitions of $M/a\ba \ell$ and $M/b\ba \ell$, respectively, and $\ell \in \cl_{M/a}(Y_a - \ell)$ and $\ell \in \cl_{M/b}(Y_b - \ell)$; 
\item[(iv)] $(X_a \cup a,Y_a - \ell)$ and $(X_b \cup b,Y_b-\ell)$ are $2$-separating partitions of $M\ba \ell$; and 
\item[(v)] either $(Y_a - \ell) \cap (Y_b - \ell) \neq \emptyset$; or each of $X_b \cap (Y_a - \ell)$ and $X_a \cap (Y_b - \ell)$ consists of a single point, both $a$ and $b$ are lines of $M$, and $\ell$ is a point of $M$.
\end{itemize}

\begin{sublemma}
\label{no9.3}
$(Y_a - \ell) \cap (Y_b - \ell) \neq \emptyset$.
\end{sublemma}

Assume the contrary. Then, by (v), $X_b \cap (Y_a - \ell)$  consists of a point, $a'$, say. By (iii), $\ell \in \cl_{M/a}(\{a'\})$, so $\ell \in \cl(\{a',a\})$. As $r(M) - 3 = r(X)$, it follows that $r(X \cup a \cup \ell) \le r(M) - 1$. Hence the line $\{b\}$ is 2-separating in $M$; a \cn. Thus \ref{no9.3} holds.

\begin{sublemma}
\label{no9.3.5}
$|(Y_a - \ell) \cup (Y_b - \ell)| \geq 2$.
\end{sublemma}

Assume $(Y_a - \ell) \cup (Y_b - \ell)$ contains a unique element, $z$. Then, by \ref{no9.3}, $z \in (Y_a - \ell) \cap (Y_b - \ell)$. Now $\ell \in \cl_{M/a}(\{z\})$, so $\ell \in \cl_M(\{z,a\})$. Thus 
$$r(X \cup a \cup \ell) = r(X \cup a) = r(X) + 2 = r(M) -1,$$
so$(X \cup a \cup \ell, \{b\})$ is a \tws\ of $M$; a \cn. Thus \ref{no9.3.5} holds. 

By \ref{no9.3} and uncrossing, we see that 
$\lambda_{M\ba \ell}((X_a \cup a) \cap (X_b \cup b)) = 1$. Next we show the following.

\begin{sublemma}
\label{no9.4}
$(Y_a - \ell) \cup (Y_b - \ell)$ is the non-$N$-side of a $2$-separation of $M\ba \ell$ and it is a $2$-element set, both members of which are lines.
\end{sublemma}

By \ref{no9.3.5}, $((X_a \cup a) \cap (X_b \cup b),(Y_a - \ell) \cup (Y_b - \ell))$ is a \tws\ of $M\ba \ell$. 
Suppose $(X_a \cup a) \cap (X_b \cup b)$ is the non-$N$-side of this 2-separation. Then, as $\mu(\ell) = 2$, we deduce that  $(X_a \cup a) \cap (X_b \cup b) = \{a,b\}$. 
Thus, as $\ell \in \cl((Y_a - \ell) \cup a)$, 
\begin{align*}
r(M) + 1 & = r((Y_a - \ell) \cup (Y_b - \ell)) + r(\{a,b\})\\
& = r((Y_a - \ell) \cup (Y_b - \ell) \cup a) + r(\{b\})\\
&  = r(Y_a  \cup Y_b  \cup a) + r(\{b\}).
\end{align*}
Hence $\{b\}$ is 2-separating in $M$; a \cn. 
Thus $(Y_a - \ell) \cup (Y_b - \ell)$ must be the non-$N$-side of a 2-separation of $M\ba \ell$, so this set has cardinality two. Moreover, 
by Corollary~\ref{pointless}, both elements of this set are lines. Thus \ref{no9.4} holds.

We deduce from \ref{no9.4} that $Y_a - \ell$ and $Y_b - \ell$ are the  non-$N$-sides of  2-separations of $M\ba \ell$. Thus, by symmetry, we may assume that 
$Y_b - \ell \subseteq Y_a - \ell$. Hence 

\begin{equation}
\label{unc}
(Y_a - \ell) \cup (Y_b - \ell) = Y_a - \ell. 
\end{equation}

\begin{sublemma}
\label{no9.5}
$(Y_a \cup \{a,b\}, X_a \cap X_b)$ is a $2$-separation of $M$.
\end{sublemma}

Since $Y_a - \ell \supseteq Y_b - \ell$, we have
$$r(Y_a \cup a) = r((Y_a - \ell) \cup a) = r(Y_a - \ell) + 2$$  
and
$$r(Y_a \cup b) = r((Y_a - \ell) \cup b) = r(Y_a - \ell) + 2.$$
Moreover, 
$$r(Y_a \cup \{a,b\}) = r((Y_a- \ell) \cup \{a,b\})   \ge r(Y_a - \ell) + 3.$$
Thus, by submodularity, 
\begin{align*}
r(Y_a - \ell) + 2 + r(Y_a - \ell) + 2 & = r(Y_a \cup a) + r(Y_a \cup b)\\
& \ge r(Y_a \cup a \cup b) + r(Y_a)\\
& \ge r(Y_a - \ell) + 3 + r(Y_a)\\
& \ge r(Y_a - \ell) + 3 + r(Y_a - \ell) + 1,
\end{align*}
where the last step follows because $\ell \notin \cl(Y_a - \ell)$. 

We see that equality must hold throughout the last chain of inequalities. 
Hence $r(Y_a) = r(Y_a - \ell) + 1$ and $r(Y_a \cup \{a,b\}) = r(Y_a - \ell) + 3 = r(Y_a) + 2$. 
As $\lambda_{M\ba \ell}(Y_a - \ell) = 1$, it follows that $\lambda_{M}(Y_a) = 2$, that is,
$$r(Y_a) + r((X_a \cup a) \cap (X_b \cup b)) - r(M) = 2.$$
Hence
\begin{align*}
r(Y_a \cup \{a,b\}) + r(X_a \cap X_b) - r(M) & \le r(Y_a) + 2 + r((X_a \cup a) \cap (X_b \cup b))\\
& \hspace*{2in} - 3 - r(M)\\
& = 1.
\end{align*}
Thus $(Y_a \cup \{a,b\}, X_a \cap X_b)$ is a 2-separating partition of $M$. Since $(X_a \cup a) \cap (X_b \cup b)$ is the $N$-side of a 2-separation of $M\ba \ell$, it follows that $X_a \cap X_b$ contains at least two elements of $E(N)$ as $\{a,b\}$ contains at most one element of $E(N)$. Thus
$(Y_a \cup \{a,b\}, X_a \cap X_b)$ is a 2-separation of $M$, that is, \ref{no9.5} holds.  
But \ref{no9.5} gives a \cn\ and thereby completes the proof of Lemma~\ref{no9}.
\end{proof}

We now know that there are only three possibilities for $M_Y$, namely $P_5$, $P_6$, or $P_8$. The next few lemmas will be useful in treating all three cases. 

\begin{lemma}
\label{cactus}
Assume $M\ba \ell$ has $(X,\{a,b\})$ as a $2$-separation where $r(\{a,b\}) = 3$ and each of $a$ and $b$ is a line. Then $r(X \cup \ell)  = r(X) + 1$ if and only if $\{a,b\}$ is a prickly $3$-separating set in $M$.
\end{lemma}

\begin{proof}
If  $\{a,b\}$ is a  $3$-separating set in $M$, then $r(X \cup \ell) = r(M) - 1$. But $r(X) = r(M) - 2$, so $r(X \cup \ell)  = r(X) + 1$. Conversely, if   $r(X \cup \ell) = r(X) + 1$, then $r(X \cup \ell) = r(M) - 1$, so $\{a,b\}$ is a 3-separating set in $M$. Now $r(X \cup \ell \cup a) = r(M)$ otherwise $\{b\}$ is 2-separating in $M$. By symmetry, $r(X \cup \ell \cup b) = r(M)$. Hence $\{a,b\}$ is a prickly $3$-separating set in $M$.
\end{proof}

\begin{lemma}
\label{cactus2}
Assume $M$ has $\{a,b\}$ as a prickly $3$-separating set that is $2$-separating in $M\ba \ell$. Then $M\da a$ and $M\da b$ are $3$-connected having c-minors isomorphic to $N$. 
\end{lemma}

\begin{proof} By Lemma~\ref{portia}, $M\da a$ and $M\da b$ are \thc. Since $M_X$ and $M_Y$  have ground sets $X \cup p$ and $\{a,b,p\}$, we see that $r(M_Y) = 3$.  By Lemma~\ref{pricklytime0}, $M\da a \ba \ell = M\ba \ell \da a$. But $M\ba \ell \da a$ equals the 2-sum of $M_X$ and the $2$-polymatroid consisting of a line $b$ through the point $p$. Compactifying $b$ in $M\ba \ell \da a$ gives the $2$-polymatroid that is obtained from $M_X$ by relabelling $p$ by $b$. Hence $M\da a \baba \ell$ has a c-minor isomorphic to $N$. Thus, using symmetry, so do $M\da a$ and $M\da b$.
\end{proof}

\begin{lemma} 
\label{pixl}
If $M_Y$ is $P_5$, $P_6$, or $P_8$, then $r(X \cup \ell) = r(X) +2$, so $\ell$ is a line.
\end{lemma}

\begin{proof} Assume $r(X \cup \ell) =   r(X) +1$. Then, by Lemma~\ref{cactus}, $\{a,b\}$ is a prickly 3-separating set in $M$. Then, by Lemma~\ref{cactus2}, $M\da a$ and $M\da b$ are $3$-connected having c-minors isomorphic to $N$; a \cn\ to the fact that $(M,N)$ is a counterexample to Theorem~\ref{modc}. Thus $r(X \cup \ell) \neq   r(X) +1$. Since $\ell \not\in \cl(X)$, we deduce that $r(X \cup \ell) = r(X) +2$, so $\ell$ is a line. 
\end{proof}

Next we deal with the case when $M\ba \ell$ has $(X,Y)$ as its only \tws\ with $|Y| = 2$, beginning with the possibility that $M_Y= P_6$.   

\begin{lemma}
\label{duh}
Suppose $M_Y = P_6$ and $(X,Y)$ is the only non-trivial $2$-separation of $M\ba \ell$. Then   
\begin{itemize}
\item[(i)] $M\baba a$ or $M\baba b$ is $3$-connected having a special $N$-minor; or 
\item[(ii)] each of $\{a,\ell\}$ and $\{b,\ell\}$ is a prickly $3$-separator of $M$, and each of $M\da a$ and $M\da b$ is $3$-connected having a c-minor isomorphic to $N$.
\end{itemize}
\end{lemma}

\begin{proof} By Lemma~\ref{pixl}, $\ell$ is a line of $M$ and $\sqcap(X,\ell) = 0$. In $M\baba \ell$, we see that $a$ and $b$ are parallel points. Hence each of $M\ba a$ or $M\ba b$ has a special $N$-minor. But 
$r(E - \{a,b,\ell\}) = r(M) - 2$ and $r(E - \{a,\ell\}) = r(M) - 1$, so $\{\ell\}$ is 2-separating in $M\ba a$. Now both $M\baba a$ or $M\baba b$   have special $N$-minors. Hence we may assume that neither of these matroids is \thc.   


Next we show that 

\begin{sublemma}
\label{duh2}
$r(\{a,\ell\}) = 3 = r(\{b,\ell\})$. 
\end{sublemma}

We shall show that $r(\{b,\ell\}) = 3$, which, by symmetry, will suffice. As $M\baba a$ is not \thc, $M\ba a$ has a non-trivial 2-separation $(A,B)$ in which $A$   contains $\ell$. Then $(A - \ell,B)$ is a 2-separating partition of $M\ba a \ba \ell$. Observe that $r(M\ba a \ba \ell) = r(M) - 1$. Suppose $b \in B$. Then $r(B \cup a) = r(B) + 1$. Thus $(A - \ell, B\cup a)$ is a 2-separating partition of $M \ba \ell$. Since $B\cup a \neq \{a,b\}$, we deduce that  $A - \ell$ contains a unique element. Moreover, as $\sqcap(X,\ell) = 0$, it follows that $r(A) = r(A - \ell) + 2$. Thus $(A- \ell, B \cup a)$ is a 1-separating partition of $M\ba a$; a \cn\ to Lemma~\ref{Step1}. 

We may now assume that $b \in A - \ell$. Then $((A - \ell) \cup a, B)$ is a non-trivial \tws\ of $M\ba \ell$. Thus $(A - \ell) \cup a = \{a,b\}$, so $A = \{b,\ell\}$. Hence $B = X$ and $r(\{b,\ell\}) = 3$. Thus \ref{duh2} holds.



As $r(X \cup a) = r(M) - 1$, we deduce that $\{b, \ell\}$ is a prickly $3$-separator of $M$. Now $M\ba \ell \da b$, which, by Lemma~\ref{pricklytime0},  equals $M \da b \ba \ell$, has a c-minor isomorphic to $N$. Hence so does $M\da b$ and, by symmetry, $M\da a$. Thus, by Lemma~\ref{portia},   
part~(ii) of the lemma holds.
\end{proof}

\begin{lemma}
\label{duh85}
Suppose $M_Y$ is $P_5$ or $P_8$. Let $a$ be an element of $Y$ for which $\sqcap(\{a\},\{p\}) = 0$. Then 
\begin{itemize}
\item[(i)] $M/a$ has a $2$-separation; and 
\item[(ii)] for every   $2$-separation $(A,B)$ of $M/a$ with $\ell$ in $A$,
\begin{itemize}
\item[(a)] $b \in B$; 
\item[(b)] $(A - \ell,B \cup a)$  is a $2$-separation of $M\ba \ell$ and $|B-b| \ge 2$;
\item[(c)] $|A - \ell| \le 2$ and if $|A - \ell| = 1$, then $A - \ell$ consists of a line of $M/a$; 
\item[(d)] $r_{M/a}(A - \ell) = r_{M/a}(A)$; and 
\item[(e)] $\sqcap(\{a,b\}, A - \ell) = 0$.
\end{itemize}
\end{itemize}
Moreover, if $(X,Y)$ is the unique non-trivial $2$-separation of $M\ba \ell$, then $M/a$ has a unique $2$-separation $(A,B)$ with $\ell$ in $A$. Further,  $A - \ell$ consists of a line of $M/a$.
\end{lemma}

\begin{proof} 
Certainly $M\ba \ell/a$ and hence $M/a$ has a c-minor isomorphic to $N$. By Lemma~\ref{pixl}, $\ell$ is a line and $\sqcap(X,\ell) = 0$. As the theorem fails, $M/a$ is not \thc, but, by Lemma~\ref{Step1}, it is $2$-connected. Let $(A,B)$ be a \tws\ of $M/a$ with $\ell$ in $A$. 

\begin{sublemma}
\label{bB}
$b \in B$.
\end{sublemma}

Suppose $b \in A$. Then $a$ is skew to $B$ in $M$, so $(A\cup a,B)$ is a \tws\ of $M$; a \cn. Thus \ref{bB} holds.

\begin{sublemma}
\label{mab}
$M$ does not have a point $c$ such that $B = \{b,c\}$.
\end{sublemma}

Assume the contrary. We have $r_{M/a}(A)+r_{M/a}(B)-r(M/a)=1$, that is, $r(A\cup a)-2+r(\{a, b, c\})-r(M)=1$. But $r(A-\ell)\le r(A\cup a)-2$ and $A - \ell = X - c$.  Hence $r(X-c)+r(\{a, b, c\})-r(M)\le 1$. Since $r(M)=r(M\ba \ell)$, this implies that $(X-c, \{a, b, c\})$ is a $2$-separation of $M\ba \ell$ that violates the fact that $\mu(\ell) = 2$.

If such a point $c$ exists, then $A \cup a \supseteq X \cup a$, so $r(A \cup a) = r(M)$. Hence $r(\{a,b,c\}) = 3 = r(\{a,b\})$, so $(X - c, \{a,b,c\})$ is a \tws\ of $M\ba \ell$ that violates the choice of $Y$. Thus \ref{mab} holds. 

Next we show that 

\begin{sublemma}
\label{aminusl}
$(A - \ell, B)$ is a $2$-separation of $M\ba \ell/a$.
\end{sublemma}

Certainly $(A - \ell, B)$ is $2$-separating in $M\ba \ell/a$. We need to show that $\max\{|A - \ell|, r(A - \ell)\} \ge 2$. By Lemma~\ref{skewer}, $A \neq  \{\ell\}$. Assume $A = \{\ell,c\}$ where $c$ is a point of $M/a$. Then $c$ is a point in $M$ as $a$ is skew to $X$. Moreover, 
\begin{equation}
\label{seec}
c \in \cl_M(X - c)
\end{equation} 
otherwise $(X - c, \{a,b,c\})$ is a \tws\ of $M\ba \ell$ that violates the choice of $Y$.

By Lemma~\ref{skewer}, $a$ is not skew to $\{c,\ell\}$, so $r_{M/a}(\{c, \ell\}) < r_{M}(\{c, \ell\}) \le 3$. Suppose $r_{M/a}(\{c,\ell\}) = 2$. Then 
$r_M(B \cup a) = r(M) - 1$, so $(\{c\}, B \cup a)$ is a 1-separation of $M\ba \ell$; a \cn. We conclude that

\begin{sublemma}
\label{aminusl.1}
$r_{M/a}(\{c,\ell\}) = 1$, so $r_M(\{a,c,\ell\}) = 3$ and $r(M\ba \ell/a) = r(M/a)$. 
\end{sublemma}

Since $c$ and $\ell$ are parallel points in $M/a$, we deduce that $M\ba c$ has a c-minor isomorphic to $N$. Thus $M\ba c$ has a \tws\ $(U,V)$ where we may assume that $\ell \in U$ and $a \in V$ otherwise $M$ has a $2$-separation.

Continuing with the proof of \ref{aminusl}, next we show that 

\begin{sublemma}
\label{bisinP}
$b \in U$.
\end{sublemma}

Suppose $b \in V$. Then, as $a \in V$, we see that $r(V \cup \ell) \le r(V) + 1$ and $r(U - \ell) = r(U) - 2$. Thus  $U = \{\ell\}$ otherwise $(U - \ell, V \cup \ell)$ is a 1-separation of $M\ba c$. But, by (\ref{seec}), $c \in \cl(E - c - \ell)$. Hence $(U,V \cup c)$ is a \tws\ of $M$; a \cn. Hence \ref{bisinP} holds. 

\begin{sublemma}
\label{pnotlb}
$U \neq \{\ell,b\}$. 
\end{sublemma}

Assume $U  = \{\ell,b\}$. Then $V = (X - c) \cup a$. Thus $r(V) \ge r(X) + 1 = r(M) - 1$. But $r(U) \ge 3$ so $(U,V)$ is not a \tws\ of $M\ba c$. Thus \cn\ completes the proof of \ref{pnotlb}.

\begin{sublemma}
\label{pnotlbd}
$M$ does not have a point $d$ such that $U =  \{\ell,b,d\}$. 
\end{sublemma}

Assume the contrary. Then 
\begin{equation}
\label{veeq}
r(V) = r((X - \{c,d\}) \cup a) \ge r(M) - 2
\end{equation}
so, as $r(U) + r(V) = r(M) +1$, we must have that 
\begin{equation}
\label{peep}
r(\{\ell,b,d\}) = r(U) \le 3.
\end{equation}
Thus equality must hold in each of (\ref{veeq}) and (\ref{peep}). 

As $r(\{a,c,\ell\}) = 3$, we have 
\begin{align*}
r(\{b,d\}) + r((X -  d) \cup \{a,\ell\}) & = r(\{b,d\}) + r((X - \{c,d\}) \cup \{a,\ell\})\\
 & \le r(\{\ell,b,d\}) + r((X - \{c,d\}) \cup a) + 1\\
& = r(M) + 2.
\end{align*}
Now $r(\{b,d\}) = 3$, otherwise $\{a,b,d\}$ contradicts the choice of $Y$ since at most one of $a$ and $b$ is in $E(N)$. 
Hence $(\{b,d\}, (X -  d) \cup \{a,\ell\})$ is a 3-separation of $M$. 
Thus 
$r((X -  d) \cup \{a,\ell\}) = r(M) - 1$, so $r((X -  d) \cup a) \le  r(M) - 1$. Hence $r(X - d) \le r(M) - 3$, while $r(X) = r(M) - 2$. Thus $(X - d, \{a,b,d\})$ is a \tws\ of $M\ba \ell$ contradicting the choice of $Y$. We conclude that \ref{pnotlbd} holds.

Now recall that $\{\ell,b\} \subseteq U$ and $a \in V$. Moreover, $r(\{a,c,\ell\}) = 3$ and $\sqcap(a,b) = 1$. Thus 
$$r(V \cup \{\ell,b\}) \le r(V) + 2.$$ 
Also $\ell \not\in \cl(X \cup b)$ otherwise $\{a\}$ is 2-separating in $M$; a \cn. Thus 
$$r(U - \{\ell,b\}) \le r(U) - 2.$$
It follows by \ref{pnotlb} and \ref{pnotlbd} that $(U - \{\ell,b\}, V \cup \{\ell,b\})$ is a \tws\ of $M\ba c$, so $(U - \{\ell,b\}, V \cup \{\ell,b\}\cup c)$ is a \tws\ of $M$. This \cn\ completes the proof of \ref{aminusl}. 

We deduce from \ref{aminusl} that (ii)(d) of the lemma holds, that is, 

\begin{sublemma}
\label{rankla}
$r((A - \ell) \cup a) = r(A \cup a).$ 
\end{sublemma}

Moreover, since $a$ is skew to $X$, and $A - \ell \subseteq X$, it follows, by Lemma~\ref{skewer},  that 
\begin{sublemma}
\label{aminuslba}
$(A - \ell, B \cup a)$ is a $2$-separation of $M\ba \ell$.
\end{sublemma}

Now $(A,B)$ is a \tws\ of $M/a$ and $b \in B$. Since $b$ is a point of $M/a$, it follows that $|B| \ge 2$, so $|B \cup a| \ge 3$. Hence $B \cup a$ is the $N$-side of the \tws\  $(A - \ell, B \cup a)$ of $M\ba \ell$. At most one member of $\{a,b\}$ is in $E(N)$. Since $|E(N)| \ge 4$, it follows that at least two elements of $N$ are in $B- b$, so $|B-b| \ge 2$. Thus (ii)(b) of the lemma holds. 
Moreover,  $|A - \ell| \le 2$. Since $A - \ell$ is one side of a \tws, if it contains a single element, that element is a line of $M/a$. Thus (ii)(c) of the lemma holds.

Next we observe that 
\begin{sublemma}
\label{piab}
$\sqcap(\{a,b\}, A - \ell) = 0.$
\end{sublemma}

Since $\sqcap(\{a,b\}, X) = 1$, we see that $\sqcap(\{a,b\}, A - \ell) \le 1$. Assume $\sqcap(\{a,b\}, A - \ell)  = 1$. Then 
$r((A - \ell) \cup \{a,b\}) = r(A - \ell) + 2$. But $r(A - \ell) + r(B \cup a) = r(M\ba \ell) + 1$. Thus

\begin{equation}
\label{eqbb}
r((A - \ell) \cup \{a,b\}) + r(B - b) \le r(M\ba \ell) + 1.
\end{equation}
By \ref{rankla}, $r((A - \ell) \cup a) = r(A \cup a)$. Hence we obtain the \cn\ that $(A \cup \{a,b\}, B- b)$ is a \tws\ of $M$. Thus \ref{piab} holds.

Now suppose that $(X,Y)$ is the unique non-trivial $2$-separation of $M\ba \ell$. We complete the proof of the lemma by showing that 

\begin{sublemma}
\label{lonely}
$M/a$ has a unique $2$-separation $(A,B)$ with $\ell$ in $A$. Moreover, $A - \ell$ consists of a line of $M/a$.
\end{sublemma}

Let $(A_1, B_1)$ and $(A_2,B_2)$ be distinct 2-separations of $M/a$ with  $\ell$ in $A_1 \cap A_2$. Then $b \in B_1 \cap B_2$. By (ii)(c), $|A_i - \ell| \le 2$. Suppose $|A_i - \ell| = 2$. Then, by (ii)(b), $(A_i - \ell, B_i \cup a)$ is a non-trivial \tws\ of $M\ba \ell$, so $A_i - \ell = Y$; a \cn\ as $a \not\in A_i - \ell$. We deduce that $|A_i - \ell| = 1$, so $A_i - \ell$ consists of a line $m_i$ of $M/a$.

Now $(\{m_1\}, B_1 \cup a)$ and $(\{m_2\}, B_2 \cup a)$ are 2-separations of $M\ba \ell$. Thus $r(\{m_1,m_2\}) = 4$ otherwise one easily checks that 
$(\{m_1,m_2\}, (B_1\cap B_2) \cup a)$ is a \tws\ of $M\ba \ell$ that contradicts the uniqueness of $(X,Y)$. Now $\sqcap(a,X) = 0$, so $\sqcap(a,\{m_1,m_2\}) = 0$. Thus $r_{M/a}(\{m_1,m_2\}) = 4$. But, by (ii)(d) of the lemma, 
\begin{align*}
2 + 2 & = r_{M/a}(\{m_1,\ell\}) + r_{M/a}(\{m_2,\ell\})\\
& \ge r_{M/a}(\{m_1,m_2,\ell\}) + r_{M/a}(\{\ell\})\\
& \ge 4 + 1.
\end{align*}
This \cn\ finishes   the proof of \ref{lonely} and thereby completes the proof of the lemma.
\end{proof}

\begin{lemma}
\label{duh8}
If $M_Y = P_8$, then $(X,Y)$ is not the only non-trivial $2$-separation of $M\ba \ell$.
\end{lemma}

\begin{proof} 
Assume $(X,Y)$ is the unique such \tws. By Lemma~\ref{duh85},  $M/a$ and $M/b$ have unique $2$-separations $(A_1,B_1)$ and $(A_2,B_2)$ with 
$\ell$ in $A_1 \cap A_2$. Moreover, $A_1 - \ell$ and $A_2 - \ell$ consist of lines $\ell_1$ and $\ell_2$ in $M/a$ and $M/b$; and $M\ba \ell$ has $(A_1 - \ell, B_1 \cup a)$ and $(A_2 - \ell, B_2 \cup b)$ as $2$-separations. 

Assume $\ell_1 \neq \ell_2$. Then $\{b,\ell_2\} \subseteq B_1 \cup a$, so $\ell \in \cl(B_1 \cup a)$. Hence $(A_1 - \ell, B_1 \cup a \cup \ell)$ is a \tws\ of $M$; a \cn. Thus $\ell_1 = \ell_2$. Hence $r(\{\ell_1,b,\ell\}) = r(\{\ell_1,b\})   = 4$. But we also know that $r(\{\ell_1,a,\ell\}) = r(\{\ell_1,a\}) = 4$. By Lemma~\ref{duh85}(ii)(a) and (b), we see that $b \notin \cl_{M/a}(A_1)$, so $r(\{\ell_1,\ell,b,a\}) \ge 5$. Thus
\begin{align*}
4+4 & = r(\{\ell_1,\ell, b\}) + r(\{\ell_1,\ell, a\})\\
& \ge r(\{\ell_1,\ell, b,a\}) + r(\{\ell_1,\ell\})\\
& \ge 5 + r(\{\ell_1,\ell\}).
\end{align*}
Therefore $r(\{\ell_1,\ell\}) \le 3$. As $\sqcap(\{\ell_1\},\{\ell\}) = 0$, we deduce that $r(\{\ell\}) = 1$; a \cn\ to Lemma~\ref{pixl}.
 \end{proof}

\begin{lemma}
\label{duh5}
If $M_Y = P_5$, then $(X,Y)$ is not the only non-trivial $2$-separation of $M\ba \ell$.
\end{lemma}

\begin{proof}
Assume $(X,Y)$ is the unique such \tws. Label $Y$ so that $\sqcap(a,X) = 0$ and $\sqcap(b,X) = 1$.

\begin{sublemma}
\label{duh5a}
$\sqcap(a,\ell) = 0$.
\end{sublemma}

Suppose $\sqcap(a,\ell) = 1$. Then $r(\{a,\ell\}) = 3$. Now $r(E - \ell) = r(E)$ and $r(E - \{\ell,a\}) = r(E) - 1$. Thus, by Lemma~\ref{cactus},  $\{a,\ell\}$ is a prickly 3-separator of $M$. Now $M\ba \ell \da a$ has a c-minor isomorphic to $N$ since it is the 2-sum of $M_X$ and the $2$-polymatroid consisting of the line $b$ with the point $p$ on it. But, by Lemma~\ref{pricklytime0}, $M\ba \ell \da a = M\da a \ba \ell$. Thus, by Lemma~\ref{portia}, $M\da a$ is \thc\ having a c-minor isomorphic to $N$; a \cn. We conclude that \ref{duh5a} holds. 

By Lemma~\ref{duh85}, $M/a$ has a unique \tws\ and it has the form $(\{\ell_1,\ell\}, E - \{\ell_1,\ell,a\})$ where $\ell_1$ is a line of $M/a$. Moreover, $r(\{\ell_1,\ell,a\}) = r(\{\ell_1,\ell\})$. 
Now $\sqcap(\ell,a) = 0$ and, by Lemma~\ref{pixl}, $\ell$ is a line of $M$. Thus 
\begin{equation}
\label{ral}
r(\{a,\ell\}) = 4.
\end{equation}

Now $M\ba \ell \ba a$, and hence $M\ba a$, has a c-minor isomorphic to $N$. Thus $M\ba a$ has a non-trivial \tws\ $(U,V)$. Without loss of generality, we may assume that 
$\ell_1 \in U$ and $\ell \in V$ since $r(\{\ell_1,\ell\}) = 4 = r(\{\ell_1,\ell,a\})$. 

\begin{sublemma}
\label{binV}
$b \in V$.
\end{sublemma}

Suppose $b \in U$. Then, as $\sqcap(X,\ell) = 0$, we see that, unless $V = \{\ell,c\}$ for some point $c$, the partition $(U \cup \ell,V-\ell)$ is a \tws\ of $M\ba a$, so $(U \cup \ell \cup a,V-\ell)$ is a \tws\ of $M$. 

Consider the exceptional case. Then $r(V-\ell) = r(V) - 2 = 1$. Now $r(M\ba a, \ell) = r(M) - 1$ and $r(U) + r(V) = r(M) + 1$. We see that $r(U) = r(E - \{a,\ell,c\}) = r(M) - 2.$ Hence $\lambda_{M\ba a,\ell}(\{c\}) = 0$; a \cn. We conclude that \ref{binV} holds.

We now have that $V \supseteq \{\ell,b\}$. Next observe that 

\begin{sublemma}
\label{new2s}
$(U,(V - \ell) \cup a)$ is a $2$-separation of $M\ba \ell$, and $r((V - \ell) \cup a) = r(V)$.
\end{sublemma}

To see this, first note that, since $b \in V- \ell$, we have 
\begin{equation}
\label{vee1}
r((V- \ell) \cup a) \le r(V - \ell) + 1.
\end{equation} 
We also have 
\begin{equation}
\label{vee2}
r(V - \ell) \le r(V) - 1
\end{equation}
otherwise $r(V - \ell) = r(V)$ so $\ell \in \cl(E - \{a,\ell\})$. But $r(E - \{a,\ell\}) = r(E) - 1$, so $(E-a,\{a\})$ is a \tws\ of $M$; a \cn. 
Combining (\ref{vee1}) and (\ref{vee2}) gives \ref{new2s}.

Since $(X,Y)$ is the unique non-trivial \tws\ of $M\ba \ell$, we deduce that $(V - \ell) \cup a = \{a,b\}$. Moreover, by \ref{new2s}, 
$r(\{a,b\}) = 3 = r(\{b,\ell\})$. It follows using submodularity that $r(\{a,b,\ell\}) = 4$. Thus $b \in \cl_{M/a}(\{\ell\})$. Hence 
$(\{\ell_1,\ell,b\}, E- \{\ell_1,\ell,a,b\})$ is a \tws\ of $M/a$, which contradicts the fact that $(\{\ell_1,\ell\}, E - \{\ell_1,\ell,a\})$ is the unique \tws\ of $M/a$. 
This completes the proof of Lemma~\ref{duh5}.
\end{proof}

By Lemma~\ref{cactus2}, $M\ba \ell$ has no 2-element 2-separating set that is a prickly 3-separating set in $M$.

\begin{lemma}
\label{oldstep5}
Let $\{a,b\}$ and $\{c,d\}$ be disjoint $2$-separating sets of $M\ba \ell$ where each of $a$, $b$, $c$, and $d$ is a line, $r(\{a,b\}) = 3 = r(\{c,d\})$, and $\sqcap(\{a\},E - \{a,b,\ell\}) = 0$. 
Then either
\begin{itemize}
\item[(i)] $M/a$ is $3$-connected having a c-minor isomorphic to $N$; or
\item[(ii)] $M/ \ell$ has a c-minor isomorphic to $N$ and $\ell \in \cl_{M/a}(\{c,d\})$.
\end{itemize}
\end{lemma}

\begin{proof} Assume that the lemma fails. Let $Z = E - \{\ell,a,b,c,d\}$. Then, as neither $\{a,b\}$ nor $\{c,d\}$ is a prickly 3-separating set of $M$, by Lemma~\ref{cactus}, we see that 
$$\sqcap(Z \cup \{c,d\}, \{\ell\}) = 0 = \sqcap(Z \cup \{a,b\}, \{\ell\}),$$
so $\sqcap(Z,\{\ell\}) = 0$ and $\sqcap(\{a,b\},\{\ell\}) = 0$. It follows, as $\sqcap(\{a\}, Z) = 0$, that
\begin{equation}
\label{mazl}
\sqcap_{M/a}(Z,\{\ell\}) = 0.
\end{equation}

Let $X = E - \{a,b,\ell\}$ and $Y = \{a,b\}$. Then $M\ba \ell = M_X \oplus_2 M_Y$ where $M_Y$ has ground set $\{p,a,b\}$. Then $M_X$ has a c-minor isomorphic to $N$. As $\sqcap(\{a\}, X) = 0$, it follows that $M\ba \ell/a$, and hence $M/a$, has a c-minor isomorphic to $N$. 

\begin{sublemma}
\label{lcd}
$\ell \in \cl_{M/a}(\{c,d\})$.
\end{sublemma}

Assume $\ell \not\in \cl_{M/a}(\{c,d\})$. Since $M/a$ is not \thc, it has a \tws\ $(A,B)$ with $\ell \in A$ and $b \in B$. Moreover, by Lemma~\ref{duh85}, we know that $(A-\ell, B\cup a)$ is a 2-separation of $M\ba \ell$, that $|B - b| \ge 2$, that $|A - \ell| \le 2$, and that  $\ell \in \cl_{M/a}(A - \ell)$. 

Suppose $|A - \ell| = 1$. Then, by Lemma~\ref{duh85} again, $A - \ell$ consists of a line $m$ of $M/a$ and $\ell \in \cl_{M/a}(\{m\})$. Thus $m \not \in \{c,d\}$, so $m \in Z$ and we have a contradiction to (\ref{mazl}). Now suppose that $|A - \ell| = 2$. Then $\ell \in \cl_{M/a}(A - \ell)$. Thus $\{c,d\} \neq A - \ell$. If $\{c,d\}$ avoids $A - \ell$, then we again get a \cn\ to (\ref{mazl}). Thus $A- \ell$  meets $\{c,d\}$ in a single element. Then, by uncrossing the 2-separations $(A - \ell, B \cup a)$ and $(\{c,d\}, E - \{\ell,c,d\})$ of $M\ba \ell$, we see that $(A - \ell) \cup \{c,d\})$ is a 3-element 2-separating set in $M\ba \ell$. At most one element of $\{c,d\}$ is in $E(N)$. Thus $(A - \ell) \cup \{c,d\}$ is the non-$N$-side of a \tws\ of $M\ba \ell$. This is a \cn\ as this set has three elements. We conclude that \ref{lcd} holds. 

We shall complete the proof of Lemma~\ref{oldstep5} by showing that $M/ \ell$ has a c-minor isomorphic to $N$. In the argument that follows, it helps to think in terms of the matroids that are naturally derived from the $2$-polymatroids we are considering. We know that $M\ba \ell = M_X \oplus_2 M_Y$ where $M_Y$ has ground set $\{a,b,p\}$ with $p$ being the basepoint of the 2-sum. As $\{c,d\}$ is 2-separating in $M\ba \ell$, it is also 2-separating in $M_X$. Thus $M_X = M_Z \oplus_2 M_W$ where $M_W$ has ground set $\{c,d,q\}$ with $q$ being the basepoint of this 2-sum. Now $\{c,d\}$ does not span $p$ otherwise $\{a,b,c,d\}$ is 2-separating in $M\ba \ell$ and contains at most two elements of $N$, a \cn\ to the definition of $Y$. By two applications of Lemma~\ref{claim1}, we see that $M_X$, and hence $M_Z$, has a  c-minor isomorphic to $N$.

Now $M\ba \ell/a$ equals $M_X$ after relabelling the element $p$ of the latter by $b$. We will call this relabelled $2$-polymatroid $M_X'$. By \ref{lcd}, $M/a$ is obtained from $M'_X$ by adding $\ell$ to the closure of $\{c,d\}$ as a point or a line. Thus $M/a$ is the 2-sum with basepoint $q$ of $M'_Z$ and $M'_W$ where $M'_Z$ is obtained from $M_Z$ by relabelling $p$ as $b$, while $M'_W$ is obtained from $M_W$ by adding $\ell$. By (\ref{mazl}), $\ell$ is skew to $Z$ in $M/a$, so $\ell$ is skew to $q$ in $M'_W$. Now  $\ell$ is a not a line of $M'_W$, otherwise at least one of $c$ and $d$ is parallel to the basepoint $q$ in $M'_W$, so $M/a/ \ell$ and hence $M/\ell$ has a c-minor isomorphic to $N$.
Hence $\ell$ is a point of $M'_W$, so $M'_W / \ell$ has rank $2$. It has no point parallel to $q$ otherwise $M/a /\ell$ has a c-minor isomorphic to $N$. Thus $M'_W/ \ell$ can be obtained from one of  $P_1, P_2$, or $P_4$ by relabelling the element $p$   by $q$. In the first two cases, we can contract a point from $M'_W/ \ell$ to obtain a $2$-polymatroid consisting of two parallel points, one of which is $q$, so we get the \cn\ that $M/a / \ell$ has a c-minor isomorphic to $N$. In the third case, deleting one of the lines, say $c$, of $M'_W/ \ell$ leaves $d$ as a line through $q$. Thus $\{d\}$ is $2$-separating in  $M/a \ba \ell \ba c$. Compactifying $d$, we obtain a $2$-polymatroid  having a c-minor isomorphic to $N$. Again we obtain the \cn\ that  $M/a \ba \ell$ has a c-minor isomorphic to $N$. 
\end{proof}

\begin{lemma}
\label{3connel}
Let $\{a,b\}$ and $\{c,d\}$ be disjoint $2$-separating sets of $M\ba \ell$ where each of $a$, $b$, $c$, and $d$ is a line, $r(\{a,b\}) = 3 = r(\{c,d\})$. 
Assume $M/ \ell$ has a c-minor isomorphic to $N$. 
Then at least one of  $\sqcap(\{a\}, E - \{\ell,a,b\})$ and  $\sqcap(\{b\}, E - \{\ell,a,b\})$ is not equal to one.
\end{lemma}

\begin{proof}
As before, let $Z = E - \{a,b,c,d,\ell\}$. Since the theorem fails, it follows by Lemmas~\ref{Step1} and \ref{cactus2} that $M/ \ell$ is $2$-connected and neither $\{a,b\}$ nor $\{c,d\}$ is a prickly 3-separating set of $M$. Moreover, by Lemma~\ref{pixl}, $\ell$ is a line that is skew to each of $Z \cup \{a,b\}$ and $Z \cup \{c,d\}$. Thus, if $(R,B)$ is a \tws\ of $M/ \ell$, then, by 
Lemma~\ref{skewer},  $\sqcap(R,\{\ell\}) \ge 1$ and $\sqcap(B,\{\ell\}) \ge 1$.

By Lemma~\ref{general}, 
$$\sqcap(R,\{\ell\}) + \sqcap(B,\{\ell\}) + \lambda_{M/ \ell}(R) = \lambda_{M\ba \ell}(R) + \lambda_M(\{\ell\}),$$
so 
\begin{equation}
\label{eq1rb}
\sqcap(R,\{\ell\}) + \sqcap(B,\{\ell\})   = \lambda_{M\ba \ell}(R) + 1.
\end{equation}

As $\sqcap(\{\ell\}, Z \cup \{a,b\}) = 0 = \sqcap(\{\ell\}, Z \cup \{c,d\})$, it follows by Lemma~\ref{8.2.3} that 
both $R$ and $B$ meet both $\{a,b\}$ and $\{c,d\}$.  Without loss of generality, we may assume that $\{a,c\} \subseteq R$ and $\{b,d\} \subseteq B$.



Now suppose that $\sqcap(\{a\}, E - \{\ell,a,b\}) = 1 = \sqcap(\{b\}, E - \{\ell,a,b\})$. 
By Lemma~\ref{oswrules}(i),
\begin{align*}
\sqcap(\{a,c\}, \{b,d\}) + \sqcap(\{a\},\{c\}) + \sqcap(\{b\},\{d\}) & = \sqcap(\{a,b\}, \{c,d\}) + \sqcap(\{a\},\{b\})\\ 
&\hspace*{1.5in}  + \sqcap(\{c\},\{d\}).
\end{align*}
As $\mu(\ell) = 2$, we see that $\sqcap(\{a,b\}, \{c,d\}) = 0$, so $\sqcap(\{a\},\{c\}) = 0 =  \sqcap(\{b\},\{d\})$. Thus 
$$\sqcap(\{a,c\}, \{b,d\})    =     \sqcap(\{a\},\{b\}) + \sqcap(\{c\},\{d\}) = 2.$$
Hence $\sqcap(R,B) \ge 2$, that is, $\lambda_{M\ba \ell}(R) \ge 2$. Thus, by (\ref{eq1rb}),  $\sqcap(R,\{\ell\}) = 2$ or $\sqcap(B,\{\ell\}) = 2$. By symmetry, we may assume the former.  But, as $\sqcap(\{c,d\} \cup Z, \{\ell\}) = 0$ and $\sqcap(\{c,d\} \cup Z, \{a\}) = 1$, by Lemma~\ref{oswrules}(ii), 
\begin{align*}
\sqcap(\{c,d\} \cup Z \cup a, \{\ell\}) + 1 & = \sqcap(\{c,d\} \cup Z \cup a, \{\ell\}) +  \sqcap(\{c,d\} \cup Z,  \{a\})\\
& = \sqcap(\{c,d\} \cup Z \cup \ell, \{a\}) + \sqcap(\{c,d\} \cup Z, \{\ell\})\\
& \le 2 + 0.
\end{align*}
Thus $\sqcap(\{c,d\} \cup Z \cup a, \{\ell\}) \le 1$. But $R \subseteq Z \cup \{a,c\}$ so  $\sqcap(R, \{\ell\}) \le 1$; a contradiction. 
\end{proof}

\begin{lemma}
\label{22sep}
The $2$-polymatroid $M\ba \ell$ does not have  two disjoint  $2$-element $2$-separating sets. 
\end{lemma}

\begin{proof} Assume that $M\ba \ell$ has  $\{a,b\}$ and $\{c,d\}$ as disjoint $2$-separating sets. Then each of $a, b, c$, and $d$ is a line and $r(\{a,b\}) = 3 = r(\{c,d\})$. As before, let $Z = E - \{a,b,c,d,\ell\}$. Suppose $Y$ is $\{a,b\}$ or $\{c,d\}$, and $X = E - \ell - Y$. Then $M\ba \ell = M_X \oplus_2 M_Y$. By  Lemmas~\ref{not23}, \ref{fourmost}, \ref{noone}, \ref{no7}, and \ref{no9}, we know that $M_Y$ is isomorphic to $P_5$, $P_6$, or $P_8$. By Lemma~\ref{pixl},

\begin{sublemma}
\label{skewy}
$\ell$ is skew to $X$, so $\ell$ is skew to each of $a$, $b$, $c$, and $d$.
\end{sublemma}

When $M_Y \cong P_n$, we shall say that $Y$ is a {\it type-$n$ $2$-separator} of $M\ba \ell$.

\begin{sublemma}
\label{not66}
Neither $\{a,b\}$ nor $\{c,d\}$ is of type-$6$.
\end{sublemma}

Assume the contrary. Suppose $\{a,b\}$ is of type-6. Then, by Lemma~\ref{3connel}, $M/ \ell$ does not have a c-minor isomorphic to $N$. Thus, by Lemma~\ref{oldstep5}, neither $\sqcap(\{c\},X)$ nor $\sqcap(\{d\},X)$ is $0$. Hence $\{c,d\}$ is also of type-6. 
Suppose $\alpha \in \{a,b\}$ and $\gamma \in \{c,d\}$. Then $r(Z \cup \{\alpha, \gamma\}) = r(Z) + 2$. Of course, $r(M) = r(Z) + 4$.

Suppose $r(Z \cup \{\alpha, \gamma\} \cup \ell) =  r(M)$. Then $\sqcap(Z \cup \{\alpha, \gamma\}, \ell) = 0$. Let  the elements of $\{a,b,c,d\} - \{\alpha, \gamma\}$ be $\beta$ and $\delta$. In $M\ba \beta \ba \delta$, the set $\{\ell\}$  is $1$-separating. Thus $M\ba \beta \ba \delta\ba \ell =  M\ba \beta \ba \delta/ \ell$. As $M\ba \beta \ba \delta\ba \ell$ has a c-minor isomorphic to $N$, so does $M/\ell$. We then get a \cn\ to Lemma~\ref{3connel} since 
$\sqcap(a, E - \{\ell,a,b\}) = 1 = \sqcap(b, E - \{\ell,a,b\})$.

We may now assume that $r(Z \cup \{\alpha, \gamma\} \cup \ell) \le  r(M) - 1.$ By  \ref{skewy}, $\ell$ is skew to $Z \cup \{a,b\}$, so 
$r(Z \cup a \cup \ell) =  r(M) - 1.$ Thus, using the submodularity of $r$, we have 
\begin{align*}
2r(M) - 1 & = r(Z \cup a \cup \ell) + r(M)\\
& \le r(Z \cup \{a,c\} \cup \ell) + r(Z \cup \{a,d\} \cup \ell)\\
& \le 2r(M) - 2.
\end{align*}
This \cn\ establishes \ref{not66}.

We now know that  each of  $\{a,b\}$ and $\{c,d\}$  is of type-5 or of type-8. 
In particular, we may assume that 
$\sqcap(\{a\}, Z \cup \{c,d\}) = 0 = \sqcap(\{c\}, Z \cup \{a,b\})$. 
Since $\mu(\ell) = 2$ and $\{a,b,c,d\}$ contains at most two elements of $N$, we see that 
\begin{equation}
\label{abcd6}
r(\{a,b,c,d\}) = 6. 
\end{equation}

By Lemma~\ref{oldstep5}, 

\begin{sublemma}
\label{mell}
$\ell \in \cl_{M/a}(\{c,d\})$ and $\ell \in \cl_{M/c}(\{a,b\})$.
\end{sublemma}
We deduce that $r(\{a,c,d,\ell\}) = r(\{a,c,d\}) = 5$ and 
$r(\{a,b,c,\ell\}) = r(\{a,b,c\}) = 5$.  By submodularity and (\ref{abcd6}),
\begin{align*}
10 & = r(\{a,c,d,\ell\}) + r(\{a,b,c,\ell\})\\
& \ge r(\{a,b,c,d,\ell\}) + r(\{a,c,\ell\})\\
& \ge 6 + 4 = 10.
\end{align*}
We conclude that 
\begin{equation}
\label{eqnacl}
r(\{a,c,\ell\}) = 4.
\end{equation}

Next we show the following. 

\begin{sublemma}
\label{type5}
Both $\{a,b\}$ and $\{c,d\}$ are of type-$5$.
\end{sublemma}

Suppose $\{a,b\}$ is of type-8. Then $\sqcap(\{b\}, Z \cup \{c,d\}) = 0$. Thus we can replace $a$ by $b$ in the argument used to prove (\ref{eqnacl}) to get that $r(\{b,c,\ell\}) = 4$. Hence 
\begin{align*}
4 + 4 & = r(\{a,c,\ell\}) + r(\{b,c,\ell\})\\
& \ge r(\{a,b,c,\ell\}) + r(\{c,\ell\})\\
& \ge 5 + 4.
\end{align*}
This \cn\ and symmetry implies that 
\ref{type5} holds.

Now, by Lemma~\ref{claim1}, $M\ba \ell \ba a$, and hence $M\ba a$, has a c-minor isomorphic to $N$. Thus $M\ba a$ is not \thc. Let $(U,V)$ be a non-trivial \tws\ of $M\ba a$. Then we may assume that $\ell \in U$ and $c\in V$ otherwise $M$ has a \tws. 

Suppose $d \in U$. Then,   by \ref{skewy}, $(U \cup c, V - c)$ is a 1-separation of $M\ba a$; a \cn. Thus $d \in V$. By \ref{skewy} again, $r(U- \ell) = r(U) - 2$, so we obtain the \cn\ that $(U - \ell, V \cup \ell \cup a)$ is a 1- or \tws\ of $M$ unless $U - \ell$ 
consists of a single point, $u$, and $r(U) = 3$. In the exceptional case, since $M\ba a\ba \ell$ is $2$-connected, we see that $u \in \cl(V)$, so $(U - u, V \cup u)$ is a $1$-separation of $M\ba a$; a \cn.
\end{proof}

\begin{lemma}
\label{muend}
Suppose that $M$ has an element $\ell$ such that $M\ba \ell$ has $N$ as a c-minor.   
Then the largest non-$N$-side in a $2$-separation of $M\ba \ell$ has size exceeding two.
\end{lemma}

\begin{proof}
Assume $\mu(\ell) = 2$. Then $M\ba \ell = M_X \oplus_2 M_Y$ where $|Y| = 2$. In Lemma~\ref{old2}, we identified the nine possibilities for $M_Y$. We showed in   Lemmas~\ref{not23}, \ref{fourmost}, \ref{noone}, \ref{no7}, and \ref{no9}  that $M_Y$ must be isomorphic to $P_5$, $P_6$, or $P_8$. In Lemmas~\ref{duh}, \ref{duh8}, and \ref{duh5}, we showed that $(X,Y)$ cannot be the sole non-trivial 2-separation of $M\ba \ell$. Lemma~\ref{22sep} completes the proof by showing that $M\ba \ell$ cannot have a second non-trivial 2-separation.
\end{proof}

\begin{lemma}
\label{dualmu}
Suppose that $M$ has an element $\ell$ such that $M/\ell$ has $N$ as a c-minor.   
Then the largest non-$N$-side in a $2$-separation of $M/ \ell$ has size exceeding two.
\end{lemma}

\begin{proof} 
By Lemma~\ref{csm}, $(M/ \ell)^*$ has a c-minor isomorphic to $N^*$. By Lemma~\ref{compact0}, $(M/ \ell)^* = (M^*\ba \ell)^{\flat}$. Thus $M^*\ba \ell$ has a c-minor isomorphic to $N^*$. Let $Y$ be a largest non-$N$-side in a \tws\ of $M/ \ell$. By Lemma~\ref{compact0} again, $Y$ is a largest non-$N^*$-side in a \tws\ of $M^*\ba \ell$. Replacing $(M,N)$ by $(M^*,N^*)$ in Lemma~\ref{muend}, we deduce that $|Y| > 2$.
\end{proof}

\section{Finding a doubly labelled line}
\label{fdll}

Recall that we are assuming that $(M,N)$ is a counterexample to Theorem~\ref{modc} where $N$ is a $3$-connected $2$-polymatroid that is a c-minor of $M$. In this section, we prove some lemmas that will eventually enable us to deduce that $M$ has a doubly labelled line. The first step in this process is to prove the following elementary but useful lemma. 

\begin{lemma}
\label{predichotomy} 
Suppose $y \in E(M) - E(N)$. If $y$ is not a doubly labelled element of $M$, and $M'$ has a special $N$-minor for some $M'$ in $\{M\ba y,M/ y\}$, then $M'$ has $N$ as a c-minor.
\end{lemma}

\begin{proof} Since $y \in E(M) - E(N)$,   some $M''$ in  $\{M\ba y,M/ y\}$ has $N$ as a c-minor. Since $y$ is not doubly labelled, we see that $M'' = M'$.  
\end{proof}

The next lemma identifies an important dichotomy. 

\begin{lemma}
\label{dichotomy} 
Let $M'$ be a c-minor of $M$ having $N$ as a c-minor and let $(X',Y')$ be a $2$-separation of $M'$ having $X'$ as the $N$-side. 
Assume that, for all elements $y$ of $Y'$, at least one of $M'\ba y$ and $M'/y$ does not have a special $N$-minor. Then either 
\begin{itemize}
\item[(i)] $\sqcap_{M'}(\{y\},X') = 1$ for all $y$ in $Y'$; or 
\item[(ii)] $\sqcap_{M'}(Y'-y,X') = 0$ for all $y$ in $Y'$.
\end{itemize}
\end{lemma}

\begin{proof} Suppose $y \in Y'$. If $\sqcap_{M'}(\{y\},X) = 0$, then, by Lemma~\ref{obs1}, $\sqcap_{M'/y}(X',Y'-y) = 1$, so, by Lemma~\ref{claim1}(ii), $M'/y$ has a special $N$-minor. If $\sqcap_{M'}(Y'-y,X') = 1$, then, by Lemma~\ref{claim1}(i), $M'\ba y$ has a special $N$-minor. By hypothesis, $M'\ba y$ or $M'/y$ has no special $N$-minor. 
We deduce the following. 

\begin{sublemma}
\label{dich1}
Either $\sqcap_{M'}(\{y\},X') = 1$ or $\sqcap_{M'}(Y'-y,X') = 0$. 
\end{sublemma}

Next we show that all the elements of $Y'$ behave similarly. 

\begin{sublemma}
\label{dich2}
If $\sqcap_{M'}(\{y\},X') = 1$, then $\sqcap_{M'}(\{z\},X) = 1$ for all $z$ in $Y'$.
\end{sublemma}

To see this, note first that $M' = M'_{X'} \oplus_2 M'_{Y'}$. Since $\sqcap_{M'}(\{y\},X') = 1$, it follows that $p \in \cl_{M'_{Y'}}(\{y\})$. Suppose $z \in Y'-y$. Then $p \in \cl_{M'_{Y'}}(Y' - z)$. Hence 
$\sqcap_{M'}(X',Y'-z) = 1$ so $M'\ba z$ has a special $N$-minor. Thus $M'/z$ does not have a special $N$-minor. Hence, by Lemma~\ref{claim1}(ii), 
$\sqcap_{M'/z}(X',Y'-z) = 0$, so, by Lemma~\ref{obs1}, $\sqcap_{M'}(X',\{z\}) = 1$, and \ref{dich2} holds.

Now suppose that $\sqcap_{M'}(\{y\},X') = 0$. Then, by \ref{dich2}, $\sqcap_{M'}(\{z\},X') = 0$ for all $z$ in $Y'$. Thus $M'/z$ has a special $N$-minor for all $z$ in $Y'$. The hypothesis implies that $M'\ba z$ has no special $N$-minor for all $z$ in $Y'$. Then, by Lemma~\ref{claim1}(i), $\sqcap_{M'}(Y'-z,X') = 0$ and the lemma follows.
\end{proof}

The next lemma describes what happens when (i) of Lemma~\ref{dichotomy} holds. 

\begin{lemma}
\label{p63rev}  
Suppose $M\baba \ell$ has $N$ as a c-minor. Let $(X,Y)$ be a $2$-separation of $M\ba \ell$ in which $X$ is the $N$-side and $|Y| \ge 3$. Then 
\begin{itemize}
\item[(i)] $Y$ contains a doubly labelled element; or 
\item[(ii)] $\sqcap(\{y\},X) \neq 1$ for some $y$ in $Y$; or 
\item[(iii)] $Y$ contains an element $y$ such that $M\baba y$ has $N$ as a c-minor and every non-trivial $2$-separation of $M\ba y$ has the form $(Z_1,Z_2)$ where $Z_1$ is the $N$-side and $Z_2 \subseteq Y - y$.
\end{itemize}
\end{lemma}

\begin{proof} Suppose that $\sqcap(\{y\},X) = 1$ for all $y$ in $Y$ and that $Y$ does not contain any doubly labelled elements. As usual, we write  $M\ba \ell$ as the $2$-sum with basepoint  $p$ of the $2$-polymatroids $M_X$ and $M_Y$ having ground sets $X \cup p$ and $Y \cup p$, respectively. First we show that 

\begin{sublemma}
\label{nopoint}
$Y$ does not contain a point.
\end{sublemma}

Assume that $Y$ does contain  a point, $z$. Then, since $\sqcap(\{z\},X) = 1$, we see that $z$ is parallel to $p$ in $M_Y$. 
By Proposition~\ref{connconn}, $M\ba \ell\ba z$ is $2$-connected. Hence $M\ba z$ is $2$-connected. Also, in $M_X$ and $M_Y$, the sets $X$ and $Y-z$ span $p$, and hence span $z$. We show next that 

\begin{sublemma}
\label{nopointsub}
$M\ba z$ is $3$-connected.
\end{sublemma}

Suppose that $M\ba z$ has a $2$-separation $(R,B)$ where $\ell \in R$. Then $(R-\ell,B)$ is 2-separating in $M\ba z \ba \ell$. Note that $r(M\ba \ell) = r(M)$, so $r(M\ba \ell\ba z) = r(M)$. We have 
$$r(R) + r(B) = r(M\ba z) + 1.$$ 
Thus 
$$r(R - \ell) + r(B) \le r(M\ba z,\ell) +1.$$ 
Now $R \neq \{\ell\}$ otherwise $Y - z \subseteq B$ and we obtain the contradiction that $(R,B \cup z)$ is a $2$-separation of $M$. Observe that, since $M\ba \ell\ba z$ is $2$-connected, $r(R - \ell) = r(R)$. As $M$ is $3$-connected, neither $B$ nor $R - \ell$ spans $z$. Thus neither $X$ nor $Y - z$ is contained in $B$ or $R - \ell$. Hence $(X,Y-z)$ and $(R - \ell,B)$ cross. 

Now $\lambda_{M\ba \ell \ba z}(Y-z) = \lambda_{M\ba \ell}(Y) = 1$ and $\lambda_{M\ba \ell \ba z}(B) = 1$. Thus, by uncrossing, 
$\lambda_{M\ba \ell \ba z}(B\cap (Y-z)) = 1$. Since $\ell \in \cl(R - \ell)$ and $z \in \cl(X)$, we deduce that $\lambda_M(B \cap (Y-z)) = 1$. As $M$ is $3$-connected, it follows that $B \cap (Y - z)$ consists of a single point $y$. Then, by assumption, $\sqcap(X,\{y\}) = 1$. But $\sqcap(X,\{z\}) = 1$. Thus $y$ is parallel to $p$ in $M_Y$. Hence $y$ and $z$ are parallel points in $M$; a contradiction. We conclude that \ref{nopointsub} holds. 

To complete the proof of \ref{nopoint}, we shall show that $M\baba z$ has a special $N$-minor. We know that $M\ba \ell = M_X \oplus_2 M_Y$ where $z$ is parallel in $M_Y$ to the basepoint $p$ of the $2$-sum. Moreover, by Lemma~\ref{p69}, $M_X$ has a special $N$-minor. 
Now $M\ba \ell\ba z$ is $2$-connected and, by \cite[Proposition 3.1]{hall}, $M\ba \ell\ba z = M_X \oplus_2 (M_Y \ba z)$. 
Hence $M\ba z$ has a special $N$-minor. Thus $M\baba z$ is $3$-connected having a c-minor isomorphic to $N$; a contradiction. We deduce that \ref{nopoint} holds.

We now know that every element of $Y$ is a line $y$ with $\sqcap(X,\{y\}) = 1$. Hence, in $M_Y$, the basepoint $p$ lies on $y$. Thus, for all $y$ in $Y$, we see that $M\ba \ell\ba y$ is $2$-connected. Then, 
by Lemma~\ref{p49} again,  we deduce that 

\begin{sublemma}
\label{usey}
for all $y$ in $Y$, both $M\ba \ell\ba y$ and $M\ba y$ have   special $N$-minors.
\end{sublemma}

Since every line in $Y$ contains $p$, it follows that $M_Y/p$ is a matroid. Next we show that 

\begin{sublemma}
\label{caseD1.1}
$M_Y/p$ has a circuit.
\end{sublemma}

Assume that $M_Y/p$ has no circuits. Let $y$ and $y'$ be two distinct elements of $Y$. Then $r(X \cup (Y - \{y,y'\})) = r(X) + |Y - \{y,y'\}|$ and 
$r(X \cup Y) = r(X) + |Y|$. 
As a step towards \ref{caseD1.1}, we show that 

\begin{sublemma}
\label{caseD1.1sub}
$\sqcap((X\cup Y) - \{y,y'\},\{\ell\}) = 0$. 
\end{sublemma}

Suppose that $\sqcap((X\cup Y) - \{y,y'\},\{\ell\}) \ge 1$. 
Then, as $r(Y) = |Y| + 1$, 
\begin{align*}
\lambda_M(\{y,y'\}) & = r(X \cup (Y - \{y,y'\})\cup \ell) + r(\{y,y'\}) - r(M)\\
& \le r(X) + |Y - \{y,y'\}| +1 + 3 - r(M)\\
& = r(X) + r(Y) - r(M) + 1 = 2.
\end{align*}
As $M$ is $3$-connected, we see that $\lambda_M(\{y,y'\}) = 2$, so equality holds thoughout the last chain of inequalities. Thus $\{y,y'\}$ is a prickly 3-separator of $M$ and $\lambda_{M\ba \ell}(\{y,y'\}) = 1$. By Lemma~\ref{portia}, $M\da y$ is $3$-connected. 
By Lemma~\ref{dennisplus}(vi), $(M\ba \ell)\da y =  M_X \oplus_2 (M_Y \da y)$.  Thus $\sqcap_{M\ba \ell \da y}(X,Y-y) = 1$ so, by Lemma~\ref{claim1}(iii), $(M\da y)\ba \ell$, and hence $M\da y$, has a special $N$-minor. This contradiction 
 implies that \ref{caseD1.1sub} holds for all distinct $y$ and $y'$ in $Y$.


As the next step towards proving \ref{caseD1.1}, we now show that 

\begin{sublemma}
\label{lcn}
$M/ \ell$ has a c-minor isomorphic to $N$.
\end{sublemma}

In $M\ba \ell$, deleting all but one element, $y$, of $Y$ leaves the $2$-polymatroid that, when $y$ is compactified, equals $M_X$ with $p$ relabelled as $y$. Hence $M\ba \ell \ba (Y - y)$ has a c-minor isomorphic to $N$. By \ref{caseD1.1sub}, since $|Y| \ge 3$, we deduce that $\{\ell\}$ is 1-separating in $M\ba (Y - y)$. Hence $M\ba (Y - y) \ba \ell = M\ba (Y - y) / \ell$, so, by \ref{usey}, we deduce that \ref{lcn} holds. 

Still continuing towards the proof of  \ref{caseD1.1}, next  we observe that

\begin{sublemma}
\label{linear}
$\ell$ is a line of $M$.
\end{sublemma}

Suppose $\ell$ is a point. By Lemma~\ref{newbix}, $M/ \ell$ is $2$-connected having one side of every $2$-separation being a pair of points of $M$ that are parallel in $M/ \ell$. By \ref{lcn}, $M$ must have such a pair $\{u,v\}$ of points. Then both $M\ba u$ and $M\ba v$ have c-minors isomorphic to $N$. By \cite[Lemma 4.2]{oswww}, $M$ has a triad   of points containing $\ell$ and one of $u$ and $v$, say $u$. Let $w$ be the third point in this triad. Then $M\ba \ell$ has $\{u,w\}$ as a series pair of points, so $M\ba \ell/u$, and hence $M/u$, has a c-minor isomorphic to $N$. Thus the point $u$ contradicts  Lemma~\ref{Step0}.

By \ref{lcn}, $M/ \ell$ has a 2-separation $(U,V)$. Thus $r(U \cup \ell) + r(V \cup \ell) - r(M) = 3$. By symmetry, we may assume that 
$U \subseteq (X \cup Y) - \{y,y'\}$ for some $y'$ in $Y - y$. Then, by \ref{caseD1.1sub} and \ref{linear}, $r(U \cup \ell) = r(U) + 2$. Hence $(U, V \cup \ell)$ is a \tws\ of $M$. This \cn\ completes the proof of \ref{caseD1.1}.

Choose $y$ in $Y$ such that $y$ is in a circuit of $M_Y/p$ and $y \in E(M) - E(N)$. By \ref{usey}, $M\ba y$ has a special $N$-minor. Thus, by Lemma~\ref{predichotomy}, $M\ba y$ has $N$ as a c-minor. Now $r(M\ba \ell\ba y) = r(M\ba \ell) = r(M) = r(M\ba y)$. Hence $\ell \in \cl_{M\ba y}(X \cup (Y - y))$ and $M\ba \ell \ba y$ is $2$-connected. Next we show the following.

\begin{sublemma}
\label{caseD1.2}
Every non-trivial $2$-separation of $M\ba y$ has the form $(X \cup Y' \cup \ell, Y'')$ where $Y'$ and $Y''$ are disjoint and $Y' \cup Y'' = Y-y$.
\end{sublemma}

Let $(A,B)$ be a non-trivial $2$-separation of $M\ba y$ that is not in the stated form. Without loss of generality, $\ell \in A$. Then $X \not \subseteq A$. Since $M\ba \ell\ba y$ is $2$-connected having the same rank as $M\ba y$, it follows that $r(A - \ell) = r(A)$ and $(A- \ell,B)$ is a 2-separation of $M\ba \ell\ba y$. We also know that $(X,Y-y)$ is a 2-separation of $M\ba \ell \ba y$. Now $\ell \not\in \cl(X)$ and $\ell \not\in \cl(Y-y)$. But $\ell \in \cl(A - \ell)$,  so $(A - \ell) \cap (Y -y) \neq \emptyset  \neq (A - \ell) \cap X$. By uncrossing, $\lambda_{M\ba \ell\ba y}(B \cap X) = 1$. As $\ell \in \cl(A - \ell)$ and $y \in \cl(Y-y)$, we deduce that $\lambda_M(B\cap X) = 1$. Thus $B \cap X$ consists of a single point $x$ of $M$. Then $B \cap (Y-y) \neq \emptyset$. Therefore, by uncrossing again, $\lambda_{M\ba \ell\ba y}(X \cap (A - \ell)) = 1$, so $\lambda_{M\ba \ell}(X \cap (A - \ell)) = 1$. Thus $(X-x,Y\cup x)$ is a 2-separation of $M\ba \ell$. If $r(Y \cup x) = r(Y)$, then $x$ is parallel to $p$ in $M_X$. Hence,  we see that $x$ lies on $y$. Then $M\ba x$ is $3$-connected having a special $N$-minor; a contradiction. Thus we may assume that $r(Y \cup x)  = r(Y) + 1$. Then $r(X- x) = r(X) - 1$. Hence, in $M_X$, the points $p$ and $x$ are a series pair. Thus $M_X$ is the 2-sum with basepoint $q$ of a $2$-polymatroid $M'_X$, say, and a copy of $U_{2,3}$ with ground set $\{q,p,x\}$. Moreover, every element of $Y$ is a line through $p$ in $M_Y$. Thus  we see that both $M\ba y$ and $M/y$ have   special $N$-minors; a contradiction. We conclude that \ref{caseD1.2} holds,  so (iii) of the lemma holds, and  the proof of the lemma is complete. 
\end{proof}

\begin{lemma}
\label{prep65rev}  
Suppose $M\baba \ell$ has $N$ as a c-minor. Let $(X,Y)$ be a $2$-separation of $M\ba \ell$ in which $X$ is the $N$-side and $|Y| \ge 3$. 
Let $M_X \oplus_2 M_Y$ be the associated $2$-sum decomposition of $M\ba \ell$ with respect to the basepoint $p$. 
Then 
\begin{itemize}
\item[(i)] $Y$ contains a doubly labelled element; or 
\item[(ii)] $\sqcap(Y-y,X) > 0$ for some $y$ in $Y$; or 
\item[(iii)] $r(X\cup \ell \cup y_0) > r(X \cup y_0)$ for some $y_0$ in $Y$, and $M/y_0$ has a special $N$-minor. Moreover,  either
	\begin{itemize} 
	\item[(a)]  every non-trivial $2$-separation of $M/y_0$ has the form $(Z_1,Z_2)$ where $Z_1$ is the $N$-side and $Z_2 \subseteq Y - y_0$; or
	\item[(b)]   $M_X$ is the $2$-sum with basepoint $q$ of two  $2$-polymatroids, one of which is a copy of $U_{2,3}$ with ground set $\{p,z,q\}$.
	\end{itemize}
	\end{itemize}
\end{lemma}

\begin{proof} Assume that neither (i) nor (ii) holds. Suppose $y \in Y$. As $\sqcap(Y,X) = 1$, it follows that $r(Y) > r(Y-y)$ so
\begin{sublemma}
\label{*1}
$r(Y-y) \le r(Y) - 1.$
\end{sublemma}

Next we show that 
\begin{sublemma}
\label{MYy}
$\lambda_{M_Y}(\{y\}) = \lambda_{M\ba (X \cup \ell)}(\{y\}) + 1.$
\end{sublemma}

We see that $\lambda_{M_Y}(\{y\}) = r_M(\{y\}) + r_{M_Y}((Y-y) \cup p) - r(M_Y).$  Since $\sqcap(Y-y, X) = 0$, we deduce that 
$r_{M_Y}((Y-y) \cup p) = r_M(Y-y) + 1.$ As $M_Y$ is $2$-connected, $r(Y) = r(M_Y)$ and \ref{MYy} follows.

We now extend \ref{*1} as follows. 

\begin{sublemma}
\label{*2}
Let $\{y_1,y_2,\dots,y_k\}$ be a subset of $Y$. Then 
$$r(Y - \{y_1,y_2,\dots,y_k\}) \le r(Y) - k.$$
\end{sublemma}

By \ref{*1}, $r(Y-y_1) \le r(Y) - 1$ and $r(Y-y_2) \le r(Y) - 1$. Thus, by submodularity, $r(Y-\{y_1,y_2\}) \le r(Y) - 2$. Repeating this argument gives \ref{*2}.

Next we show the following.

\begin{sublemma}
\label{mcony}
For all $y$ in $Y$, the $2$-polymatroid $M\ba \ell/y$ has a special $N$-minor and $\lambda_{M\ba \ell/y}(X) = 1$.
\end{sublemma}

Let $M' = M\ba \ell$. By Corollary~\ref{general2}, 

\begin{align}
\label{Yyy}
\lambda_{M'/y}(X) & = \lambda_{M'\ba y}(X) - \sqcap_{M'}(X,y) -  \sqcap_{M'}(Y-y,y) + r(\{y\}) \nonumber \\
& =  \lambda_{M'\ba y}(X)  - r_{M'}(Y-y) + r_{M'}(Y) \text{~as $\sqcap(X,Y-y') = 0$ for all $y'$ in $Y$;} \nonumber\\
& =  r_{M'}(Y) -   r_{M'}(Y-y).
\end{align} 
But 
\begin{align*}
1 & = \lambda_{M'}(X)\\
 & = r(X) + r(Y) - r(M')\\
 & \ge r(X \cup y) - r(\{y\}) + r(Y) - r(\{y\}) - r(M') + r(\{y\})\\
 & = \lambda_{M'/y}(X).
\end{align*} 
We conclude, using (\ref{Yyy}) that, since $r(Y) \neq r(Y-y)$, we have $\lambda_{M'/y}(X) = 1$ for all $y$ in $Y$. Then, by Lemma~\ref{claim1}(ii), $M'/y$ has a special $N$-minor. Hence $M\ba \ell /y$ has a special $N$-minor, that is, \ref{mcony} holds.

\begin{sublemma}
\label{neworg}
If $y \in Y$ and $\ell$ is in a parallel pair of  points in $M/y$, then $r(X \cup \ell \cup y) = r(X \cup y)$.
\end{sublemma}

To see this, observe that, as $M$ is $3$-connected, $\ell \not\in \cl_M(Y)$. Thus $\ell$ is parallel to a point of $X$ in $M/y$, and \ref{neworg} follows.

\begin{sublemma}
\label{subsume}
Let $Y = \{y_1,y_2,\dots,y_n\}$. If $r(X \cup \ell \cup y_i) = r(X \cup y_i)$ for all $i$ in $\{1,2,\dots,n\}$, then $\{y_{n-1},y_n\}$ is a prickly $3$-separator of $M$, and $M\da y_n$ is $3$-connected having a special $N$-minor. 
\end{sublemma}

First observe that each $y_i$ in $Y$ is a line for if $y_i$ is a point, then 
$$r(X \cup \ell \cup y_i) = r(X \cup y_i) = r(X) + r(\{y_i\}) = r(X) + 1.$$
As $r(Y-y_i) \le r(Y) - 1$, we deduce that $(X \cup \ell \cup y_i,Y-y_i)$ is a 2-separation of $M$; a contradiction.

Continuing with the proof of \ref{subsume}, next we show the following.

\begin{sublemma}
\label{subsume2} 
For $1 \le k \le n-1$, 
\begin{align*}
r(X \cup \ell \cup \{y_1,y_2,\dots,y_k\}) & = r(X) + 1 + k \text{~~and}\\
r(Y -  \{y_1,y_2,\dots,y_k\}) & = r(Y) - k.
\end{align*}
\end{sublemma}

We argue by induction on $k$. By assumption, $r(X \cup \ell \cup y_1) = r(X) + r(\{y_1\}) = r(X) + 2$. Moreover, $r(Y-y_1) \le r(Y) - 1$. Equality must hold otherwise we get the contradiction that $(X \cup \ell \cup y_1, Y- y_1)$ is a 2-separation of $M$. We deduce that the result holds for $k = 1$. Assume it holds for $k < m$ and let $k = m\ge 2$. Then

\begin{multline*}
r(X \cup \{y_1,y_2,\dots,y_{m-1}\} \cup \ell) + r(X \cup \{y_2,y_3,\dots,y_{m}\} \cup \ell)\\
\shoveleft{\hspace*{1in}\ge r(X \cup \{y_2,y_3,\dots,y_{m-1}\} \cup \ell) + r(X \cup \{y_1,y_2,\dots,y_{m}\} \cup \ell).}
\end{multline*}
If $m = 2$, then $r(X \cup \{y_2,y_3,\dots,y_{m-1}\} \cup \ell) = r(X \cup \ell) \ge r(X) + 1$. If $m>2$, then 
$r(X \cup \{y_2,y_3,\dots,y_{m-1}\} \cup \ell) = r(X) + m - 1$ by the induction assumption. Thus
\begin{align}
\label{eq1}
r(X  \cup \{y_1,y_2,\dots,y_m\}\cup \ell) & \le r(X) + m + r(X) + m - (r(X) + m-1) \nonumber \\
 & = r(X) + m + 1.
\end{align}
But 
\begin{equation}
\label{eq2}
r(Y - \{y_1,y_2,\dots,y_m\}) \le r(Y) - m.
\end{equation}
It follows that equality must hold in (\ref{eq1}) and (\ref{eq2}). Thus, by induction, \ref{subsume2} holds.

By \ref{subsume2}, $r(Y - \{y_1,y_2,\dots,y_{n-1}\}) = r(Y) - (n-1).$ But $r(Y - \{y_1,y_2,\dots,y_{n-1}\}) = r(\{y_n\}) = 2$. Thus $r(Y) = n+1$, 
 and it follows by \ref{subsume2} that $r(\{y_{n-1},y_n\}) = 3$ and $\{y_{n-1},y_n\}$ is a prickly 3-separating set in $M$. Hence, by Lemma~\ref{portia}, 
$M\da y_n$ is $3$-connected. Recall that 
\begin{equation*} 
r_{M\downarrow y_n}(Z) = 
\begin{cases} 
r(Z), & \text{if   $r(Z \cup y_n) > r(Z)$;   and}\\
r(Z) - 1, & \text{otherwise.}
\end{cases}
\end{equation*} 
Thus 
\begin{align*}
\sqcap_{M\da y_n}(X,Y-y_n) & = r_{M\downarrow y_n}(X) + r_{M\downarrow y_n}(Y -y_n) - r_{M\downarrow y_n}(X \cup (Y-y_n))\\
 & = r(X) + r(Y-y_n) - r(M) + 1\\
 & = r(X) + r(Y) - r(M) \text{~~by \ref{subsume2};}\\
 & = 1.
\end{align*} 
It follows by Lemma~\ref{dennisplus}(vi) that 
$(M\ba \ell)\da y_n =  M_X \oplus_2 M_Y \da y_n$. Then, by Lemma~\ref{claim1}(iii), $(M\ba \ell)\da y_n$ has a special $N$-minor. We deduce that $M\da y_n$ is $3$-connected having a special $N$-minor. Thus \ref{subsume} holds. 

Since we have assumed that the theorem fails, it follows by \ref{subsume} that, for some element $y_0$ of $Y$,
$$r(X \cup \ell \cup y_0) > r(X \cup y_0).$$ 
By \ref{mcony}, $M/y_0$ has a special $N$-minor. 
Thus $M/y_0$ is not \thc. 
Moreover, by \ref{neworg}, the element $\ell$ is not in a pair of parallel points of $M/y_0$. 

Let $(A \cup \ell, B)$ be a \tws\ of $M/y_0$ with $\ell \not\in A$. 
Next we show that 
\begin{sublemma}
\label{2sepab}
$(A,B)$ is an exact $2$-separation of $M/y_0\ba \ell$, and $\ell \in \cl_{M/y_0}(A).$
\end{sublemma}

If $(A,B)$ is not exactly 2-separating in $M/y_0\ba \ell$, then, by Proposition~\ref{connconn}, $M_Y/y_0$ is not $2$-connected, so we obtain the \cn\ that $Y$ contains a  doubly labelled element. Thus $r_{M/y_0}(A \cup \ell) = r_{M/y_0}(A)$ and 
\ref{2sepab} holds. 

We shall show that 

\begin{sublemma}
\label{crosspath}
either (iii)(b) holds, or $(A,B)$ does not cross $(X,Y-y_0)$.
\end{sublemma}

Assume each of $A$ and $B$ meets each of $X$ and $Y - y_0$. Then, by uncrossing, $\lambda_{M\ba \ell/y_0}(X \cap B) = 1$. But $\sqcap(X,\{y_0\}) = 0$, so $r_M(X \cap B) = r_{M/y_0}(X \cap B).$ Also $r_M((Y-y_0) \cup A \cup \ell \cup y_0) = r_{M/y_0}((Y-y_0) \cup A \cup \ell) + r(\{y_0\}).$ Then 
\begin{multline*}
r(X\cap B)) + r((Y-y_0) \cup A \cup \ell \cup y_0) - r(M)\\
\shoveleft{= r_{M/y_0}(X \cap B) + r_{M/y_0}((Y-y_0) \cup A \cup \ell) + r(\{y_0\}) - r(M/y_0) - r(\{y_0\})}\\
\shoveleft{= \lambda_{M/y_0}(X \cap B)}\\
\shoveleft{= \lambda_{M/y_0\ba \ell}(X \cap B) \text{~~ as $\ell \in \cl_{M/y_0}(A)$;}}\\
\shoveleft{= 1.}\\
\end{multline*}
Since $M$ is \thc, it follows that $X \cap B$ consists of a point $z$ of $M$.  

Now $\lambda_{M\ba \ell /y_0}((Y- y_0) \cup z) = 1$, so 
\begin{align*}
1 & = r_{M/y_0}((Y-y_0) \cup z) + r_{M/y_0}(A \cap X) - r(M/y_0)\\
 & = r(Y \cup z) - r(\{y_0\}) + r((A \cap X)\cup y_0)  - r(\{y_0\}) - r(M) + r(\{y_0\})\\
  & = r(Y \cup z)) + r(A \cap X) - r(M\ba \ell) \text{~~since $\sqcap(X,\{y_0\}) = 0$.}\\
\end{align*} 
Thus $Y \cup z$ is $2$-separating in $M\ba \ell$. If $r(Y \cup z) = r(Y)$, then $z$ is parallel to the basepoint $p$ of the 2-sum. Hence each element of $Y$ is doubly labelled; a \cn. Thus we may assume that $r(Y \cup z) = r(Y) + 1$. Then $r(X- z) = r(X) - 1$. Now $M_X$ is $2$-connected, so $r(M_X) = r(X)$ and $M_X$ has $\{p,z\}$ as a series pair of points. It follows that $M_X$ is the 2-sum with basepoint $q$ of a $2$-polymatroid $M'_X$ and a copy of $U_{2,3}$ with ground set $\{q,z,p\}$. Thus (iii)(b) of the lemma holds. Hence so does \ref{crosspath}.

We shall now assume that (iii)(b) does not hold. 

\begin{sublemma}
\label{notsubset}
$A\not \subseteq Y - y_0$ and $B\not \subseteq X$ and $A\not \subseteq X$.
\end{sublemma}

To see this, first suppose that $A \subseteq Y - y_0$. Then, as $\ell \in \cl_{M/y_0}(A)$, we deduce that $\ell \in \cl_M(Y)$; a \cn. Thus $A\not \subseteq Y - y_0$.

Now suppose that $B \subseteq X$. We have 
\begin{align*}
1 & = \lambda_{M/y_0}(B)\\
& = r_{M/y_0}(B)   + r_{M/y_0}(A \cup \ell) - r(M/y_0)\\
 & = r(B \cup y_0) - r(\{y_0\}) + r(A \cup \ell \cup y_0)  - r(\{y_0\}) - r(M) +  r(\{y_0\})\\
 & = r(B)   + r(A \cup \ell \cup y_0)    - r(M)  \text{~~as $B \subseteq X$.}
\end{align*} 
Thus $(A \cup \ell \cup y_0,B)$ is a \tws\ of $M$; a \cn. Thus $B\not\subseteq X$. 

Next suppose that $A \subseteq X$. As $(A \cup \ell,B)$ is a \tws\ of $M/y_0$, we have  
\begin{align*}
1 & =  r_{M/y_0}(A \cup \ell)   + r_{M/y_0}(B) - r(M/y_0)\\
& = r(A \cup \ell \cup y_0)  - r(\{y_0\}) + r(B \cup y_0) - r(\{y_0\}) - r(M) +  r(\{y_0\})\\
& \ge r(A \cup y_0) - r(\{y_0\}) + r(B  \cup y_0)   - r(M)\\
& \ge r(A)  + r(B  \cup y_0)   - r(M\ba \ell) \text{~~as $A \subseteq X$;}\\
& \ge 1 \text{~~as $M\ba \ell$ is $2$-connected.}
\end{align*} 
We deduce that equality holds throughout, so $r(A \cup \ell \cup y_0) = r(A \cup y_0)$. But $A \subseteq X$, so $r(X \cup \ell \cup y_0) = r(X \cup y_0)$, which contradicts the choice of $y_0$. Hence $A\not\subseteq X$, so \ref{notsubset} holds.

By \ref{crosspath}, we deduce that $B \subseteq Y - y_0$. Since, by \ref{mcony}, $M/y_0$ has a special $N$-minor, we see that (iii)(a) of the lemma holds, so the lemma is proved.
\end{proof}

We now combine the above lemmas to prove one of the two main results of this section. 

\begin{lemma}
\label{bubbly}
Suppose $M\ba \ell$ has $N$ as a c-minor. Let $(X,Y)$ be a $2$-separation of $M\ba \ell$ having $X$ as the $N$-side and $|Y| = \mu(\ell)$. Then $Y$ contains a doubly labelled element.
\end{lemma}

\begin{proof}
By Lemma~\ref{muend}, $|Y| \ge 3$. Assume that $Y$ does not contain a doubly labelled element. Then, by Lemma~\ref{p63rev}, 

\begin{itemize}
\item[(i)(a)] $\sqcap(\{y\},X) \neq 1$ for some $y$ in $Y$; or 
\item[(i)(b)] $Y$ contains an element  $y$ such that $M\baba y$ has $N$ as a c-minor and every non-trivial $2$-separation of $M\ba y$ has the form $(Z_1,Z_2)$ where $Z_1$ is the $N$-side and $Z_2 \subseteq Y - y$.
\end{itemize}

Now,  since $|Y| = \mu(\ell)$, outcome (iii)(b) of Lemma~\ref{prep65rev} does not arise. Thus, by that lemma and Lemma~\ref{predichotomy},

\begin{itemize}
\item[(ii)(a)] $\sqcap(Y-y,X) > 0$ for some $y$ in $Y$; or 
\item[(ii)(b)] $Y$ contains an element $y$ such that $M/ y$ has $N$ as a c-minor and every non-trivial $2$-separation of $M/y$ has the form $(Z_1,Z_2)$ where $Z_1$ is the $N$-side and $Z_2 \subseteq Y - y$.
\end{itemize}

By Lemma~\ref{dichotomy}, (i)(a) and (ii)(a) cannot both hold. Thus (i)(b) or (ii)(b) holds. Therefore, for some $y$ in $Y$, either $M\baba y$ has $N$ as a c-minor and has a \tws\ $(Z_1,Z_2)$ where $Z_1$ is the $N$-side, $Z_2 \subseteq Y - y$, and $|Z_2| = \mu(y) < \mu(\ell)$, or  $M/ y$ has $N$ as a c-minor and has a \tws\ $(Z_1,Z_2)$ where $Z_1$ is the $N$-side, $Z_2 \subseteq Y - y$, and $|Z_2| = \mu^*(y) < \mu(\ell)$. We can now repeat the argument above using $(y,Z_2)$ in place of $(\ell,Y)$ and, in the latter case,  $M^*$ in place of $M$. Since we have eliminated the possibility that $\mu(\ell) = 2$ or $\mu^*(\ell) = 2$, after finitely many repetitions of this argument, we obtain a \cn\ that completes the proof. 
\end{proof}
 
\begin{corollary}
\label{doubly}
The $2$-polymatroid $M$ contains a doubly labelled element.
\end{corollary}

\begin{proof} Take $\ell$ in $E(M) - E(N)$. Then $M\ba \ell$ or $M/ \ell$ has $N$ as a c-minor, so  applying the last lemma to $M$ or its dual gives the result.
\end{proof}

\section{Non-$N$-$3$-separators exist}
\label{keylargo}

The purpose of this section is prove the existence of a non-$N$-3-separating set in $M$ where we recall that such a set $Y$ is exactly $3$-separating, meets $E(N)$ in at most one element, and, when it has exactly two elements, both of these elements are lines.   The following lemma will be key in what follows.

\begin{lemma}
\label{key} 
Let $(X,Y)$ be a $2$-separation of $M\ba \ell$ where $X$ is the $N$-side, $|Y| \ge 2$, and $Y$ is not a series pair of points in $M\ba \ell$. Then $Y$ contains no points. 
\end{lemma}

\begin{proof}
Assume that $Y$ contains a point $y$. Then, by Lemma~\ref{Step0}, $y$ is not doubly labelled.

\begin{sublemma}
\label{key1}
$M\ba y$ or $M/y$ has a special $N$-minor.
\end{sublemma}

To see this, consider the $2$-connected $2$-polymatroid $M_Y$. By Lemma~\ref{Tutte2}, $M_Y\ba y$ or $M_Y/y$ is $2$-connected, so $\sqcap_{M\ba y}(X,Y-y) = 1$ or $\sqcap_{M/ y}(X,Y-y) = 1$. As $M_X$ has a special $N$-minor, so does $M\ba y$ or $M/y$.

\begin{sublemma}
\label{key2}
$M\ba y$ does not have  a special $N$-minor.
\end{sublemma}

Assume $M\ba y$ does have a special $N$-minor. Then, as $y$ is not doubly labelled, $M/y$ does not have a special $N$-minor. Then, by Lemma~\ref{claim1}(ii), $\sqcap_{M/y}(X,Y-y)  = 0$, that is, $r_{M/y}(X) + r_{M/y}(Y-y) - r(M/y) = 0$, so 
$r(X \cup y) + r(Y) = r(M) + r(\{y\}) = r(M) + 1$. But $r(X) + r(Y) = r(M) +1$, so $r(X \cup y) = r(X)$ and $r(Y - y) = r(Y)$ otherwise $(X\cup y, Y-y)$ is a 1-separation of $M\ba \ell$; a \cn\ to Lemma~\ref{Step1}. Since $y \in Y$ and $r(X \cup y) = r(X)$, we see that $\sqcap(X,\{y\}) = 1$. But $\sqcap(X,Y) = 1$. Thus, in $M_Y$, the point $y$ is parallel to the basepoint $p$ of the 2-sum. Hence $M\ba \ell \ba y$ is $2$-connected and $r(M\ba \ell \ba y) = r(M)$.

Let $(A \cup \ell, B)$ be a non-trivial \tws\ of $M\ba y$ where $\ell \not \in A$. Now 
\begin{align*}
1 & \le r(A) + r(B) - r(M\ba \ell,y)\\
& \le r(A \cup \ell) + r(B) - r(M\ba y)\\
& = 1.
\end{align*}
Thus $r(A) = r(A \cup \ell)$. Hence $\ell \in \cl(A)$ so $r(A) \ge 2$. Continuing with the proof of \ref{key2}, we now show the following.

\begin{sublemma}
\label{key3}
$(A,B)$ crosses $(X,Y-y)$.
\end{sublemma}

Because $y \in \cl(X) \cap \cl(Y-y)$ but $y \notin \cl(A) \cup \cl(B)$, we deduce that neither $A$ nor $B$ contains $X$ or $Y-y$, so \ref{key3} holds.

By uncrossing,
$\lambda_{M\ba \ell,y}(B \cap (Y-y)) = 1$. But $\ell \in \cl(A)$ and $y \in \cl(X)$ so $\lambda_M(B \cap (Y - y)) = 1$. Hence $B \cap (Y - y)$ consists of a single point, say $z$. As $z$ is not parallel to $y$, we deduce that $\sqcap(X,\{z\}) = 0$. Thus, by Lemma~\ref{obs1}, 
$\sqcap_{M/z}(X,Y-z) = \lambda_{M\ba \ell/z}(X) = 1$. Hence, by Lemma~\ref{claim1}(ii), $M\ba \ell/z$, and hence $M/z$, has a special $N$-minor. On the other hand, 
$$1 = \sqcap(X,\{y\}) \le \sqcap(X,Y-z) \le \sqcap(X,Y) = 1.$$
Thus $\sqcap_{M\ba z}(X,Y-z) = 1$ so $M\ba z$ has a special  $N$-minor. Since $z$ is a point, we have a \cn\ to Lemma~\ref{Step0} that proves \ref{key2}.

By combining \ref{key1} and \ref{key2}, we deduce that $M/y$ has a special $N$-minor but $M\ba y$ does not. Since $(M,N)$ is a counterexample, $M/y$ is not \thc. By Lemma~\ref{Step1}, $M/y$ is $2$-connected.

As $M\ba y$ does not have a special $N$-minor, by Lemma~\ref{claim1}(i), $\sqcap(X,Y-y) = 0$. But $\sqcap(X,Y) = 1$. As $y$ is a point, it follows that $$r(Y-y) = r(Y) - 1$$ 
and $r(X \cup (Y-y)) = r(X \cup Y)$.  
Moreover, as $(X \cup y,Y-y)$ is not a 1-separation of $M\ba \ell$, we deduce that 
\begin{sublemma}
\label{xy1}
$r(X \cup y) = r(X) + 1.$
\end{sublemma}

Now $r(M_Y \ba p,y) = r(Y - y) = r(Y) - 1$. But $r(M_Y \ba p) = r(Y)$. If $r(M_Y\ba y) = r(Y) - 1$, then $\{y\}$ is a 1-separating set  in $M_Y$. We deduce that $\{p,y\}$ is a series pair of points in $M_Y$. Thus  $M_Y\ba y$ is not 2-connected but $M_Y$ is, so, by Lemma~\ref{Tutte2}, $M_Y/y$ is 2-connected. Hence, by Proposition~\ref{connconn}, $M\ba \ell/y$ is 2-connected.

\begin{sublemma}
\label{notel}
$(\{\ell\}, X \cup (Y-y))$ is not a $2$-separation of $M/y$.  
\end{sublemma}

Assume the contrary. Then 
$r(\{\ell,y\}) + r(X \cup Y) = r(M) + 2.$
But $r_{M/y}(\{\ell\}) = 2$ otherwise we do not have a \tws. Thus $r(\{\ell,y\}) = 3$, so $(\{\ell\}, X \cup Y)$ is a \tws\ of $M$; a \cn. Therefore \ref{notel} holds.

Let $(A\cup \ell,B)$ be a \tws\ of $M/y$ with $\ell$ not in $A$. By \ref{notel}, $A \neq \emptyset.$ Since $M/y \ba \ell$ is 2-connected, $\lambda_{M/y\ba \ell}(A) > 0$. Hence $\lambda_{M/y\ba \ell}(A) = 1$, so $\ell \in \cl_{M/y}(A).$ Hence one easily checks that 
\begin{sublemma}
\label{crank1}
\begin{itemize}
\item[(i)] $r(A \cup y \cup \ell) = r(A \cup y)$; and 
\item[(ii)] $r(A \cup y) + r(B \cup y) = r(M\ba \ell) + 2$.
\end{itemize} 
\end{sublemma}

Next we show that 
\begin{sublemma}
\label{crossagain}
$(A,B)$ crosses $(X,Y-y)$.  
\end{sublemma}

Assume $B \cap (Y - y) = \emptyset$ or $B \cap X = \emptyset$. As $r(X \cup y) = r(X) + 1$ and $r(Y) = r(Y-y) + 1$,   we have $r(B \cup y) = r(B) + 1$. Then, as $r(A\cup y \cup \ell) = r(A \cup y)$, we have, by \ref{crank1},  
$$r(A \cup y \cup \ell) + r(B) = r(M) + 1,$$
that is, $(A \cup y \cup \ell, B)$ is a \tws\ of $M$; a \cn. We deduce that $B \cap (Y-y) \neq \emptyset \neq B \cap X$. 


Now assume that $A \cap (Y-y) = \emptyset.$ Then $A \subseteq X$ and $Y-y \subseteq B$, so $r(X \cup y \cup \ell) = r(X \cup y)$. As 
$r(X \cup y) = r(X) + 1$  and 
$r(Y-y) = r(Y) -1$, it follows that $(X \cup y \cup \ell, Y-y)$ is 2-separating in $M$. Hence $Y-y$ consists of a single point $z$. Now $r(X) + r(Y) = r(M\ba \ell) + 1$, so $r(X) = r(M\ba \ell) - 1$. As $M\ba \ell$ is connected, neither $y$ nor $z$ is in $\cl(X)$ so $\{y,z\}$ is a series pair of points in $M\ba \ell$; a \cn. Hence $A \cap (Y - y) \neq \emptyset$. 

Finally, assume that $X \cap A = \emptyset$. Then $A \subseteq Y -y$, so, as    $r(A \cup y \cup \ell) = r(A \cup y)$, it follows that $r(Y \cup \ell) = r(Y)$, so $(X,Y \cup \ell)$ is a \tws\ of $M$; a \cn. We conclude that \ref{crossagain} holds.

Next we determine the structure of the set $B$.
\begin{sublemma}
\label{whatsb}
In $M$, the set $B$ consists of two points, $x'$ and $y'$, that lie in $B\cap X$ and $B\cap (Y-y)$, respectively.  
\end{sublemma}

By uncrossing, $\lambda_{M\ba \ell/y}(X \cap B) = 1$, so 
$$r((X \cap B) \cup y) + r(A \cup Y) - r(M\ba \ell) = 2.$$ As $X \cap B \subseteq X$, we deduce that $r((X \cap B) \cup y) = r(X \cap B) + 1$. Also $y \in Y$, so $r(A \cup Y) = r(A \cup Y \cup \ell)$. Thus $(X\cap B, A \cup Y \cup \ell)$ is $2$-separating in $M$. Hence $X \cap B$ consists of a point, say $x'$.

By uncrossing again, we see that $\lambda_{M\ba \ell/y}((Y-y) \cap B) = 1$, so
$$r(((Y-y) \cap B) \cup y) + r(A \cup X \cup y) - r(M\ba \ell) = 2.$$
Thus
$$r((Y-y) \cap B) + r(A \cup X \cup y\cup \ell) = r(M) + 1$$ 
since $r(((Y-y) \cap B) \cup y)  = r((Y-y) \cap B) + 1$ and $r(A \cup X \cup y) = r(A \cup X \cup y \cup \ell).$ Hence 
$((Y-y) \cap B, A \cup X \cup y \cup \ell)$ is 2-separating in $M$, so $(Y-y) \cap B$ consists of a single matroid point, $y'$. We deduce that \ref{whatsb} holds.

\begin{sublemma}
\label{doubleup}
The element $y'$ is doubly labelled. 
\end{sublemma}

To see this, first observe that, in $M/y$, the set $B$ is a 2-separating set consisting of two matroid points, $x'$ and $y'$. Suppose $r_{M/y}(B)  = 2$. Then $r_{M/y}(A \cup \ell) = r(M/y) - 1$, so $r(A \cup \ell \cup y) = r(M) -1$. Hence $(A \cup \ell \cup y, \{x',y'\})$ is a \tws\ of $M$; a \cn. We deduce that $r_{M/y}(B) = 1$ so $\{x',y'\}$ is a pair of parallel points in $M/y$. Then   $M/y\ba y'$, and so $M\ba y'$,  has a special $N$-minor.

Now $r_{M/y}(\{x',y'\}) = 1$, so $r(\{x',y',y\}) = 2$. Thus $y \in \cl_{M/y'}(X)$, so $r(X \cup y' \cup y) = r(X \cup y')$. But, by \ref{xy1}, $r(X \cup y) > r(X)$, so $r(X \cup y') > r(X)$. 
 Hence 
$\sqcap(X,\{y'\}) = 0$. Thus, by Lemma~\ref{obs1}, $\sqcap_{M/y'}(X,Y- y') = \sqcap(X,Y) = 1$. We conclude by Lemma~\ref{claim1} that $M/y'$ has a special $N$-minor. Therefore \ref{doubleup} holds.

As \ref{doubleup} contradicts Lemma~\ref{Step0}, we deduce that  Lemma~\ref{key} holds. 
\end{proof}

\begin{lemma}
\label{nonN}
There is a  c-minor $N_0$ of $M$ that is isomorphic to $N$ such that $M$ has a non-$N_0$-$3$-separating set.
\end{lemma}

\begin{proof} By Corollary~\ref{doubly}, $M$ has a doubly labelled element $\ell$. By Lemma~\ref{Step0}, $\ell$ is a line. Moreover, by Lemma~\ref{Step1}, each of $M\ba \ell$ and $M/ \ell$ is 2-connected.

Assume the lemma fails. Let $N_D$ and $N_C$ be special $N$-minors of $M\ba \ell$ and $M/ \ell$, respectively. We now apply what we have learned earlier using $N_D$ in place of $N$. Let $(X,Y)$ be a \tws\ of $M\ba \ell$ in which $X$ is the $N_D$-side and $|Y| = \mu(\ell)$. Then $|Y| \ge 3$. Now $\sqcap(X,\{\ell\}) \in \{0,1\}$.

We show next that
\begin{sublemma}
\label{pixel}
$\sqcap(X,\{\ell\}) = 0$ and $\sqcap(Y,\{\ell\}) = 0$.
\end{sublemma}

Assume that $\sqcap(X,\{\ell\}) = 1$. Then $r(X \cup \ell) = r(X) + 1$, so $\lambda_M(Y) = 2$. Thus $Y$ is a non-$N_D$-3-separating set; a \cn. Thus $\sqcap(X,\{\ell\}) = 0$. Similarly, if $\sqcap(Y,\{\ell\}) = 1$, then $\lambda_M(X)  = 2$, so $Y \cup \ell$  is a non-$N_D$-3-separating set. This \cn\ completes the proof of \ref{pixel}.

We deduce that $M\ba \ell$ has a \tws\ $(D_1,D_2)$ where $D_1$ is the $N_D$-side, $|D_2| = \mu(\ell) \ge 3$, and $\sqcap(D_1,\ell) = 0 = \sqcap(D_2,\ell)$. A similar argument to that used to show \ref{pixel} shows that $M/ \ell$ has a \tws\ $(C_1,C_2)$ where $C_1$ is the $N_C$-side, $|C_2| = \mu^*(\ell) \ge 3$, and $\sqcap(C_1,\ell) = 2 = \sqcap(C_2,\ell)$. We observe here that the definition of $\mu^*(\ell)$ depends on $N_C$ here rather than on $N_D$. 

By the  local connectivity conditions between $\ell$ and each of $D_1,D_2, C_1$, and $C_2$,  

\begin{sublemma}
\label{crosscd}
$(C_1,C_2)$ and $(D_1,D_2)$ cross.
\end{sublemma}

We have 
$r(D_1) + r(D_2) = r(E - \ell ) + 1$ and $r(C_1) + r(C_2) = r(E - \ell) + 3$. By uncrossing, 
$$\lambda_{M\ba \ell}(D_2 \cap C_2) + \lambda_{M\ba \ell}(D_1 \cap C_1) \le 4.$$

Suppose $\lambda_{M\ba \ell}(D_2 \cap C_2) \le 1$.  Since $\ell \in \cl(D_1 \cup C_1)$, it follows that 
$\lambda_{M}(D_2 \cap C_2) \le 1$. Thus $D_2 \cap C_2$ consists of a single point, $z$. Then 
$$2 = \sqcap(C_2,\{\ell\}) \le \sqcap(D_1\cup z, \{\ell\}) \le \sqcap(D_1,\{\ell\}) + 1 = 1;$$
a \cn. We deduce that $\lambda_{M\ba \ell}(D_2 \cap C_2) = 2 =  \lambda_{M\ba \ell}(D_1 \cap C_1)$, so 
$\lambda_{M}(D_2 \cap C_2) = 2 =  \lambda_{M}(D_1 \cap C_1)$. By symmetry, 
$\lambda_{M}(D_1 \cap C_2) = 2 =  \lambda_{M}(D_2 \cap C_1)$. 

Clearly each of $D_2 \cap C_1$  and $D_2 \cap C_2$ contains at most one element of $N_D$. As $|D_2| \ge 3$, we deduce from Lemma~\ref{key} that $D_2$ contains no points. 
Hence, some $Z$ in 
$\{D_2 \cap C_1, D_2 \cap C_2\}$ contains at least two elements.  Then $Z$ is a non-$N_D$-3-separator of $M$.
\end{proof}

For the rest of the proof of Theorem~\ref{modc}, we will use the c-minor $N_0$ of $M$ found in the last lemma. To avoid cluttering the notation, we will relabel $N_0$ as $N$.

\begin{lemma}
\label{p124}
Let $Y_1$ be a minimal non-$N$-$3$-separating set in $M$ with $|Y_1| \ge 3$, and let $X_1 = E(M) - Y_1$. Let $\ell$ be an element of $Y_1$ such that $M\ba \ell$ has $N$ as a c-minor. Let $(A,B)$ be a $2$-separation of $M\ba \ell$ where $A$ is the $N$-side and $|B| = \mu(\ell)$. Then one of the following holds.
\begin{itemize}
\item[(i)] $\lambda_{M\ba \ell}(Y_1 - \ell) = 1$; or 
\item[(ii)] $B \subseteq Y_1 - \ell$; or 
\item[(iii)] $(A,B)$ crosses $(X_1, Y_1 - \ell)$ and $\lambda_{M\ba \ell}(A \cap (Y_1 - \ell)) = 1 = \lambda_{M\ba \ell}(B \cap (Y_1 - \ell))$, while 
$\lambda_{M\ba \ell}(A \cap X_1) = 2 = \lambda_{M\ba \ell}(B \cap X_1) = \lambda_{M\ba \ell}(Y_1 - \ell)$.
\end{itemize}
\end{lemma}

\begin{proof}  Assume neither (i) nor (ii) holds.  Then $\ell \in \cl(Y_1 - \ell)$ and $B \not \subseteq Y_1 - \ell$. If $B \subseteq X_1$, then $\lambda_M(B) = 1$; a \cn. If $B \supseteq Y_1 - \ell$, then $\lambda_M(A) = 1$; a \cn. Finally, observe that $|X_1 \cap A| \ge 2$ since $|E(N)| \ge 4$ and $X_1$ and $A$ are the $N$-sides of their separations. We conclude that $(A,B)$ crosses $(X_1, Y_1 - \ell)$.

By Lemma~\ref{muend}, $|B| \ge 3$. By Lemma~\ref{key}, $B$ contains no points. Now $\lambda_{M\ba \ell}(B \cap X_1) \ge 2$ otherwise, as $\ell \in \cl(Y_1 - \ell)$, we get the \cn\ that 
$\lambda_{M}(B \cap X_1)= 1$. By uncrossing, we deduce that $\lambda_{M\ba \ell}(A \cap (Y_1 - \ell)) \le 1$. Since $|X_1 \cap A| \ge 2$, we get, similarly, that 
$\lambda_{M\ba \ell}(A \cap X_1) \ge 2$, so $\lambda_{M\ba \ell}(B \cap (Y_1- \ell)) \le 1$. As $M\ba \ell$ is $2$-connected, we deduce that 
$\lambda_{M\ba \ell}(A \cap (Y_1 - \ell)) = 1 = \lambda_{M\ba \ell}(B \cap (Y_1 - \ell))$. Hence $\lambda_{M\ba \ell}(A \cap X_1) = 2 = \lambda_{M\ba \ell}(B \cap X_1)$. We conclude that (iii) holds. Hence so does the lemma.
\end{proof}

\section{Finding big enough $3$-separators}
\label{bigtime}

In this section, we first establish (\ref{Step5}) and then we start  the proof of (\ref{Step6}). Specifically, we begin by showing the following.

\begin{lemma}
\label{Step5+} 
$M$ has a minimal non-$N$-$3$-separator with at least three elements.
\end{lemma}





\begin{proof}
Assume every minimal non-$N$-3-separating set has exactly two elements. Let $\{a,b\}$ be such a set, $Z$. Then both of its members are lines. 
We may assume that $b \not\in E(N)$. Suppose first that $r(Z) = 2$. Then $a$ and $b$ are parallel lines.  Suppose that $N$ is a c-minor of $M/b$. Since $a$ is a loop of $M/b$, we deduce that $a \notin E(N)$ so $M\ba a$ has $N$ as a c-minor. Since $M\ba a$ is \thc, this is a \cn. We may now assume that $M\ba b$ has $N$ as a c-minor. Since it is $3$-connected, we have a \cn\ that implies that $r(Z) > 2$.

Suppose next that $r(Z) = 4$. Then $r^*(Z) = ||Z|| + r(E - Z) - r(M)= 4 -2  = 2$. Hence $Z$ consists of a pair of parallel lines in $M^*$, so we obtain a \cn\ as above. We may now assume that $r(Z) = 3$. Then $Z$ is a prickly $3$-separating set and, by Lemma~\ref{portia}, $M\da b$ is \thc. Hence $M\da b$ has no c-minor isomorphic to $N$.

Now $M\ba b$ or $M/b$ has a c-minor isomorphic to $N$. We begin by assuming the former. Let $(S \cup a, T)$ be a non-trivial \tws\ of $M\ba b$ with $a \not\in S$. Suppose the non-$N$-side of  $(S \cup a, T)$ has $\mu(b)$ elements.  By Lemma~\ref{muend}, $\mu(b) \ge 3$. We have 
$r(S \cup a) + r(T) - r(M) = 1$. As $\sqcap(\{a\},\{b\}) = 1$ and $M$ is \thc,  $r(S \cup a \cup b)  = r(S \cup a) + 1$, so 
\begin{equation}
\label{eqmt}
\lambda_M(T) = 2.
\end{equation}
Moreover, 
$$r(S \cup a)  \ge r(S) + 1$$
otherwise $r(S \cup a) = r(S)$ so $r(E - b) = r(E - \{a,b\})$; a \cn. 

Next we show the following.

\begin{sublemma}
\label{addon}
Suppose $M\ba b$ has a $2$-separation $(S_1,S_2)$ where $S_1$ is the $N$-side and $S_2$ contains a prickly $3$-separator $\{u,v\}$ where $u \notin E(N)$. Then $M\ba b \da u$ is not $2$-connected. 
\end{sublemma}

Suppose $M\ba b \da u$ is  $2$-connected. Now $M\ba b = M_1 \oplus_2 M_2$ where $M_i$ has ground set $S_i \cup p$. Since $M\ba b \da u$ is  $2$-connected, $\sqcap_{M\ba b\da u}(S_1, S_2 - u) = 1$. Then, by Lemma`\ref{claim1}(iii), $M\ba b \da u$ has a special $N$-minor. By Lemma~\ref{portia}, $M\da u$ is \thc. Since it has a c-minor isomorphic to $N$, we have a \cn. Thus \ref{addon} holds.

Now suppose that $T$ is the $N$-side of $(S \cup a, T)$. Then, by Lemma~\ref{key}, $S \cup a$ contains no points. Assume that $r(S \cup a) = r(S) + 1$. 
As $r(S \cup a) + r(T) - r(M\ba b) = 1$, we see that 
$$[r(S) + 1] + r(T) - [r(M\ba b,a) + 1] = 1.$$ Hence $\sqcap(S,T) = 1$, so, by Lemma~\ref{claim1}(i), $M\ba b\ba a$ has a special $N$-minor. 
As $\{a,b\}$ is a prickly 3-separating set, we see that $M\ba b \ba a = M\da b \ba a$ so $M\da b$ has a c-minor isomorphic to $N$; a \cn. 

Next we consider the case when $T$ is the  $N$-side of $(S \cup a, T)$, and $r(S \cup a) = r(S) + 2$. Then 
$r(S) + r(T \cup a \cup b) = r(M) + 2$. Thus $S$ is a non-$N$-$3$-separator and so contains a minimal such set, $\{u,v\}$ where $u \notin E(N)$. From above, we know that $\{u,v\}$ is a prickly $3$-separator of $M$. By \ref{addon}, 
$M\ba b \da u$ is  not $2$-connected.
 Now $M\ba b \da u = M\da u\ba b$.  Let $(J,K)$ be a $1$-separation of $M\da u \ba b$ with $a \in J$. Then      $r_{M\da u}(J \cup b) \le r_{M\da u}(J) + 1$. Thus
\begin{align*}
r_{M\da u}(J \cup b) + r_{M\da u}(K) - r(M\da u)
& \le [r_{M\da u}(J) + r_{M\da u}(K) - r(M\da u\ba b)] \\
& \hspace*{0.9in} + [1 +  r(M\da u\ba b) - r(M\da u)]\\
& = 1 +  r(M\da u\ba b) - r(M\da u).
\end{align*}
By Lemma~\ref{portia}, $M\da u$ is \thc, so $r(M\da u\ba b) = r(M\da u)$, and $K$ consists of a single point, $k$, of $M\da u$. Then
\begin{align*}
1 & = r_{M\da u}(J) + r_{M\da u}(\{k\}) - r(M\da u\ba b)\\
& =   r_{M\da u}(E - \{b,u,k\}) + r_{M\da u}(\{k\})  - r(M\da u)\\
& = r(E - \{b,k\}) - 1 + r(\{k\}) - r(M) + 1\\
& = r(E - \{b,k\})   + r(\{k\}) - r(M\ba b).
\end{align*}
Hence  $\{k\}$ is $1$-separating in $M\ba b$.  Thus $k$ contradicts Lemma~\ref{Step0}.

When $M\ba b$ has a c-minor isomorphic to $N$, it remains to consider the case when $S \cup a$ is the $N$-side of $(S \cup a, T)$. As $\mu(b) \ge 3$, it follows that $|T| \ge 3$. By (\ref{eqmt}), $\lambda_M(T) = 2$. By assumption, $T$ contains a minimal non-$N$-$3$-separating set $T'$. The latter consists of a pair, $\{u,v\}$, of lines that form a prickly 3-separating set. We may assume that  $u \not\in E(N)$. Now $M\ba b$ is certainly $2$-connected. By Lemma~\ref{portia},   $M\da u$ is \thc. Since $\sqcap(\{a\},\{b\}) = 1$, it follows that $M\da u \ba b$ is $2$-connected; a \cn. We conclude that 
$M\ba b$ does not have a c-minor isomorphic to $N$.


We now know that $M/b$ has a c-minor isomorphic to $N$. Moreover,  $M^*$ has a c-minor isomorphic to $N^*$ and has $\{a,b\}$ as a prickly 3-separating set; and $(M/b)^* = (M^* \ba b)^{\flat}$. We use $M^* \ba b$ in place of $M\ba b$ in the argument above to complete the proof  of the lemma.
\end{proof}

The argument to establish that $M$ has a minimal non-$N$-$3$-separator with at least four elements is much   longer than that just given   since it  involves analyzing a number of cases. We shall use three preliminary results. In each, we denote $E(M) - Y_1$ by $X_1$. 

\begin{lemma}
\label{pre3lines}
Let $Y_1$ be a minimal-non-$N$-$3$-separator with exactly three elements. Suppose $\ell \in Y_1$ and $M\ba \ell$ has $N$ as a c-minor. Let $(A,B)$ be a \tws\ of $M\ba \ell$ where $A$ is the $N$-side and $|B|   \ge 3$. Suppose $\ell \in \cl(Y_1 - \ell)$. Then $(A,B)$ crosses $(X_1, Y_1 - \ell)$ and $\lambda_{M\ba \ell}(X_1 \cap A) \ge 2$. Moreover, $Y_1 \cap B$ consists of a single line.
\end{lemma}

\begin{proof} As $\ell \in \cl(Y_1 - \ell)$, we see that $\lambda_{M_1 \ba \ell}(Y_1 - \ell) = 2$. To see that $(A,B)$ crosses 
$(X_1, Y_1 - \ell)$, 
note first that, as $|Y_1 - \ell| = 2$ and $|A|, |B| \ge 3$,   neither $A$ nor $B$ is contained in $Y_1 - \ell$. Moreover, $Y_1 - \ell$ is not contained in $A$ or $B$ otherwise $(A \cup \ell,B)$ or $(A,B \cup \ell)$ is a \tws\ of $M$; a \cn. Hence $(A,B)$ crosses 
$(X_1, Y_1 - \ell)$. 


As $|E(N)| \ge 4$, we see that $|X_1 \cap A| \ge 2$. Then 
$$\lambda_{M\ba \ell}(X_1 \cap A) \ge 2$$
 otherwise, as $\ell \in \cl(Y_1 - \ell)$, we get the \cn\ that $\lambda_{M}(X_1 \cap A) \le 1$. By uncrossing, $\lambda_{M\ba \ell}(Y_1 \cap B) \le 1$. By Lemma~\ref{key}, $B$ contains no   points, so $Y_1 \cap B$ contains no points. As $|Y_1 - \ell| = 2$, we see that $Y_1 \cap B$ consists of a single line.
 \end{proof}

\begin{lemma}
\label{3lines}
Let $Y_1$ be a minimal-non-$N$-$3$-separator with exactly three elements. If $Y_1$ contains a line $\ell$ such that $M\ba \ell$ has $N$ as a c-minor, then $Y_1$ consists of three lines.
\end{lemma}

\begin{proof} Assume that the lemma fails. Let  $(A,B)$ be a \tws\ of $M\ba \ell$ where $A$ is the $N$-side and $|B|  \ge 3$. First we show that 

\begin{sublemma}
\label{3lines1}
$\ell \in \cl(Y_1 - \ell)$.
\end{sublemma}

Assume that $\ell \not\in \cl(Y_1 - \ell)$. Then $(X_1, Y_1 - \ell)$ is a \tws\ of $M\ba \ell$ with $|Y_1 - \ell| = 2$.  By Lemma~\ref{key}, we may assume that $Y_1 - \ell$ consists of a series pair $\{y_1,y_2\}$ of points. Now $r(M\ba \ell) = r(M) = r(X_1) + 1$, so $r(\{\ell,y_1,y_2\}) = 3$. Moreover, for each $i$ in $\{1,2\}$, we see that $M\ba \ell/y_i$, and hence $M/y_i$, has a special $N$-minor.

As the theorem fails for $M$, we know that $M/y_i$ is not \thc. Now $M/y_i$ is certainly $2$-connected. Let $(J,K)$ be a \tws\ of it where we may assume that $\ell \in J$. Now $r_{M/y_i}(\{\ell,y_j\}) = 2$ where $\{i,j\} = \{1,2\}$.

Suppose $r_{M/y_i}(\{\ell\}) = 2$.  Assume $y_j \in K$. Then $(J \cup y_j, K- y_j)$ is a \tws\ of $M/y_i$ unless $K-y_j$ consists of a single point. In the exceptional case, $y_j$ is in a parallel pair of points in $M/y_i$. Hence $M\ba y_j$ has a special $N$-minor. As $M/y_j$ also has such a minor, we contradict Lemma~\ref{Step0}. We deduce that  we may assume that $J$ contains $\{\ell,y_j\}$. Then $r(J \cup y_i) + r(K \cup y_i) = r(M) + 2$, so $r(J \cup y_i,K)$ is a \tws\ of $M$; a \cn.

We may now assume that $r_{M/y_i}(\{\ell\}) = 1$. Then $y_i$ lies on the line $\ell$. Since this must be true for each $i$ in $\{1,2\}$, we see that $r(\{\ell,y_1,y_2\}) = 2$; a \cn. We deduce that \ref{3lines1} holds.

By Lemma~\ref{pre3lines}, we know that $(A,B)$ crosses $(X_1,Y_1)$, that $\lambda_{M\ba \ell}(X_1 \cap A) \ge 2$, and that $Y_1 \cap B$ consists of a single line. As the lemma fails, $A \cap (Y_1 - \ell)$ consists of a single point, $a$. As $\lambda_{M\ba \ell}(X_1 \cap A) \ge 2$ and $\lambda_{M\ba \ell}(A) = 1$, we deduce that $r(A-a) = r(A)$ and $r(B \cup a) = r(B) + 1$. Hence $a \in \cl(X_1)$. Thus $Y_1 - a$ is a minimal non-$N$-$3$-separator; a \cn.
\end{proof}

The next lemma verifies (\ref{Step5.5}).

\begin{lemma}
\label{y13}
Let $Y_1$ be a minimal-non-$N$-$3$-separator having exactly three elements. Then $Y_1$ consists of three lines.
\end{lemma}

\begin{proof} 
As $|Y_1 \cap E(N)| \le 1$, at least two of the elements of $Y_1$ are not in $E(N)$. Let $\ell$ be one of these elements. Suppose $\ell$ is a line. If $M\ba \ell$ has $N$ as a c-minor, then the result follows by Lemma~\ref{3lines}. If $M/\ell$ has $N$ as a c-minor, then $(M^*\ba \ell)^{\flat}$, and hence 
$M^*\ba \ell$ has $N^*$ as a c-minor and again the result follows by Lemma~\ref{3lines}. 

We may now assume that $\ell$ is a point. By switching to the dual if necessary, we may assume that $M\ba \ell$ has $N$ as a c-minor. Let $(A,B)$ be a \tws\ of $M\ba \ell$ where $A$ is the $N$-side and $|B|   \ge 3$.  Next we show that 

\begin{sublemma}
\label{ellisnot}
$\ell \notin \cl(Y_1 - \ell)$.
\end{sublemma}

Assume $\ell \in  \cl(Y_1 - \ell)$. Then, by Lemma~\ref{pre3lines}, we know that $(A,B)$ crosses $(X_1,Y_1- \ell)$, that $\lambda_{M\ba \ell}(X_1 \cap A) \ge 2$, and that $Y_1 \cap B$ consists of a single line, say $m$. Now $|B \cap X_1| \ge 2$ since $|B| \ge 3$. Then 
$$\lambda_{M\ba \ell}(B \cap X_1) \ge 2$$ 
otherwise, since $\ell \in \cl(Y_1 - \ell)$, we deduce that $\lambda_{M}(B \cap X_1) \ge 1$; a \cn. 
By uncrossing, $\lambda_{M\ba \ell}(Y_1 \cap A) \le 1$. 

Since $|Y_1| = 3$ and $Y_1 \cap B$ consists of the line $m$, we deduce that $A \cap (Y_1 - \ell)$ consists of a single point, say $a$, otherwise one of the elements of $Y_1 - \ell$ is a line that is not in $E(N)$ and we have already dealt with that case. As $\lambda_{M\ba \ell}(X_1 \cap A) \ge 2$ and $\lambda_{M\ba \ell}(A) = 1$, we deduce that 
\begin{equation}
\label{aab}
r(A-a) = r(A) \text{~~and~~} r(B \cup a) = r(B) + 1.
\end{equation}
Hence 
\begin{equation}
\label{ax1}
a \in \cl(X_1). 
\end{equation}
We may assume that $m \in E(N)$ otherwise $m$ is removed in forming $N$ and that case was dealt with in the first paragraph. 

Now $Y_1 = \{a,\ell,m\}$. As $m \in B$, it follows by (\ref{aab}) that $r(\{m,a\}) = 3$.  Moreover, as $\{m,a\} = Y_1 - \ell$ and $\ell \in \cl(Y_1 - \ell)$, we deduce that $r(Y_1) = 3$.  By (\ref{ax1}), $r(X_1 \cup a) = r(X_1)$. We deduce that 
\begin{equation}
\label{lmn}
r(\{a,\ell,m\}) = r(\{\ell,m\}) = 3 \text{~~and~~} r(X_1 \cup a) = r(M) - 1.
\end{equation}

Since $m \in E(N)$, it follows that $a \not\in E(N)$. Suppose that $M\ba \ell /a$ has $N$ as a c-minor.  Still as part of the proof of \ref{ellisnot}, we show next that 

\begin{sublemma}
\label{am2s}
$M/a$ is the $2$-sum with basepoint $q$ of two $2$-polymatroids, one of which consists of the line $m$ having non-parallel points $q$ and $\ell$ on it.
\end{sublemma}

By (\ref{lmn}), $(\{\ell,m\},X_1)$ is a 2-separation of $M/a$. Thus $M/a$ is the 2-sum with basepoint $q$ of two $2$-polymatroids, one of which, $Q$ say, consists of the line $m$ having  points $q$ and $\ell$ on it. Suppose $q$ and $\ell$ are parallel points in $Q$. Then $(\{m\},X_1 \cup \ell)$ is a 2-separation of $M/a$. It follows that  
$(\{m\},X_1 \cup \ell \cup a)$ is a 2-separation of $M$; a \cn. Thus \ref{am2s} holds. 

By \ref{am2s}, both $M/\ell$ and $M\ba \ell$ have $N$ as a c-minor; a \cn\ to Lemma~\ref{Step0}.

We now know that $N$ is a c-minor of $M\ba \ell\ba a$. In that $2$-polymatroid, $\{m\}$ is $2$-separating so, in the formation of $N$, the element $m$ is compactified. As the next step towards showing \ref{ellisnot}, we now show that 

\begin{sublemma}
\label{m1m'}
$M\ud m$ is $3$-connected.
\end{sublemma}

To see this, it will be helpful to consider the $2$-polymatroid $M_1$ that is obtained from $M$ by freely adding the point $m'$ on $m$. By definition, $M\ud m = M_1/m_1$. Certainly $M_1$ is \thc, so $M_1/m'$ is $2$-connected. Assume it has a \tws\ $(U,V)$ where $m \in U$. Then 
$$r(U \cup m') + r(V \cup m') - r(M_1) = 2.$$
But $r(U \cup m') = r(U).$ Hence $r(V \cup m') = r(V)$ otherwise $M_1$ has a \tws; a \cn. But, as $m'$ was freely placed on $m$, we deduce that 
$r(V\cup m' \cup m) = r(V \cup m') = r(V)$. Now, in $M_1\ba m$, we see that $\{\ell,m'\}$ is a series pair of points. As $m' \in \cl(V)$, it follows that $\ell \in V$. Then $r(U - m) < r(U)$ since $\{m\}$ is 2-separating in $M\ba \ell$. Now $r(U - m) = r(U) - 1$ otherwise $r(U - m) = r(U) - 2$ and 
$(U-m,V\cup \{m',m\})$ is a $1$-separation of $M_1$. As $(U-m,V\cup \{m',m\})$ is not a $2$-separation of $M_1$, it follows that $U - m$ consists of a single point $u$ and $r(\{u,m\}) = 2$.  Thus, in $M\ba \ell$, when we compactify $m$, we find that $u$ and $m$ are parallel. Since $m\in E(N)$, we see that $u \not\in E(N)$. Moreover, $M\ba u$ has $N$ as a c-minor. Since $u$ lies on $m$ in $M$, we deduce that $M\ba u$ is \thc\ having $N$ as a c-minor. This \cn\ completes the proof of \ref{m1m'}.

Now, in $M_1/m'$, the elements $a, \ell$, and $m$ form a triangle of points. We know that $M_1/m'\ba \ell$ is not \thc\ otherwise $(M\ba \ell)^{\flat}$ is \thc\ having $N$ as a c-minor. Because $M\ba a$ has $N$ as a c-minor, $M\ba a$ is not \thc, so $M_1 \ba a$ is not \thc. Still continuing with the proof of 
\ref{ellisnot}, we show next that 

\begin{sublemma}
\label{m1m'a}
$M_1\ba a/m'$   is  not $3$-connected.
\end{sublemma}

Let $(G,H)$ be a $2$-separation of $M_1 \ba a$ with $m$ in $G$. Then $(G \cup m',H-m')$ is a \tws\ of $M_1\ba a$ unless $H$ consists of two points. In the exceptional case, $r(H) = 2$ so $r(G) = r(M) - 1$. But then $a \cup (H - m')$ is a series pair in $M$; a \cn. We conclude that we may assume that $m' \in G$. Then $\ell \in H$, otherwise, by (\ref{lmn}),  $(G \cup a, H)$ is a \tws\ of $M_1$; a \cn. 

Observe that $G \neq \{m,m'\}$ otherwise $\{m\}$ is 2-separating in $M\ba a$ and so, as $a \in \cl(X_1)$ , we obtain the \cn\ that $\{m\}$ is 2-separating in $M$.

Now 
\begin{equation}
\label{m1m1m1}
r_{M_1 \ba a/m'}(G - m') + r_{M_1 \ba a/m'}(H) - r(M_1 \ba a/m')= r(G)  + r(H \cup m') - 1 - r(M_1\ba a).
\end{equation} 

Suppose  that $r(H \cup m') = r(H)$. Then $r(H \cup m'\cup m) = r(H)$ as $m'$ is freely placed on $m$. Thus, as  $G \supsetneqq \{m,m'\}$ and $\{m\}$ is 2-separating in $M\ba \ell \ba a$, we see that $(G- m- m', H \cup \{m,m'\})$ is a 1-separation of $M_1\ba a$. Therefore  $(G- m- m', H \cup \{m,a\})$ is a 1-separation of $M$; a \cn. 

We now know that  $r(H \cup m') = r(H) + 1$. Then, as $(G,H)$ is a \tws\ of $M_1 \ba a$, it follows by (\ref{m1m1m1}) that $(G-m',H)$ is a \tws\ of $M_1 \ba a /m'$ unless either $|H| = 1$ and $r_{M_1/m'}(H) = 1$, or $|G-m'| = 1$ and $r_{M_1/m'}(G-m') = 1$. Consider the exceptional cases. The first of these cannot occur since $m'$ is freely placed on $m$; the second cannot occur since it implies that $G = \{m,m'\}$, which we eliminated above. As neither of the exceptional cases occurs, $M_1 \ba a /m'$ has a 2-separation and so 
 \ref{m1m'a} holds. 
 
Recall that $M_1/m' = M\ud m$. In this 
$2$-polymatroid, we have $\{a, \ell,m\}$ as a triangle such that the deletion of either $a$ or $\ell$ destroys $3$-connectedness. Hence, by \cite[Lemma 4.2]{oswww}, there is a triad of $M_1/m'$ that contains $a$ and exactly one of $\ell$ and $m$.  Assume this triad contains $\ell$. 
 Thus, in $M\ba \ell\ud m$, we have that $a$ is in a series pair with some element $b$. Then $M\ba \ell/a$ has $N$ as a c-minor, so $a$ is a doubly labelled point of $M$; a \cn\ to Lemma~\ref{Step0}. We deduce that $M_1/m'$ has a triad containing $\{a,m\}$ but not $\ell$.
Then $M_1/m'\ba \ell$, which equals $M\ba \ell\ud m$, either has a triad   containing $\{a,m\}$ or has $a$ in a series pair. This is straightforward to see by considering  the matroid that is naturally derived from  $M\ba \ell\ud m$ and using properties of the cocircuits in this matroid. Now $a$ is not in a series pair in $M\ba \ell\ud m$ otherwise we again obtain the \cn\ that $a$ is a doubly labelled point. We deduce that $M\ba \ell\ud m$ has a triad containing $\{a,m\}$. Since $m \in B$ and, by (\ref{aab}), $r(A - a) = r(A)$, we must have that the third point, $b$, of this triad is in $A-a$. 

Now $M\ba \ell \ud m$ has $(A,B)$ as a $2$-separation and has $\{a,b,m\}$ as a triad with $\{a,b\} \subseteq A$. Thus 
$(A \cup m,B - m)$ is a \tws\ of $M\ba \ell \ud m$. Since $\ell$ is in the triangle $\{a,m,\ell\}$ in $M \ud m$, it follows that $(A \cup m\cup \ell,B - m)$ is a \tws\ of $M \ud m$. This \cn\ to \ref{m1m'} completes the proof of \ref{ellisnot}.


Since $\ell \not\in \cl(Y_1 - \ell)$, we deduce that $(X_1\cup \ell,Y_1 - \ell)$ is 3-separating in $M$. Because $Y_1$ is a minimal non-$N$-3-separating set, $Y_1 - \ell$ does not consist of two lines. Moreover, $(X_1,Y_1 - \ell)$ is a 2-separation in $M\ba \ell$. 

\begin{sublemma}
\label{pointline}
$Y_1 - \ell$  does not consist of a point and a line.
\end{sublemma}

Assume that $Y_1- \ell$ consists of a line $k$ and a point $y$. If $k \not\in E(N)$, then the argument in the first paragraph of the proof of the lemma gives a \cn. Thus $k \in E(N)$, so $y \not\in E(N)$. If $r(Y_1 - \ell) = 2$, then $M\ba \ell \ba y$, and hence $M\ba y$,  has $N$ as a c-minor. Since $y$ is on the line $k$, we see that $M\ba y$ is  \thc; a \cn. We deduce that $r(Y_1 - \ell) = 3$. Hence
\begin{equation}
\label{x1rk}
r(X_1) = r(M) - 2 \text{~~and~~} r(X_1 \cup \ell) = r(M) - 1.
\end{equation}

Now $M\ba \ell$ is the 2-sum with basepoint $p$, say, of two $2$-polymatroids, $M_X$ and $M_Y$, with ground sets $X_1 \cup p$ and $(Y_1 - \ell) \cup p$, respectively. Then $r(M_Y) = 3$. Moreover, $y$ does not lie on $k$ in $M_Y$, otherwise $M_Y$ is not $2$-connected, a \cn\ to Proposition~\ref{connconn}. 
Thus $M^*\ba y$ has $N^*$ as a c-minor. Then, by applying \ref{ellisnot} to $M^*\ba y$, we deduce that $y \not\in \cl_{M^*}(Y_1 - y)$. Thus 
$r^*(Y_1 - y) = r^*(Y_1) - 1$. It follows that $r(X_1 \cup \ell \cup y) = r(X_1 \cup \ell)$. But $r(X_1 \cup \ell \cup y) = r(M\ba k) = r(M)$ yet $r(X_1 \cup \ell) = r(M) - 1$. This \cn\ completes the proof of \ref{pointline}

We now know that $Y_1 - \ell$ consists of a series pair of points, say $y_1$ and $y_2$. Now $r(M\ba \ell) = r(M) = r(X_1) + 1$. Also $r(\{\ell,y_1,y_2\}) = 3$. Thus $\{\ell,y_1,y_2\}$ is a triad of $M$. Moreover, both $M/y_1$ and $M/y_2$ have special $N$-minors. Thus neither is \thc. By \cite[Lemma 4.2]{oswww}, $M$ has a triangle that contains $y_1$ and exactly one of $y_2$ and $\ell$. Likewise, $M$ has a triangle that contains $y_2$ and exactly one of $y_1$ and $\ell$. Thus either 
\begin{itemize}
\item[(i)] $M$ has a triangle $\{y_1,y_2,z\}$; or 
\item[(ii)] $M$ has triangles $\{y_1,\ell,z_1\}$ and $\{y_2,\ell,z_2\}$ but no triangle containing $\{y_1,y_2\}$. 
\end{itemize}

In the first case, $M/y_1$ has $\{y_2,z\}$ as a pair of parallel points. Hence $M\ba y_2$ has a special $N$-minor. Thus $y_2$ is doubly labelled; a \cn. We deduce that (ii) holds. Thus $M$ contains a fan $x_1,x_2,\dots,x_n$ where $(x_1,x_2,x_3,x_4,x_5) = (z_2,y_2,\ell,y_1,z_1)$. Hence $M/x_2$ has a c-minor isomorphic to $N$. Then, by Lemmas~\ref{fantan} and \ref{Step0}, we obtain a \cn.
\end{proof}

We complete the proof of Lemma~\ref{Step6} by analyzing the various possibilities for a minimal non-$N$-3-separator consisting of exactly three lines.

\section{A minimal non-$N$-3-separator consisting of exactly three lines}
\label{threeel}

In this section, we finish the proof of (\ref{Step6}).  We begin by restating that assertion. 

\begin{lemma}
\label{Step6+} 
$M$ has a minimal non-$N$-$3$-separator with at least four elements. 
\end{lemma}

We have $(X_1,Y_1)$ as a $3$-separation of the \thc\ $2$-polymatroid $M$. We shall consider the extension $M+z$ of $M$ that is obtained by adjoining the line $z$ to $M$ so that $z$ is in the closure of each of $X_1$ and $Y_1$ in $M+z$. To see that this extension exists, we note that, by building on a result of Geelen, Gerards, and Whittle~\cite{ggwtconn}, Beavers~\cite[Proposition~2.2.2]{beavs}
showed that, when $(A,B)$ is a \ths\ in a \thc\ matroid $Q$, we can extend $Q$ by an independent set $\{z_1,z_2\}$ of size two so that  these two points are clones, and each lies in the closure of both $A$ and $B$ in the extension $Q'$. 
By working in the matroid naturally derived from $M$, we can add $z_1$ and $z_2$. This corresponds to adding the line $z$ to $M$  to form $M+z$ where $z = \{z_1,z_2\}$.

More formally, recall that the natural matroid $M'$ derived from $M$ is obtained from $M$ by freely adding two points, $s_{\ell}$ and $t_{\ell}$, on each line $\ell$ of $M$ and then deleting all such lines $\ell$. After we have extended $M'$ by $z_1$ and $z_2$, we have a matroid with points $\{z_1,z_2\} \cup \{p:\text{~$p$ is a point of $M$}\} \cup \{s_{\ell},t_{\ell}:\text{~$\ell$ is a line of $M$}\}$. Taking $z= \{z_1,z_2\}$, we see that $M+ z$ is the $2$-polymatroid with elements $\{z\} \cup \{p:\text{~$p$ is a point of $M$}\} \cup \{\ell:\text{~$\ell$ is a line of $M$}\} = \{z\} \cup E(M)$. We call $M+z$ the $2$-polymatroid that is obtained from $M$ by {\it adding the guts line $z$ of $(X_1,Y_1)$.}

When we have $Y_1$ as a minimal non-$N$-$3$-separator of $M$ consisting of three lines, we look at $(M+z)|(Y_1 \cup z)$. This $2$-polymatroid consists of exactly four lines.

\begin{lemma}
\label{claim1y1} 
$(M+z)|(Y_1 \cup z)$ has no parallel lines, so $r_{M+z}(Y_1 \cup z) \ge 3$. 
\end{lemma}

\begin{proof}
Suppose $a$ and $b$ are parallel lines in $Y_1$. Then we may assume that $b\not\in E(N)$. Now $M\ba b$ or $M/b$ has $N$ as a c-minor. In the latter case, as $a$ is a loop of $M/b$, it follows that $a \not \in E(N)$ and $M\ba a$ has $N$ as a c-minor. We conclude that $M\ba b$ or $M\ba a$ has $N$ as a c-minor. Since each of $M\ba b$ and $M\ba a$ is \thc, we obtain the \cn\ that the theorem holds. Thus $Y_1$ contains no pair of parallel lines.

Suppose $z$ is parallel to some element $y$ of $Y_1$. Then $(X_1 \cup y, Y_1 - y)$ is a non-$N$-$3$-separator of $M$ contradicting the minimality of $Y_1$. Thus $(M+z)|(Y_1 \cup z)$ has no parallel lines and the lemma holds.
\end{proof}

\begin{lemma}
\label{claim3} 
 $r(Y_1) > 3$. 
\end{lemma}

\begin{proof} 
Assume that 
$r(Y_1) = 3$. Then $r_{M+z}(Y_1 \cup z) = 3$, so $\sqcap(z,y) = 1$ for all $y$ in $Y_1$. Moreover, $r(Y_1 - y) = 3 = r(Y_1)$ for all $y$ in $Y_1$.

Suppose that $y \in Y_1 - E(N)$ and $N$ is a c-minor of $M/y$. Then the remaining two elements, $y_1$ and $y_2$, of $Y$ are parallel points in $M/y$. We may assume that $y_1 \not \in E(N)$. Thus $M\ba y_1$ has $N$ as a c-minor. We conclude that $N$ is a c-minor of $M\ba y$ for some element $y$ of $Y_1$. We now focus on this element $y$.

Let $(R,G)$ be a non-trivial \tws\ of $M\ba y$, that is, $\lambda_{M\ba y}(R)  = 1$ and $\min\{|R|,|G|\} \ge 2$. We show next that 

\begin{sublemma}
\label{claim4} 
$(R,G)$ crosses $(X_1,Y_1 - y)$.
\end{sublemma}

If $R \subseteq X_1$, then $G \supseteq Y_1 - y$ so $y \in \cl_M(G)$ and $(R,G \cup y)$ is a \tws\ of $M$. This \cn\ implies, using symmetry, that both $R$ and $G$ meet $Y_1 - y$.

Suppose $R \cap X_1 = \emptyset$. Then $R$ consists of  single line, so $(R,G)$ is a trivial \tws. This \cn, combined with symmetry, completes the proof of \ref{claim4}.

Let $Y_1 - y = \{a,b\}$. We may assume that $a \in R$ and $b \in G$. Now, as $y \in \cl(Y_1 - y)$, we see that $\lambda_{M\ba y}(Y_1 - y) = 2$. Thus 
\begin{align*} 
1+2 & = \lambda_{M\ba y}(R) + \lambda_{M\ba y}(Y_1 - y)\\
& \ge \lambda_{M\ba y}(\{a\}) + \lambda_{M\ba y}(R\cup (Y_1 - y))\\
& = \lambda_{M\ba y}(\{a\}) + \lambda_{M\ba y}(G \cap X_1).
\end{align*}

We know $r(E- Y_1) = r(X_1) = r(M) - 1$ since $r(Y_1) = 3$. Thus $r(E - \{y,a\}) = r(X_1 \cup b) = r(M)$. Hence 
$$ \lambda_{M\ba y}(\{a\}) = r(\{a\}) + r(E - \{y,a\}) - r(E - y) = r(\{a\}) = 2,$$
so $\lambda_{M\ba y}(G\cap X_1) \le 1$. But $y \in \cl(\{a,b\})$ so $\lambda_M(G \cap X_1) \le 1$. By symmetry, $\lambda_M(R \cap X_1) \le 1$. We conclude that $|G \cap X_1| \le 1$ and $|R \cap X_1| \le 1$, so $|X_1| \le 2$. This is a \cn\ since $|E(N)| \ge 4$. We conclude that the lemma holds.
\end{proof}

\begin{lemma}
\label{claim5} 
 $r_{M+z}(Y_1 \cup z) =r_M(Y_1) = 4$. 
\end{lemma}

\begin{proof} We know that $r_{M+z}(Y_1 \cup z) = r_M(Y_1) \ge 4$. Suppose $r_M(Y_1) \ge 5$. Then $r_M(X_1) \le r(M) - 3$, so
$$r_{M^*}(Y_1) = \sum_{y \in Y_1} r_M(\{y\}) + r_M(X_1) - r(M) \le 6 + r(M) - 3 - r(M) = 3.$$  By using $M^*$ in place of $M$, we get a \cn\ to Lemma~\ref{claim3}. We conclude that the lemma holds. 
\end{proof}

We will now work with the $2$-polymatroid $(M+z)|(Y_1 \cup z)$, which we  rename $P$. This has rank 4 and consists of four lines, $z, a,b,$ and $c$. 

\begin{lemma}
\label{claim5.5} 
If $B \subseteq Y_1$ and $A = Y_1 - B$, then 
$$\sqcap_P(A \cup z, B) = \sqcap_{M+z}(A \cup X_1 \cup z, B).$$
\end{lemma}

\begin{proof} Since $P = (M+z)|(Y_1 \cup z)$, we can do all of these local connectivity calculations in $M+z$. Now 
$ \sqcap(A \cup z, X_1) = \sqcap(A \cup z, X_1\cup z)$, so 
$$2 = \sqcap(Y_1 \cup z, X_1) \ge \sqcap(A \cup z, X_1) = \sqcap(A \cup z, X_1\cup z) \ge 2.$$
Thus 
$$r(A \cup z) - 2 = r(A \cup z \cup X_1) - r(X_1).$$
Hence
\begin{align*}
\sqcap(A \cup z, B) & = r(A \cup z) + r(B) - r(A\cup z \cup B)\\
& = r(A \cup z \cup X_1) - r(X_1) + 2 + r(B) - r(Y_1)\\
& = r(A \cup z \cup X_1) + r(B) - [r(X_1) + r(Y_1) - 2]\\
& = r(A \cup z \cup X_1) + r(B) - r(M)\\
& = \sqcap(A \cup X_1 \cup z, B).
\end{align*}
\end{proof}

\begin{lemma}
\label{claim6} 
$P$ is $3$-connected. 
\end{lemma}

\begin{proof} From the last lemma, if $(A,B)$ is a $k$-separation of $P$ for some $k$ in $\{1,2\}$ and $z \in A$, then $ (A\cup X_1\cup z,B)$ is a $k$-separation of $M+z$; a \cn.
\end{proof}

\begin{lemma}
\label{3one} 
If $y \in Y_1$ and $\sqcap(X_1,\{y\}) = 1$, then $r(Y_1 - y) = 4$.
\end{lemma}

\begin{proof}
By Lemma~\ref{claim1y1}, $r(Y_1 - y) > 2$. If $r(Y_1 - y) = 3$, then $(X_1 \cup y, Y_1 - y)$ is a \ths\ violating the choice of $(X_1,Y_1)$.
\end{proof}

\begin{lemma}
\label{3m} 
Suppose $y \in Y_1$ and $r(Y_1 - y) = 4$. If $m$ is a line  such that $\{m\}$ is $2$-separating in  $M\ba y$, then $m \in Y_1 - y$.
\end{lemma}

\begin{proof}
We have $1 = r(\{m\}) + r(E - \{y,m\}) - r(M\ba y)$. Thus $r(E - \{y,m\}) = r(M) - 1$. Suppose $m \not\in Y_1 - y$. Then $E- \{y,m\}$ contains $Y_1 - y$ and so spans $y$. Thus  $r(E - \{y,m\}) = r(M\ba m) = r(M)$; a \cn.
\end{proof}

The next four lemmas will help eliminate many of the  possibilities for $P$.

\begin{lemma}
\label{parallel}  
If $c$ is skew to $X_1$ in $M$, and $M/c$ has $a$ and $b$ as parallel lines, then $M/c$ is $3$-connected.
\end{lemma}

\begin{proof}
Assume $(A,B)$ is a $k$-separation of $M/c$ for some $k$ in $\{1,2\}$ where $|A| \le |B|$. If $\{a,b\} \subseteq Z$ for some $Z$ in $\{A,B\}$, and $\{Z,W\} = \{A,B\}$, then $r(Z \cup c) + r(W \cup c) - r(M) = k+1$. But $c$ is skew to $W$ since $W \subseteq X_1$, so $(Z \cup c,W)$ is a  $k$-separation of $M$; a \cn. We may now assume that $a \in A$ and $b \in B$. Then $(A \cup b, B-b)$ is a $k$-separation of $M/c$ with $\{a,b\} \subseteq A \cup b$ and this possibility has already been eliminated.
\end{proof}

\begin{lemma}
\label{earlier}  
If $c$ is skew to each of $a$, $b$, and $X_1$ in $M$,  then $M/c$ has no c-minor isomorphic to $N$.
\end{lemma}

\begin{proof} 
We see that $M/c$ has $a$ and $b$ as parallel lines. Since $(M,N)$ is a counterexample to the theorem, we obtain this lemma as a direct consequence  of the last one.
\end{proof}

\begin{lemma}
\label{earlybird}  
Assume that $M\ba b$ has a c-minor isomorphic to $N$ and  that $P\ba b$ has rank $4$, has $c$ skew to each of $a$ and $z$, and has $\sqcap(\{a\},\{z\})   = 1$.   Then $M/c$ has a c-minor isomorphic to $N$.
\end{lemma}

\begin{proof}
Let $(A,C)$ be a non-trivial \tws\ of $M\ba b$. If $\{a,c\}$ is contained in $A$ or $C$, then $M$ has a \tws; a \cn. Thus we may assume that $a \in A$ and $c \in C$. Now $c$ is skew to $C - c$ so $(A \cup c, C-c)$ is 2-separating in $M\ba b$. Hence $(A \cup c \cup b, C - c)$ is 2-separating in $M$. Thus $C - c$ consists of a point $d$  of $M$. Now, by Lemma~\ref{3m}, the only $2$-separating lines in $M\ba b$ can be $a$ and $c$. But $a$ is not 2-separating. Thus $(M\ba b)^{\flat} = M\ba b\ud c$, so $c$ is a point of $M\ba b\ud c$. The rank of this $2$-polymatroid is $r(M) - 1$, and it has $\{c,d\}$ as a series pair since $A$ has rank $r(M) - 2$ in it. Thus $M\ba b\ud c/c$, and hence $M/c$, has a c-minor isomorphic to $N$. 
\end{proof}

\begin{lemma}
\label{earlybird2}  
If $\sqcap(\{a\},\{z\}) = 1$ and both $b$ and $c$ are skew to each other and to $z$, then    $M\ba a$ has no c-minor isomorphic to $N$.
\end{lemma}

\begin{proof} Assume that $M\ba a$ has a c-minor isomorphic to $N$. Let $(B,C)$ be a $k$-separation of $M\ba a$ for some $k$ in $\{1,2\}$. If $B$ or $C$ contains $\{b,c\}$, then $M$ has a $k$-separation. Thus we may assume that $b \in B$, that $c \in C$, and that $|B| \ge |C|$. Then $b$ is skew to $B- b$, so the partition $(B-b,C \cup b \cup a)$ of $E(M)$ shows that $M$ is not \thc; a \cn. 
\end{proof}

By Lemma~\ref{claim1y1}, for all $y$ in $Y_1$, we have $\sqcap(\{y\},\{z\}) \in \{0,1\}$. We shall   treat the  possibilities for $P$ based on the number $\theta$ of 
members $y$ of $Y_1$ for which $\sqcap(\{y\},\{z\}) = 1$. The most difficult case is when $\theta = 3$ and we will treat that after we deal with the cases when $\theta=2$ and when $\theta = 1$.

\begin{lemma}
\label{3two} 
$\theta \neq 2$.
\end{lemma}

\begin{proof} 
Suppose that $\sqcap(\{a\},\{z\}) = 1 =  \sqcap(\{b\},\{z\})$ and $\sqcap(\{c\},\{z\}) = 0$. Then, by Lemma~\ref{3one}, $r(\{b,c\}) = 4 = r(\{a,c\})$. Thus, by Lemma~\ref{earlier}, $M/c$ has no c-minor isomorphic to $N$. By Lemma~\ref{earlybird}, neither $M\ba a$ nor $M\ba b$ has a c-minor isomorphic to $N$. Thus, without loss of generality, we may assume that $M/a$ has a c-minor isomorphic to $N$. Now, in $M/a$, we have $\{b,c\}$ as a $2$-separating set where $c$ is a line and $b$ is either a point on that line or is a parallel line. Thus, by Lemma~\ref{claim1}, $M/a\ba b$, and hence $M\ba b$, has a c-minor isomorphic to $N$; a \cn.
\end{proof}

We can exploit duality to eliminate the case when $\theta = 1$.

\begin{lemma}
\label{3three} 
$\theta \neq 1$.
\end{lemma}

\begin{proof}
Suppose that $\sqcap(\{a\},\{z\}) = 1$ and   $\sqcap(\{b\},\{z\}) = 0 = \sqcap(\{c\},\{z\})$. Then, by Lemma~\ref{3one}, $r(\{b,c\}) = 4$. 
By Lemma~\ref{general4}, for $y$ in $Y_1$, we have $\sqcap^*(\{y\}, X_1) = \lambda_{M/(Y-y_1)}(\{y\})$. Since $\{b,c\}$ spans $a$ in $M$, we deduce that $\sqcap^*(\{a\},X_1) = 0$. If $\sqcap^*(\{b\},X_1) = 1 = \sqcap^*(\{c\},X_1)$, then $\theta = 2$ in $M^*$ so the result follows by Lemma~\ref{3two}. Thus, we may assume, by symmetry, that $\sqcap^*(\{b\},X_1) = 0$. Hence $\{a,c\}$ spans $b$ in $M$, so $r(\{a,c\}) = 4$. Thus, by Lemma~\ref{earlier}, $M/c$ does not have a c-minor isomorphic to $N$. By Lemma~\ref{earlybird}, $M\ba b$ has no c-minor isomorphic to $N$. If $\sqcap(\{a\},\{b\}) = 0$, then, by symmetry, the argument of the last two sentences shows that neither $M/b$ nor $M\ba c$ has a c-minor isomorphic to $N$. Thus both $b$ and $c$ must be in every c-minor of $M$ isomorphic to $N$; a \cn. We deduce that $\sqcap(\{a\},\{b\}) = 1$. 

By Lemma~\ref{earlybird2}, $M\ba a$ has no c-minor isomorphic to $N$. Suppose $M/a$ has a c-minor isomorphic to $N$. In $M/a$, we see that $\{c,b\}$ is a 2-separating set with $c$ as a line and $b$ as a point on it.  Hence, by Lemma~\ref{claim1}, $M/a\ba b$, and so $M\ba b$,  has a c-minor isomorphic to $N$; a \cn.  We conclude that $M/a$ has no c-minor isomorphic to $N$. It follows that $a$ is in every c-minor of $M$ isomorphic to $N$. Thus $M/b$ has $N$ as a c-minor. In $M/b$, we see that $a$ is a point on the line $c$. Suppose that $a$ is parallel to some point $e$, say. Then $e \in X_1$. Moreover,  $M/b\ba e$, and hence $M\ba e$, has a c-minor isomorphic to $N$. Now $r(X_1 \cup \{a,b\}) = r(X_1) + 2$. Thus
\begin{align*}
r(X_1) + 1 + 3 & = r(X_1 \cup a) + r(\{a,b,e\})\\
& \ge r(\{a,e\})  + r(X_1 \cup \{a,b\})\\
& = r(\{a,e\})  + r(X_1) + 2.
\end{align*}
Hence $r(\{a,e\}) = 2$, so 
 $e$ lies on $a$ in $M$. Thus $M\ba e$ is \thc\ having a c-minor isomorphic to $N$; a \cn. We deduce that, in $M/b$, the point $a$ is not parallel to another point, so $M/b$ is simple. 

We complete the proof by showing that $M/b$ is \thc. Suppose it has $(A,C)$ as a $2$-separation. 
 If $A$ or $C$, say $A$, contains $\{a,c\}$, then $b$ is skew to $C$, so $(A\cup b,C)$  is a $2$-separation of $M$; a \cn. Thus, we may assume that $a \in A$ and $c \in C$. Then, as $a$ is a point on the line $c$ in $M/b$, we see that $(A - a,C\cup a)$ is 2-separating in $M/b$. It is not a 2-separation otherwise we obtain a \cn\ as before. It follows that $A$ is a parallel pair of points in $M/b$, contradicting the fact that $M/b$ is simple.
\end{proof}

Next we eliminate the case when $\theta = 3$. The core of the argument in this case mimics the argument used to prove Tutte's Triangle Lemma for matroids (see, for example, \cite[Lemma 8.7.7]{oxbook}).

\begin{lemma}
\label{3five} 
$\theta \neq 3$.
\end{lemma}

\begin{proof} 
Assume that $\sqcap(\{a\},\{z\}) =  \sqcap(\{b\},\{z\}) = \sqcap(\{c\},\{z\}) = 1$. Then, by Lemma~\ref{3one}, $r(\{a,b\}) = r(\{b,c\}) = r(\{a,c\}) = 4$. First we show the following. 

\begin{sublemma}
\label{two2}
There are at least two members $y$ of $Y_1$ such that $M\ba y$ has a c-minor isomorphic to $N$.
\end{sublemma}

Assume that this fails. Since $|Y_1 - E(N)| \ge 2$, there is an element, say $a$, of $Y_1 - E(N)$  such that $M/a$ has $N$ as a c-minor. In $M/a$, we see that $b$ and $c$ are parallel lines and $\{b,c\}$ is 2-separating. Thus, by Lemma~\ref{claim1}, each of $M/a\ba b$ and $M/a\ba c$ have special $N$-minors. This \cn\ implies that \ref{two2} holds.

We now assume that both $M\ba a$ and $M\ba b$ have special $N$-minors. Clearly,  $M\ba a$ has $b$ and $c$ as 2-separating lines,  and, by Lemma~\ref{3m}, these are the only 2-separating lines in $M\ba a$. Thus $(M\ba a)^{\flat} = M\ba a \ud b \ud c$. Symmetrically, 
$(M\ba b)^{\flat} = M\ba b \ud a \ud c$. As the theorem fails, neither $(M\ba a)^{\flat}$ nor $(M\ba b)^{\flat}$ is \thc. Thus each of $M\ud c \ba a$ and 
$M\ud c \ba b$ have non-trivial 2-separations. It will be convenient to work in the $2$-polymatroid $M\ud c$, which we shall rename $M_c$. Let $(X_a,Y_a)$ and $(X_b,Y_b)$ be non-trivial 2-separations of $M_c\ba a$ and $M_c\ba b$, respectively, with $b$ in $Y_a$ and $a$ in $Y_b$. 

Now it is straightforward to check the following.

\begin{sublemma}
\label{zex}
If $Z \subseteq X_1$ and $e \in \{a,b\}$, then 
$\sqcap_{M}(Z, \{e\}) =   \sqcap_{M_c}(Z,\{e\})$.
\end{sublemma}

We deduce that 

\begin{sublemma}
\label{two2.5}
$\sqcap_{M_c}(X_1,\{a\}) = 1 = \sqcap_{M_c}(X_1,\{b\})$.
\end{sublemma}


Next we show that 
\begin{sublemma}
\label{two3}
$c \in X_a \cap X_b$.
\end{sublemma}

Suppose $c$ in $Y_a$. Since $\{c,b\}$ spans $a$ in $M_c$, it follows that $(X_a,Y_a \cup a)$ is a 2-separation of $M_c$ and hence of $M$; a \cn. We deduce that $c \in X_a$ and, by symmetry, \ref{two3} holds.

\begin{sublemma}
\label{zztop}
For $Z \subseteq X_1$, if $\sqcap_M(Z,\{a\}) = 1 = \sqcap_M(Z,\{b\})$, then $\sqcap_M(Z,\{a,b\}) = 2.$
\end{sublemma}

Assume $\sqcap_M(Z,\{a,b\}) < 2.$ Then $\sqcap_M(Z,\{a,b\}) = \sqcap_M(Z,\{a\}) = 1$. Thus 
$$r(Z) + r(\{a,b\}) - r(Z \cup \{a,b\}) = r(Z) + r(\{a\}) - r(Z \cup a),$$ 
so $r(\{a,b\}) - r(\{a\}) = r(Z \cup \{a,b\}) - r(Z \cup a)$. Hence $b$ is skew to $Z \cup a$, so $b$ is skew to $Z$; a \cn. We deduce that 
\ref{zztop} holds.

\begin{sublemma}
\label{zztop2}
For $Z \subseteq X_1$, if $\sqcap_{M_c}(Z,\{a\}) = 1 = \sqcap_{M_c}(Z,\{b\})$, then $\sqcap_{M_c}(Z,\{a,b\}) = 2.$
\end{sublemma}

By \ref{zex},  $\sqcap_{M_c}(Z,\{a\}) = \sqcap_{M}(Z,\{a\})$. Moreover, 
\begin{align*}
\sqcap_{M_c}(Z,\{a,b\}) & = r_{M_c}(Z) + r_{M_c}(\{a,b\}) - r_{M_c}( Z \cup \{a,b\})\\
& = r_M(Z) + [r_M(\{a,b\}) - 1] - [ r_{M}( Z \cup \{a,b\}) -1]\\
& = \sqcap_{M}(Z,\{a,b\}).
\end{align*}
Thus \ref{zztop2} follows immediately from \ref{zztop}.

\begin{sublemma}
\label{zztop3}
Assume $Z \subseteq X_1$ and $\sqcap_{M_c}(Z,\{a,b\}) = 2.$ Then $c \in \cl_{M_c}(Z).$
\end{sublemma}

To see this, note that 
$$r_{M_c}(Z \cup \{a,b,c\}) = r_{M_c}(Z \cup \{a,b\}) = r_{M_c}(Z) + r_{M_c}(\{a,b\}) - 2 = r_{M_c}(Z) + 1.$$ 
By submodularity,
$$r_{M_c}(E - \{a,b\}) + r_{M_c}(Z \cup \{a,b,c\}) \ge r(M_c) + r_{M_c}(Z \cup c).$$ 
Thus
$$r(M_c) - 1 + r_{M_c}(Z) + 1 \ge r(M_c) + r_{M_c}(Z \cup c).$$ 
Hence $r_{M_c}(Z) \ge r_{M_c}(Z \cup c)$ and \ref{zztop3} holds.

\begin{sublemma}
\label{two4}
Neither $a$ nor $b$ has a point on it in either $M$ or   $M_c$.
\end{sublemma}

Assume there is a point $e$ on $a$ in $M$. Then $M\ba e$ is \thc. Moreover, in $(M\ba b)^{\flat}$, we see that $e$ is parallel to $a$ so $(M\ba b)^{\flat}\ba e$, and hence $M\ba e$, has a c-minor isomorphic to $N$; a \cn. We conclude that \ref{two4} holds. 

The next  step in the proof of Lemma~\ref{3five} is to show that 

\begin{sublemma}
\label{two5}
$M_c\ba a,b$ is $2$-connected.
\end{sublemma}

Suppose  $(A,B)$ be a 1-separation of $M_c\ba a,b$ having $c$ in  $A$. Then

\begin{equation}
\label{abc}
r_{M_c}(A) + r_{M_c}(B) = r(M_c \ba a,b) = r(M_c) - 1 = r(M) - 2.
\end{equation}
Thus 
\begin{multline*}
r_{M_c}(A \cup a)  + r_{M_c}(B) - r(M_c)\\
\shoveleft{\hspace*{0.8in}= r_{M_c}(A) + r_{M_c}(\{a\}) - \sqcap_{M_c}(A,\{a\}) + r_{M_c}(B) - r(M_c)}\\
\shoveleft{\hspace*{0.8in}= [r_{M_c}(A) + r_{M_c}(B) - r(M_c) + 1] - 1 + r_{M_c}(\{a\}) - \sqcap_{M_c}(A,\{a\})}\\
\shoveleft{\hspace*{0.635in}= 0 - 1 + 2 - \sqcap_{M_c}(A,\{a\}) = 1 -  \sqcap_{M_c}(A,\{a\}).}
\end{multline*}

If $\sqcap_{M_c}(A,\{a\}) =  1$, then $(A \cup a \cup b, B)$ is a 1-separation of $M_c$ and hence of $M$; a \cn. We deduce that $\sqcap_{M_c}(A,\{a\}) =  0$ and $(A\cup a \cup b, B)$ is 2-separating in $M_c$ and hence in $M$. Thus $B$ consists of a point, say $d$, of $M$. Moreover,
$r_{M_c}(A \cup a)  =   r(M_c)$. Thus, as $\sqcap_{M_c}(A,\{a\}) =  0$, we see that 

\begin{equation}
\label{abcd}
r_{M_c}(A) = r(M_c) - 2.
\end{equation}

Still working towards proving \ref{two5}, we show next that 

\begin{sublemma}
\label{two6}
$\{b,d\}$ is a series pair of points in $(M\ba a)^{\flat}$.
\end{sublemma}

Recall that $(M\ba a)^{\flat} = M_c\ba a \ud b$. Now 

$$r_{M_c}(\{d,b\})  + r_{M_c}(A) - r(M_c\ba a) \le 3 + r(M_c) - 2 - r(M_c) = 1.$$
Thus $\{d,b\}$ is 2-separating in $M_c\ba a$. It follows that it is also 2-separating in $M_c\ba a \ud b$, that is, in $(M\ba a)^{\flat}$. 
But $d$ and $b$ are points in $(M\ba a)^{\flat}$, which is $2$-connected. We deduce by \ref{two4} that  \ref{two6} holds.

By \ref{two6},  $(M\ba a)^{\flat}/d$, and hence $M/d$, has a c-minor isomorphic to $N$. Next we show that 

\begin{sublemma}
\label{two7}
$(A- c,\{a,b,c\})$ is a $2$-separation of $M\ba d$.
\end{sublemma}
By (\ref{abcd}), $r_{M_c}(A-c) \le r(M_c) - 2 = r(M) - 3$ and \ref{two7} follows.

It follows from \ref{two7} and Lemma~\ref{newbix} that $M/d$ is \thc\ unless 
$M$ has a pair $\{e,f\}$ of points such that $e$ and $f$ are parallel in $M/d$.  
 Consider the exceptional case. Then $M$ has  $\{d,e,f\}$ as a triangle. Then $\{e,f\} \subseteq A - c$. Thus, by \ref{two7}, $((A- c) \cup d,\{a,b,c\})$ is a $2$-separation of $M$; a \cn. We conclude that \ref{two5} holds.

By \ref{two5}, we deduce that 

\begin{sublemma}
\label{har0}
$\lambda_{M_c\ba a,b}(X_a) = 1 = \lambda_{M_c\ba a}(X_a)$ and $\lambda_{M_c\ba a,b}(X_b) = 1 = \lambda_{M_c\ba b}(X_b)$.
\end{sublemma}

Since $r(M_c \ba a,b) = r(M_c\ba a) - 1$, it follows from \ref{har0}  and symmetry that 
\begin{sublemma}
\label{har1}
$r_{M_c}(Y_a - b) = r_{M_c}(Y_a) - 1$ and $r_{M_c}(Y_b - a) = r_{M_c}(Y_b) - 1$.
\end{sublemma}

It follows from this, symmetry,  and the fact that $r_M(Y_a \cup c) > r_M(Y_a)$ that 

\begin{sublemma}
\label{har2}
$r_{M}(Y_a - b) = r_{M}(Y_a) - 1$ and $r_{M}(Y_b - a) = r_{M}(Y_b) - 1$.
\end{sublemma}

By uncrossing,
\begin{align}
\label{subm}
2 & = \lambda_{M_c\ba a,b}(X_a) + \lambda_{M_c\ba a,b}(Y_b - a)  \nonumber\\
& \ge \lambda_{M_c\ba a,b}(X_a\cap (Y_b - a)) + \lambda_{M_c\ba a,b}(X_a \cup (Y_b - a)).
\end{align}

\begin{sublemma}
\label{xeyfnot}
$X_a\cap Y_b \neq \emptyset \neq  X_b\cap Y_a.$
\end{sublemma}

Suppose $X_a\cap Y_b  =  \emptyset$. Then $Y_b - a \subseteq Y_a - b$. Thus, by \ref{har1}, 
$\sqcap_{M_c}(Y_a - b,\{b\}) = 1 = \sqcap_{M_c}(Y_a - b,\{a\})$. Hence, by \ref{zztop2}, 
$\sqcap_{M_c}(Y_a - b,\{a,b\}) = 2$. Thus, by \ref{zztop3}, $c \in \cl_{M_c}(Y_a - b).$ 
It follows that $(Y_a \cup c, X_a - c)$ is 2-separating in $M_c \ba a$. Thus 
$(Y_a \cup c \cup a, X_a - c)$ is 2-separating in $M$. As $M$ is \thc, we deduce that 
$X_a$ consists of exactly two points, $c$ and $x$, say. If $r_{M_c}(\{x,c\}) = 1$, then, in $M$, we see that $x$ is a point  that lies on the line $c$. Thus $M\ba x$ is \thc. As $(M\ba a)^{\flat}$ has a c-minor isomorphic to $N$ and has $x$ and $c$ as a parallel pair of points, we deduce that $M\ba x$ has a c-minor isomorphic to $N$; a \cn. We conclude that $r_{M_c}(\{x,c\}) = 2$. Thus $\{x\}$ is 1-separating in $M_c$; a \cn.  We deduce that  $X_a\cap Y_b \neq \emptyset$ and \ref{xeyfnot} follows by symmetry.

We now choose the non-trivial 2-separation $(X_a, Y_a)$ of $M_c \ba a$ such that $|X_a|$ is a minimum subject to the condition that $b \in Y_a$. Since $X_a\cap Y_b$  and  $X_b\cap Y_a$ are both non-empty, we deduce from (\ref{subm}) and symmetry that 
$$\lambda_{M_c\ba a,b}(X_a\cap Y_b)  = 1 =  \lambda_{M_c\ba a,b}(X_b \cap Y_a).$$ 

We show next that 

\begin{sublemma}
\label{haz5} 
$\lambda_{M_c\ba a}(X_a\cap Y_b)  = 1 =  \lambda_{M_c\ba b}(X_b \cap Y_a).$
\end{sublemma}

We have $1 = r_{M_c}(X_a \cap Y_b) + r_{M_c}((Y_a - b) \cup X_b) - r(M_c \ba a,b).$ But 
$r(M_c \ba a,b) =  r(M_c \ba a) - 1$ and, by \ref{har1},  $r_{M_c}(Y_a - b) = r_{M_c}(Y_a) - 1$. Hence $r_{M_c}((Y_a - b) \cup X_b) = r_{M_c}(Y_a \cup X_b) - 1$. Thus 
\ref{haz5} follows by symmetry. 

By the choice of $X_a$ and the fact that $b$ and $c$ are the only 2-separating lines of $M\ba a$, we deduce that  $X_a \cap Y_b$ consists of a single point, say $w$.

\begin{sublemma}
\label{haz6} 
$X_a$ consists of a series pair $\{w,c\}$ in $M_c\ba a$.
\end{sublemma}

Suppose $w \notin \cl_{M_c}(X_a - w)$. Then $( X_a - w, Y_a \cup w)$ violates the choice of $(X_a,Y_a)$ unless $|X_a - w| = 1$. In the exceptional case, $\{w,c\}$ is a series pair in $M_c\ba a$.

Now suppose that $w \in \cl_{M_c}(X_a - w)$. Then $w \in \cl_{M_c}(X_b)$. Thus $(X_b \cup w, Y_b - w)$ is a 2-separation of $M_c \ba b$. But $Y_b - w$ avoids $X_a$ so we have a \cn\ to \ref{xeyfnot} when we replace $(X_b,Y_b)$ by $(X_b \cup w, Y_b - w)$ unless $Y_b = \{a,w\}$. In the exceptional case, by \ref{har1}, $r(Y_b) = 2$ and we have a \cn\ to \ref{two4}. We conclude that \ref{haz6} holds. 

Since $M_c \ba a$ has $\{w,c\}$ as a series pair. It follows that $M_c \ba a/w$ has a c-minor isomorphic to $N$. Thus so do $(M\ba a)^{\flat}/w$ and $M/w$. In $M\ba a$, we have $\{c,w\}$ and $\{b\}$ as 2-separating sets. Now $w \notin \cl_{M_c \ba a}(X_1 - w)$. 
Hence $r_M(X_1 - w) = r_M(X_1) - 1 = r(M) - 3$. As $r(Y_1) = 4$, we deduce that $(X_1 - w, Y_1)$ is a \tws\ in $M\ba w$. Thus, by Lemma~\ref{newbix}, $M/w$ is \thc\ unless $M$ has a triangle $T$ of points including $w$. In the exceptional case, $T- w \subseteq X_1 - w$, so 
$(X_1, Y_1)$ is a \tws\ of $M$. This \cn\ completes the proof of Lemma~\ref{3five}.
\end{proof}

\begin{lemma}
\label{3four} 
$\theta \neq 0$.
\end{lemma}

\begin{proof}
Assume that $\theta= 0$. Thus $\sqcap(X_1,\{y\}) = 0$ for all $y$ in $Y_1$. We may assume that $\sqcap^*(X_1,\{y\})  = 0$ for all $y$ in $Y_1$ otherwise, in $M^*$, we have $\theta \in \{1,2,3\}$. Thus, for all $y$ in $Y_1$, we have $r(Y_1 - y) = r(Y_1) = 4$. 
Then, by Lemma~\ref{earlier}, none of $M/a$, $M/b$, nor $M/c$ has a c-minor isomorphic to $N$. Hence we may assume that $a$ and $b$ are deleted to get $N$. But, in $M\ba a,b$, we see that $\{c\}$ is a component, so $c$ can be contracted to get $N$; a \cn.
\end{proof}

\begin{proof}[Proof of Lemma~\ref{Step6+}.] 
By Lemma~\ref{y13}, a minimal non-$N$-$3$-separator $Y_1$ of $M$ having exactly three elements consists of three lines. Above, we looked at the number $\theta$ of members $y$ of $Y_1$ for which $\sqcap(X_1,\{y\}) = 1$. In Lemmas~\ref{3two} and \ref{3three}, we showed that $\theta \neq 2$ and $\theta \neq 1$, while Lemmas~\ref{3five} and \ref{3four} showed that $\theta \neq 3$ and $\theta \neq 0$. There
are no remaining possibilities for $\theta$, so Lemma~\ref{Step6+} holds.
 \end{proof}

\section{A minimal non-$N$-$3$-separator with at least four elements} 
\label{fourel}

By \ref{Step6}, we may now assume that $M$ has a minimal non-$N$-$3$-separator $Y_1$ having at least four elements. As before, we write $X_1$ for $E(M) - Y_1$. Our next goal   is to prove \ref{Step7}, which we restate here for convenience.

\begin{lemma}
\label{dubya}
Let $Y_1$ be a minimal non-$N$-$3$-separating set having at least four elements. Then $Y_1$ contains a doubly labelled element.
\end{lemma}

\begin{proof} Assume that the lemma fails. For   each $e$ in $Y_1 - E(N)$, let $\nu(e)$ be equal to the unique member of  $\{\mu(e), \mu^*(e)\}$ that is defined. Choose $\ell$ in $Y_1 - E(N)$ to minimize $\nu(\ell)$.   By switching to the dual if necessary, we may suppose that $\nu(\ell) = \mu(\ell)$.
Let $(A,B)$ be a \tws\ of $M\ba \ell$ where $A$ is the $N$-side and $|B| = \mu(\ell)$. We now apply Lemma~\ref{p124}. Part (ii) of that lemma does not hold otherwise, by Lemma~\ref{bubbly}, $Y_1 - \ell$ contains a doubly labelled element. 

Assume next that (iii) of Lemma~\ref{p124} holds. Then $\lambda_{M\ba \ell}(Y_1 - \ell) = 2$ and 
$\lambda_{M\ba \ell}(A \cap(Y_1 - \ell)) = 1= \lambda_{M\ba \ell}(B \cap(Y_1 - \ell))$, while 
$\lambda_{M\ba \ell}(A \cap X_1) = 2= \lambda_{M\ba \ell}(B \cap X_1)$. Then using the partitions $(A\cap (Y_1 - \ell), A \cap X_1,B)$ 
and $(B\cap (Y_1 - \ell), B \cap X_1,A)$ as $(A,B,C)$ in Lemma~\ref{general3}, we deduce that 
$\sqcap(A \cap (Y_1 - \ell), A\cap X_1) = 1$ and 
$\sqcap(B \cap (Y_1 - \ell), B\cap X_1) = 1$.

Now $M\ba \ell$ is the 2-sum of $2$-polymatroids $M_A$ and $M_B$ having ground sets $A \cup q$ and $B \cup q$, respectively. Since $M\ba \ell$ is $2$-connected, it follows by Proposition~\ref{connconn}, that each of $M_A$ and $M_B$ is $2$-connected. 
Now $\lambda_{M\ba \ell}(B \cap (Y_1 - \ell))=  \sqcap_{M\ba \ell}(B \cap (Y_1 - \ell), (B \cap X_1) \cup A) = 1$ and $\sqcap_{M_B}(B \cap (Y_1 - \ell), B \cap X_1) = 1$. Noting that $M\ba \ell = P(M_A,M_B)\ba q$, we see that, in $P(M_A,M_B)$, we have $\sqcap(B \cap (Y_1 - \ell), (B \cap X_1) \cup A\cup q) = 1$. Hence 
 $\sqcap_{M_B}(B \cap (Y_1 - \ell), (B \cap X_1) \cup q) = 1$. Thus $M_B$ is the 2-sum of two $2$-connected $2$-polymatroids $M_{B,Y}$ and $M_{B,X}$ having ground sets $(B \cap (Y_1 - \ell)) \cup s$ and  $(B \cap X_1) \cup q \cup s$. Note that $M_B = P(M_{B,X},M_{B,Y})\ba s$. 
 Let $M'_B = P(M_{B,X},M_{B,Y})$ and consider $P(M_A,M'_B)$ noting that deleting $q$ and $s$ from this $2$-polymatroid gives $M\ba \ell$.
 
 By Lemma~\ref{oswrules}(ii), 
 $$\sqcap(A,B) + \sqcap(B \cap X_1, B \cap (Y_1 - \ell)) = \sqcap(A \cup (B \cap X_1),B\cap (Y_1 - \ell)) + \sqcap(A, B \cap X_1).$$
 Since the first three terms in this equation equal one, 
 \begin{equation}
 \label{labx}
 \sqcap(A, B \cap X_1) = 1.
 \end{equation}
 We deduce, by Lemma~\ref{claim1}(i) that if $y \in B \cap (Y_1 - \ell)$, then $M\ba \ell\ba y$ has a special $N$-minor.
 
Now $M_{B,X}$ has $q$ and $s$ as points. We show next that 
 
\begin{sublemma}
\label{notzero}  $\lambda_{M_{B,X}/ s}(\{q\}) =  0$.
\end{sublemma}

Assume that  $\lambda_{M_{B,X}/ s}(\{q\}) \neq  0$. When we contract $s$ in $M'_B$, the set $B \cap (Y_1 - \ell)$ becomes 1-separating. Moreover, in $M'_B/(B \cap (Y_1 - \ell))$, the element $s$ is a loop, so $M'_B\ba s/(B \cap (Y_1 - \ell)) = M'_B/s/(B \cap (Y_1 - \ell))$. It follows that 
  $\sqcap_{M\ba \ell/(B \cap (Y_1 - \ell)}(A, B\cap X_1) = 1$. Hence, by Lemma~\ref{claim1}(ii), if $y \in B \cap (Y_1 - \ell)$, then $M\ba \ell/ y$ has a special $N$-minor.
Thus  each $y$ in $B \cap (Y_1 - \ell)$ is doubly labelled. This \cn\ completes the proof of \ref{notzero}. 

By \ref{notzero},  $\{q,s\}$ is a parallel pair of points in $M_{B,X}$. From considering $P(M_A,M'_B)$, we deduce that 
$\lambda_{M\ba \ell}(B \cap X_1) = 1$. 
This \cn\ implies that (iii) of Lemma~\ref{p124} does not hold. 

It remains to consider when (i) of Lemma~\ref{p124} holds. We now apply Lemma~\ref{p63rev} to get, because of the choice of $\ell$, that $\sqcap(\{y\},X_1)  \neq 1$ for some $y$ in $Y_1 - \ell$. Then, by Lemma~\ref{dichotomy}, $\sqcap(Y_1 - y, X_1) = 0$ for all $y$ in $Y_1 - \ell$. 
Thus, by Lemma~\ref{prep65rev} 
and the choice of $\ell$, (iii)(b) rather than (iii)(a)  of that lemma holds. Then 
$(X_1 - z, (Y_1 - \ell) \cup z$ is a 2-separation of $M\ba \ell$ having $X_1 - z$ as the $N$-side. Since $z$ is a point, we have a \cn\ to Lemma~\ref{key}.      
\end{proof}

The doubly labelled element found in the last lemma will be crucial in completing the proof of Theorem~\ref{modc}. We shall need another preliminary lemma.

\begin{lemma} 
\label{series} 
Let $\ell$ be a doubly labelled element of $M$. Then $M\ba \ell$ does not have a series pair of points $\{a,b\}$ such that $r(\{a,b,\ell\}) = 3$.
\end{lemma}

\begin{proof} Assume that $M\ba \ell$ does have such a series pair $\{a,b\}$. By Lemma~\ref{Step0}, $\ell$ is a line. Thus $M/\ell$ has $\{a,b\}$ as a parallel pair of points or has $a$ or $b$ as a loop. In each case, $M$ has $a$ or $b$ as a doubly labelled point; a \cn.
\end{proof}

We will now take $\ell$ to be a doubly labelled element of $Y_1$, a minimal non-$N$-$3$-separating set having at least four elements.

\begin{lemma} 
\label{p125} There is a $2$-separating set $Q$ in $M\ba \ell$ such that $Q \subseteq Y_1 - \ell$ and $|Q| \ge 2$ and contains no points. 
\end{lemma}

\begin{proof} Suppose $\ell \not\in \cl(Y_1 - \ell)$. Then $(X_1, Y_1 - \ell)$ is a \tws\ of $M\ba \ell$ and $r(Y_1) = r(Y_1 - \ell) + 1$. Then, by Lemma~\ref{series}, $Y_1 - \ell$ does not consist of a series pair of points. Hence, by Lemma~\ref{key}, $Y_1 - \ell$ contains no points so the result holds by taking $Q  =Y_1 - \ell$. 

We may now assume that  $\ell \in \cl(Y_1 - \ell)$. Let $(A,B)$ be a \tws\ of $M\ba \ell$ where $A$ is the $N$-side and $|B| = \mu(\ell)$. Since $|B| \ge 3$, it follows by Lemma~\ref{key} that $B$ contains no points. If $B \subseteq Y_1 - \ell$, then the lemma holds by taking $Q = B$. Thus we may assume that $B \cap X_1 \neq \emptyset$.

Since $X_1$ and $A$ are the $N$-sides of their respective separations and $|E(N)| \ge 4$, we see that $|A \cap X_1| \ge 2$. If $A \subseteq X_1$, then $B \supseteq Y_1 - \ell$, so $(A,B \cup \ell)$ is a \tws\ of $M$; a \cn. Likewise, if $B \subseteq X_1$, then $(A \cup \ell,B)$ is a \tws\ of $M$; a \cn. We conclude that $(A,B)$ crosses $(X_1,Y_1 - \ell)$.

Since $|A \cap X_1| \ge 2$ and $\ell \in \cl(Y_1 - \ell)$, it follows that $\lambda_{M\ba \ell}(A \cap X_1) \ge 2$ otherwise 
$(A \cap X_1, B \cup Y_1)$ is a \tws\ of $M$; a \cn. Then, by uncrossing, we deduce that $\lambda_{M\ba \ell}(B \cap Y_1) \le 1$. Since  $B$  contains no   points, the lemma holds with $Q = B \cap Y_1$ unless this set contains a single line.

Consider the exceptional case. As $|B| \ge 3$, we deduce that $|B \cap X_1| \ge 2$. Now $\lambda_{M\ba \ell}(B \cap X_1) \ge 2$ otherwise, as $\ell \in \cl(Y_1 - \ell)$, we obtain the \cn\ that $(B\cap X_1,A \cup Y_1)$ is a \tws\ of $M$. Hence, by uncrossing, $\lambda_{M\ba \ell}(A \cap Y_1) = 1$ and, as $|Y_1| \ge 4$, it follows using Lemma~\ref{key} that the lemma holds by taking $Q = A \cap Y_1$. 
\end{proof}

\begin{lemma} 
\label{ab} The $2$-polymatroid $M\ba \ell$ has a $2$-separation $(D_1,D_2)$ where $D_2$ has at least two elements, is contained in $Y_1 - \ell$, and contains no points. Moreover, either 
\begin{itemize}
\item[(i)] $D_2 \cup \ell = Y_1$; and $\sqcap(D_1,\{\ell\}) = 0$ and $\sqcap(D_2,\{\ell\}) = 1$; or 
\item[(ii)] $Y_1 - \ell - D_2 \neq \emptyset$ and $\sqcap(D_1,\{\ell\}) = 0 = \sqcap(D_2,\{\ell\})$.
\end{itemize}
 \end{lemma}
 
\begin{proof} Let $D_2$ be the set $Q$ found in Lemma~\ref{p125} and let $D_1 = E(M\ba \ell) - D_2$. 
Now $(D_1,D_2)$ is a \tws\ of $M\ba \ell$. Thus, there are the following four possibilities.
\begin{itemize}
\item[(I)]   $\sqcap(D_1,\{\ell\}) = 1 = \sqcap(D_2,\{\ell\})$;   
\item[(II)] $\sqcap(D_1,\{\ell\}) = 1$ and  $\sqcap(D_2,\{\ell\}) = 0$; 
\item[(III)] $\sqcap(D_1,\{\ell\}) = 0$ and  $\sqcap(D_2,\{\ell\}) = 1$; and  
\item[(IV)] $\sqcap(D_1,\{\ell\}) = 0 = \sqcap(D_2,\{\ell\})$.
\end{itemize}

\begin{sublemma}
\label{elim12}
Neither (I) nor (II) holds.
\end{sublemma}

Suppose (I) or (II) holds. Then $\lambda_M(D_1 \cup \ell) = 2$, so $\lambda_M(D_2) =2$ and $|D_2| \ge 2$. Since $D_2$ contains no points and  $D_2 \subseteq Y_1 - \ell$, we get a contradiction to the minimality of $Y_1$. Thus \ref{elim12} holds.

\begin{sublemma}
\label{case3}
If (III) holds, then $D_1 \cap Y_1 = \emptyset$. 
\end{sublemma}

As $\lambda_M(D_2 \cup \ell) = 2$, we must have that $D_1 \cap Y_1 = \emptyset$ otherwise $D_2 \cup \ell$ violates the minimality of $Y_1$. Thus \ref{case3} holds.

\begin{sublemma}
\label{case4}
If (IV) holds, then $D_1 \cap Y_1 \neq \emptyset$. 
\end{sublemma}

Suppose $D_1 \cap Y_1 = \emptyset$. Then $D_1 = X_1$ and $D_2 = Y_1 - \ell$. Thus $\lambda_M(X_1) > 2$ as $\sqcap(D_2,\{\ell\}) = 0$. This \cn\ establishes that \ref{case4} holds and thereby completes the proof of the lemma.
\end{proof}

\begin{lemma} 
\label{abdual} The $2$-polymatroid $M/ \ell$ has a $2$-separation $(C_1,C_2)$ where $C_2$ contains at least two elements and  is contained in $Y_1 - \ell$, and contains no points of $M$. Moreover, either 
\begin{itemize}
\item[(i)] $C_2 \cup \ell = Y_1$; and $\sqcap(C_1,\{\ell\}) = 1$ and $\sqcap(C_2,\{\ell\}) = 2$; or 
\item[(ii)] $Y_1 - \ell - C_2 \neq \emptyset$ and $\sqcap(C_1,\{\ell\}) = 2 = \sqcap(C_2,\{\ell\})$.
\end{itemize}
 \end{lemma}
 
\begin{proof} We apply the preceding lemma to $M^*\ba \ell$ recalling that $(M^*\ba \ell)^{\flat} = (M/\ell)^*$ and that the connectivity functions of 
$M^*\ba \ell$  and $M/\ell$ are equal. Thus $M/ \ell$ does indeed have a $2$-separation $(C_1,C_2)$ where $C_2$ contains at least two elements, 
   is contained in $Y_1 - \ell$, and contains no points of $M^*$. Thus $C_2$ contains no points of $M$. Since $r(E - \ell) = r(M)$, one easily checks that 
   $\sqcap^*(C_i, \{\ell\}) + \sqcap(C_j,\{\ell\}) = 2$ where $\{i,j\} = \{1,2\}$. The lemma now follows from the preceding one. 
  \end{proof}
   

\begin{lemma} 
\label{bb} The $2$-polymatroids $M\ba \ell$ and $M/ \ell$ have $2$-separations $(D_1,D_2)$ and $(C_1,C_2)$, respectively, such that each of $D_2$ and $C_2$ contains at least two elements, both $D_2$ and $C_2$ are contained in $Y_1 - \ell$, and neither $D_2$ nor $C_2$ contains any points of $M$. Moreover, $Y_1 - \ell - D_2 \neq \emptyset \neq Y_1 - \ell - C_2$ and $\sqcap(D_1,\{\ell\}) = 0 = \sqcap(D_2,\{\ell\})$ while 
$\sqcap(C_1,\{\ell\}) = 2 = \sqcap(C_2,\{\ell\})$. 
\end{lemma}

\begin{proof} 
Assume that (i) of Lemma~\ref{ab} holds. Then, as $D_2 = Y_1 - \ell$, we see that $\sqcap(Y_1 - \ell,\{\ell\}) = 1$. Thus (i) of Lemma~\ref{abdual} cannot hold. Moreover, if (ii) of Lemma~\ref{abdual} holds, then $\sqcap(C_2, \{\ell\}) = 2$. This is a \cn\ as $\sqcap(D_2,\{\ell\}) = 1$ and $C_2 \subseteq D_2$. We conclude that (ii) of Lemma~\ref{ab} holds. If (i) of Lemma~\ref{abdual} holds, then 
$\sqcap(X_1, \{\ell\}) = 1$. But $X_1 \subseteq D_1$ and $\sqcap(D_1, \{\ell\}) = 0$; a \cn.
\end{proof}

We now use the $2$-separations $(D_1,D_2)$ and $(C_1,C_2)$ of $M\ba \ell$ and $M/ \ell$, respectively, found in the last lemma.

\begin{lemma} 
\label{linemup} The partitions $(D_1,D_2)$ and $(C_1,C_2)$ have the following properties.
\begin{itemize}
\item[(i)] $\lambda_{M\ba \ell}(C_1) = 3 = \lambda_{M\ba \ell}(C_2)$; 
\item[(ii)] $(D_1,D_2)$ and $(C_1,C_2)$ cross;  
\item[(iii)] each of $D_1 \cap C_2, D_2 \cap C_2$, and $D_2 \cap C_1$ consists of a single line, and $C_2 \cup D_2 = Y_1 - \ell$; and 
\item[(iv)] $\lambda_{M\ba \ell}(D_1 \cap C_1) = \lambda_{M}(D_1 \cap C_1) = 2.$
\end{itemize}
\end{lemma}

\begin{proof} We have $\sqcap(D_1,\{\ell\}) = 0 = \sqcap(D_2,\{\ell\})$  and  
$\sqcap(C_1,\{\ell\}) = 2 = \sqcap(C_2,\{\ell\})$. Thus neither $C_1$ nor $C_2$ is contained in $D_1$ or $D_2$, so (ii) holds. Moreover, as $r(C_1 \cup \ell) + r(C_2 \cup \ell) - r(M) = 3$,  we see that (i) holds. To prove (iii) and (iv), we use an uncrossing argument. We have, for each $i$ in $\{1,2\}$,  
\begin{align*}
1 + 3 & = \lambda_{M\ba \ell}(D_2) + \lambda_{M\ba \ell}(C_i)\\
& \ge\lambda_{M\ba \ell}(D_2\cap C_i) + \lambda_{M\ba \ell}(D_2 \cup C_i).
\end{align*}
Since $D_2 \cap C_i \neq \emptyset$ and contains no points and $\ell \in \cl(C_j)$ where $j \neq i$, we deduce that 
$\lambda_{M\ba \ell}(D_2 \cap C_i) = \lambda_{M}(D_2 \cap C_i) \ge 2$. Thus 
$\lambda_{M\ba \ell}(D_2 \cup C_i) \le 2.$
Hence, as $\ell \in \cl(C_i)$, we see that 
\begin{equation}
\label{addno}
2 \ge \lambda_{M\ba \ell}(D_2 \cup C_i) = \lambda_M(D_2 \cup C_i \cup \ell).
\end{equation}
But $D_2 \cup C_2 \subseteq Y_1 - \ell$, so, by the definition of $Y_1$, we deduce that
\begin{equation}
\label{addno2}
D_2 \cup C_2 = Y_1 - \ell \text{~and~} D_1 \cap C_1 = X_1.
\end{equation}
Moreover, as $\lambda_{M\ba \ell}(D_2 \cup C_i) = 2$, we see that $\lambda_{M\ba \ell}(D_2 \cap C_i) = 2$. Hence 
\begin{equation}
\label{addno4}
\lambda_M(D_2 \cap C_i) = 2.
\end{equation}

Since $D_1 \cap C_1 = X_1$ and $D_1 \cap C_2$ is non-empty containing no points, it follows from (\ref{addno}) that 
\begin{equation}
\label{addno3}
2  =  \lambda_{M\ba \ell}(D_2 \cup C_i) = \lambda_M(D_2 \cup C_i \cup \ell) = \lambda_{M\ba \ell}(D_1 \cap C_j) = \lambda_M(D_1 \cap C_j).
\end{equation}
Thus (iv) holds. By that and (\ref{addno4}), it follows, using the minimality of $Y_1$, that each of $D_1 \cap C_2, D_2 \cap C_2$, and $D_2 \cap C_1$ consist of a single line in $M$. Hence (iii) holds.  
\end{proof}


For each $(i,j)$ in $\{(1,2),(2,2),(2,1)\}$, let $C_i \cap D_j = \{\ell_{ij}\}$.

\begin{lemma}
\label{ranks}
The following hold.
\begin{itemize}
\item[(i)] $r(D_2) = 3$ so $r(D_1) = r(M) - 2$; 
\item[(ii)] $r(D_1) = r(X_1) + 1$;   
\item[(iii)] $r(Y_1) = 5$; and 
\item[(iv)]  $r(C_2) = 4$.
\end{itemize}
\end{lemma}

\begin{proof} 
Now $D_2$ consists of two lines, $\ell_{12}$ and $\ell_{22}$. Suppose first that $r(D_2) = 2$. Then both $M\ba \ell_{12}$ and  $M\ba \ell_{22}$ are \thc. Without loss of generality,   $M/ \ell_{12}$ has $N$ as a c-minor. But $M/ \ell_{12}$ has $\ell_{22}$ as a loop, so $M\ba \ell_{22}$ is \thc\ having a c-minor isomorphic to $N$. Thus $r(D_2) \ge 3$.

Now suppose that $r(D_2) = 4$. Then $r(D_1) = r(M) - 3$. Clearly $r(D_1 \cup \ell_{12}) \le   r(M) - 1$. Now $D_1 \cup \ell_{12} \supseteq C_2$ so 
$r(D_1 \cup \ell_{12} \cup \ell) \le   r(M) - 1$. Hence $\{\ell_{22}\}$ is 2-separating in $M$; a \cn. Hence (i) holds.

Since $C_2 \cup D_2 = Y_1 - \ell$, we see that $D_1 = X_1 \cup \ell_{21}$. Suppose $r(D_1) = r(X_1)$.  As $X_1 \subseteq C_1$, we deduce that   
 $\ell_{21} \in \cl(C_1)$. But $\ell \in \cl(C_1)$. Hence 
\begin{align*} 
r(M) + 3 & = r(C_1) + r(C_2)\\
& =r(C_1 \cup \ell) + r(C_2 \cup \ell)\\
& = r(C_1 \cup \ell\cup \ell_{21}) + r(C_2 \cup \ell)\\
& \ge r(C_1 \cup C_2 \cup \ell) + r(\{\ell,\ell_{21}\}).
\end{align*}
Thus $r(\{\ell,\ell_{21}\}) \le 3$, so $\sqcap(D_1,\{\ell\}) \ge 1$; a \cn. Hence $r(D_1) \ge r(X_1) + 1$.

Suppose $r(D_1) = r(X_1) + 2.$ Then 
$$r(M) + 1 = r(D_1) + r(D_2) = r(X_1) + 2 + 3,$$ 
so $r(X_1) = r(M) - 4$. Thus $r(Y_1) = 6$. Now $r(C_2) = r(C_2 \cup \ell)$. Thus $6 = r(Y_1) = r(Y_1 - \ell)$. Since $Y_1 - \ell$ consists of three lines, two of which are in $D_2$, we deduce that $r(D_2) = 4$; a \cn\ to (i). We conclude that $r(D_1) = r(X_1) + 1$, that is, (ii) holds.

Finally, as $r(D_2) = 3$, we see that $r(M) = r(D_1) + r(D_2) - 1 = [r(X_1) +1] + 3 - 1$. But $r(M) = r(X_1) + r(Y_1) - 2$. Thus $r(Y_1) = 5$, so (iii) holds. Moreover, $r(Y_1 - \ell) = 5$, that is, $r(C_2 \cup D_2) = 5.$ Now $C_2$ consists of two lines so $r(C_2) \le 4$. Thus
\begin{align*}
4 + 3 & \ge r(C_2) + 3\\
& = r(C_2) + r(D_2)\\
& \ge r(C_2 \cup D_2) + r(C_2 \cap D_2)\\
& = 5 + r(\{\ell_{22}\})\\
& = 5 + 2.
\end{align*}
We deduce that $r(C_2) = 4$ so (iv) holds.
\end{proof}

By proving the following lemma, we will establish the final \cn\ that completes the proof of Theorem~\ref{modc}.

\begin{lemma} 
\label{final} The $2$-polymatroid $M/ \ell_{22}$   is $3$-connected having a c-minor isomorphic to $N$. 
\end{lemma}

\begin{proof}
First we show the following.

\begin{sublemma}
\label{finalfirst}
$\sqcap(D_1,\{\ell_{i2}\}) = 0$ for each $i$ in $\{1,2\}$.
\end{sublemma}

Suppose $\sqcap(D_1,\{\ell_{i2}\}) \ge 1$. Then $r(D_1 \cup \ell_{i2}) \le r(D_1) + 1$. But $\ell \in \cl(C_i) \subseteq \cl(D_1 \cup \ell_{i2})$, so 
 $r(D_1 \cup \ell \cup \ell_{i2}) \le r(D_1) + 1$. Also $r(\{\ell_{j2}\}) = 2$ where $\{i,j\} = \{1,2\}$. Thus
 \begin{align*}
 r(D_1 \cup \ell \cup \ell_{i2})+ r(\{\ell_{j2}\}) & \le r(D_1) + 1 + 2\\
 & = r(M) - 2 + 1 +2\\
 & = r(M) + 1.
 \end{align*}
Hence $\{\ell_{j2}\}$ is 2-separating in $M$; a \cn. Hence \ref{finalfirst} holds.
 
 Now $M\ba \ell$ has a c-minor isomorphic to $N$ and $\sqcap(D_1,D_2) = 1$. 
As $\sqcap(D_1,\{\ell_{i2}\}) = 0$, Lemma~\ref{obs1} implies that  $\sqcap_{M/\ell_{i2}}(D_1,D_2 - \ell_{i2}) = 1$ for each $i$ in $\{1,2\}$. Thus, by Lemma~\ref{claim1}(ii), 
 \begin{sublemma}
\label{finalsecond}
$M\ba \ell/\ell_{i2}$  has a c-minor isomorphic to $N$ for each $i$ in $\{1,2\}$.
\end{sublemma}

It remains to show that $M/\ell_{22}$ is \thc. This matroid is certainly $2$-connected. Next we show that
 \begin{sublemma}
\label{end1}
$\ell$ and $\ell_{21}$ are parallel lines in $M/\ell_{22}$.
\end{sublemma}

To see this, note that, by Lemma~\ref{ranks}(iv), 
$$r(C_2 \cup \ell) = r(C_2) = r(\{\ell_{21},\ell_{22}\}.$$
Also, for each $i$ in $\{1,2\}$, we have $\sqcap(\{\ell\}, \{\ell_{2i}\}) \le \sqcap(\{\ell\},D_i) = 0$, so \ref{end1} holds.

Now take a fixed c-minor of $M\ba \ell/\ell_{22}$ isomorphic to $N$; call it $N_1$. Let $(A' \cup \ell, B')$ be a \tws\ of $M/\ell_{22}$ in which the non-$N_1$-side has maximum size and $\ell \not\in A'$. By Lemma~\ref{dualmu}, both $A' \cup \ell$ and $B'$ have at least three elements.

 \begin{sublemma}
\label{ell12a'}
 $\ell_{21} \in A'$.
\end{sublemma}

To see this, note that, since $\ell$ and $\ell_{21}$ are parallel lines in $M/\ell_{22}$, if  $\ell_{21} \in B'$, then $\ell \in \cl_{M/\ell_{22}}(B')$, so 
$\sqcap_{M/\ell_{22}}(A' \cup \ell, B') \ge 2$; a \cn.

 \begin{sublemma}
\label{ranksagain}
 $r_M(D_1 \cap C_1) = r(M) -3= r_{M/\ell_{22}}(D_1 \cap C_1)$.
\end{sublemma}

By Lemma~\ref{linemup}(iii), $D_1 \cap C_1 = X_1$. By Lemma~\ref{ranks}(i) and (ii), $r(D_1) = r(M) - 2$ and $r(D_1) = r(D_1 \cap C_1) + 1$, so 
$r_M(D_1 \cap C_1) = r(M) -3$. By \ref{finalfirst},  $\sqcap(D_1 \cap C_1, \{\ell_{22}\}) = 0$, so $r_M(D_1 \cap C_1) =  r_{M/\ell_{22}}(D_1 \cap C_1)$.

 \begin{sublemma}
\label{ranksagain2}
 $r_{M/\ell_{22}}((D_1 \cap C_1)\cup \ell_{12})= r(M) -2$.
\end{sublemma}

To see this, observe that 
\begin{align*}
r_{M/\ell_{22}}((D_1 \cap C_1)\cup \ell_{12}) & = r((D_1 \cap C_1)\cup \ell_{12}\cup \ell_{22}) - 2\\
& = r((D_1 \cap C_1)\cup \ell_{12}\cup \ell \cup \ell_{22}) - 2 \text{~~as $\ell \in \cl(C_1)$;}\\
& = r(M\ba \ell_{21}) - 2\\
&= r(M) - 2.
\end{align*}

By combining \ref{ranksagain} and \ref{ranksagain2}, we deduce that 

 \begin{sublemma}
\label{ell12not}
 $r_{M/\ell_{22}}(D_1 \cap C_1) = r_{M/\ell_{22}}((D_1 \cap C_1)\cup \ell_{12}) - 1$.
\end{sublemma}

Next we show that 
 \begin{sublemma}
\label{onetime}
 $\lambda_{M/\ell_{22}}(\{\ell_{12}, \ell_{21},\ell\}) = 1$ or $\lambda_{M/\ell_{22}}(\{\ell_{21},\ell\}) = 1$.
\end{sublemma}

Recall that $X_1 = D_1 \cap C_1$ and $Y_1  = \{\ell,\ell_{12},\ell_{21},\ell_{22}\}$. By
uncrossing, we have 
\begin{align*}
1 + 2 & = \lambda_{M/\ell_{22}}(B') + \lambda_{M/\ell_{22}}(X_1)\\
& \ge \lambda_{M/\ell_{22}}(B' \cup X_1) + \lambda_{M/\ell_{22}}(B' \cap X_1).
\end{align*}

As $|B'| \ge 3$, it follows by \ref{ell12a'} that   $|B'\cap X_1| \ge 2$. Suppose $\lambda_{M/\ell_{22}}(B' \cap X_1) = 1$. Now $B'\cap X_1 \subseteq D_1$, so $\sqcap(B' \cap X_1,\{\ell_{22}\}) \le \sqcap(D_1,\{\ell_{22}\}) \le 0$. Thus, by Lemma~\ref{obs1}, $ \lambda_{M}(B' \cap X_1) = 1$. This \cn\   implies that $\lambda_{M/\ell_{22}}(B' \cap X_1) = 2$, so $1 = \lambda_{M/\ell_{22}}(B' \cup X_1)$. Now, by \ref{ell12a'}, $\ell_{21} \in A'$, so 
$E- \ell - (B' \cup X_1)$ is $\{\ell_{12}, \ell_{21},\ell\}$ or $\{\ell_{21},\ell\}$. Thus  \ref{onetime} holds.  

Suppose $\lambda_{M/\ell_{22}}(\{\ell_{12}, \ell_{21},\ell\}) = 1$. Then, as $\ell_{22}$ is skew to $X_1$ in $M$, we deduce that $\lambda_{M}(\{\ell_{12}, \ell_{21},\ell, \ell_{22}\}) = 1$, that is, $\lambda_{M}(Y_1) = 1$; a \cn. We conclude that $\lambda_{M/\ell_{22}}(\{\ell_{21},\ell\}) = 1$.


By Lemma~\ref{obs1}, as $\lambda_{M}(D_1 \cap C_1) = 2$ and $\sqcap(D_1 \cap C_1, \{\ell_{22}\}) = 0$, we see that 
$\lambda_{M/\ell_{22}}(D_1 \cap C_1) = 2$. Thus, by \ref{ell12not} and Lemma~\ref{ranks}, 
\begin{align*}
2 & = r_{M/\ell_{22}}(D_1 \cap C_1) + r_{M/\ell_{22}}(\{\ell_{12},\ell_{21},\ell\}) - r(M/\ell_{22})\\
& = [r_{M/\ell_{22}}(D_1 \cap C_1) + 1] + [r_{M/\ell_{22}}(\{\ell_{12},\ell_{21},\ell\}) - 1] - r(M/\ell_{22})\\
& = r_{M/\ell_{22}}((D_1 \cap C_1) \cup \ell_{12}) + [r(Y_1) - 2 - 1] - r(M/\ell_{22})\\
& = r_{M/\ell_{22}}((D_1 \cap C_1) \cup \ell_{12}) + [r(C_2) - 2] - r(M/\ell_{22})\\
& = r_{M/\ell_{22}}((D_1 \cap C_1) \cup \ell_{12}) + r_{M/\ell_{22}}(\{\ell_{21},\ell\}) - r(M/\ell_{22})\\
& = \lambda_{M/\ell_{22}}(\{\ell_{21},\ell\}).
\end{align*}
This \cn\ to \ref{onetime} completes the proof of the lemma and thereby finishes the proof of Theorem~\ref{modc}.
\end{proof}

\end{document}